\documentclass{amsart}

\usepackage{amssymb}
\usepackage{cite}
\usepackage{esint}
\usepackage{enumerate}
\usepackage{latexsym}
\usepackage{qsymbols}
\usepackage{undertilde}
\usepackage{url}
\usepackage{verbatim}

\allowdisplaybreaks

\newtheorem{theorem}{Theorem}[section]
\newtheorem{lemma}[theorem]{Lemma}
\newtheorem{proposition}[theorem]{Proposition}
\newtheorem{corollary}[theorem]{Corollary}

\theoremstyle{definition}
\newtheorem{definition}[theorem]{Definition}

\theoremstyle{remark}
\newtheorem{remark}[theorem]{Remark}

\numberwithin{equation}{section}

\newcommand {\A}{\mathcal{A}}
\newcommand {\bmo}{\mathrm{bmo}}
\newcommand {\BMO}{\mathrm{BMO}}
\newcommand {\C}{\mathbb C}
\newcommand {\Da}{\mathcal{J}}
\newcommand {\dom}{\mathrm{dom}}
\newcommand {\F}{\mathcal{F}}
\newcommand {\hchi}{\hat{\chi}}
\newcommand {\Hp}{\mathcal{H}^{p}_{FIO}(\Rn)}
\newcommand {\Hps}{\mathcal{H}^{p}_{FIO}}
\newcommand {\HT}{\mathcal{H}}
\newcommand {\ind}{\mathbf{1}}
\newcommand {\La}{\mathcal{L}}
\newcommand {\lb}{\langle}
\newcommand {\rb}{\rangle}
\newcommand {\loc}{\mathrm{loc}}
\newcommand {\N}{\mathbb N}
\newcommand {\ph}{\varphi}
\newcommand {\R}{\mathbb R}
\newcommand {\Ra}{\mathcal{R}}
\newcommand {\ran}{\mathrm{ran}}
\newcommand {\Rn}{\mathbb{R}^{n}}
\newcommand {\Sp}{S^{*}(\Rn)}
\newcommand {\Spp}{S^{*}_{+}(\Rn)}
\newcommand {\supp}{\mathrm{supp}}
\newcommand {\Sw}{\mathcal{S}}
\newcommand {\Tp}{T^{*}(\Rn)}
\newcommand {\ud}{\mathrm{d}}
\newcommand {\veps}{\varepsilon}
\newcommand {\w}{\omega}
\newcommand {\wh}{\widehat}
\newcommand {\wt}{\widetilde}
\newcommand {\Z}{\mathbb Z}

\DeclareFontFamily{U}{mathx}{\hyphenchar\font45}
\DeclareFontShape{U}{mathx}{m}{n}{
      <5> <6> <7> <8> <9> <10>
      <10.95> <12> <14.4> <17.28> <20.74> <24.88>
      mathx10
      }{}
\DeclareSymbolFont{mathx}{U}{mathx}{m}{n}
\DeclareFontSubstitution{U}{mathx}{m}{n}
\DeclareMathAccent{\widecheck}{0}{mathx}{"71}

\begin{document}

\title[Off-singularity bounds and Hardy spaces]{Off-singularity bounds and Hardy spaces for Fourier integral operators}

\author{Andrew Hassell}
\address{Mathematical Sciences Institute\\ Australian National University\\Acton ACT 2601\\Australia}
\email{Andrew.Hassell@anu.edu.au}

\author{Pierre Portal}
\address{Mathematical Sciences Institute\\ Australian National University\\Acton ACT 2601\\Australia}
\email{Pierre.Portal@anu.edu.au}

\author{Jan Rozendaal}
\address{Mathematical Sciences Institute\\ Australian National University\\Acton ACT 2601\\Australia\\and Institute of Mathematics, Polish Academy of Sciences\\
ul.~\'{S}niadeckich 8\\
00-656 Warsaw\\
Poland}
\email{janrozendaalmath@gmail.com}

\subjclass[2010]{Primary 42B35. Secondary 42B30, 35S30, 58J40}

\thanks{This research was supported by grant DP160100941 of the Australian Research Council.}

\begin{abstract}
We define a scale of Hardy spaces $\Hp$, $p\in[1,\infty]$, that are invariant under suitable Fourier integral operators of order zero. This builds on work by Smith for $p=1$ \cite{Smith98a}. We also introduce a notion of off-singularity decay for kernels on the cosphere bundle of $\Rn$, and we combine this with wave packet transforms and tent spaces over the cosphere bundle to develop a full Hardy space theory for oscillatory integral operators. In the process we extend the known results about $L^{p}$-boundedness of Fourier integral operators, from local boundedness to global boundedness for a larger class of symbols.
\end{abstract}

\maketitle

\section{Introduction}

\subsection{Overview}

Hardy spaces have long played a role in the analysis of elliptic and parabolic equations. The purpose of this article is to develop a Hardy space theory suitable for the analysis of \emph{hyperbolic} equations. 

In recent years, harmonic analysis has developed beyond Calder\'on-Zygmund theory to treat parabolic and elliptic equations with rough coefficients. In this development, so-called adapted Hardy spaces take the place of the standard $L^{p}$-spaces. For instance, given a uniformly elliptic operator of the form $L=\text{div} A \nabla$ with $A \in L^{\infty}(\R^{n};\mathcal{L}(\C^{n}))$, the solution operators $(e^{tL})_{t \geq 0}$ to the equation $\partial_{t}u=Lu$ are bounded on $L^{p}(\R^{n})$ for a range of exponents $p_{-}(L)<p<p_{+}(L)$, where in general, the lower threshold $p_-$ is strictly larger than $1$, and the upper threshold $p_+$ is strictly less than infinity  
(see the memoir \cite{Auscher07}). 
On the other hand, $(e^{tL})_{t\geq0}$ is bounded on an appropriate Hardy space $H^{p}_{L}$ for all $1\leq p\leq \infty$, as is shown by extrapolating from $H^{2}_{L} = L^{2}(\Rn)$ in a manner similar to standard Calder\'on-Zygmund extrapolation. 
One can then prove appropriate Littlewood-Paley square function estimates to show that $H^{p}_{L} = L^{p}(\R^{n})$ for a range of values of $p$ (see \cite{Auscher07, HoLuMiMiYa11}).

It is illustrative to compare this development to the $L^{p}$-theory for hyperbolic equations. It is a classical fact that the solution operators $(\cos(t\sqrt{-\Delta}))_{t\in\R}$ and $(\sin(t\sqrt{-\Delta}))_{t\in\R}$ to the wave equation $\partial_{t}^{2}u=\Delta u$ are not bounded on $L^{p}(\Rn)$ for $n\geq2$ if $p\neq 2$ and $t\neq 0$. In fact, it was shown in \cite{Peral80, Miyachi80,Beals82} that these operators are bounded from $W^{s_{p},p}(\Rn)$ to $L^{p}(\Rn)$ for $p\in(1,\infty)$ and $s_{p}=(n-1)|\tfrac{1}{p}-\tfrac{1}{2}|$, and this exponent $s_p$ is sharp. These statements are in turn instances of a more general phenomenon concerning the class of Fourier integral operators. Indeed, a Fourier integral operator (FIO) of order zero, for which the underlying canonical relation is a local canonical graph and the Schwartz kernel is compactly supported, is bounded from $W^{s_{p},p}(\Rn)$ to $L^{p}(\Rn)$, as was shown by Seeger, Sogge and Stein in \cite{SeSoSt91}. The heart of their proof is to show that such an operator maps a suitable Sobolev space over the classical Hardy space $H^{1}$ to $L^{1}$; after that one simply uses duality and interpolation.

Although the work of Seeger, Sogge and Stein was a major breakthrough, it left open the question whether there is a natural subspace of $L^{1}(\Rn)$ that is invariant under Fourier integral operators of order zero. Since, under additional assumptions, the composition of two Fourier integral operators is again a Fourier integral operator, one would expect this to be the case. And indeed, in \cite{Smith98a} Hart Smith introduced a subspace of the local Hardy space $\HT^{1}(\Rn)$ that is invariant under Fourier integral operators of order zero. Moreover, this subspace contains all functions for which the fractional derivative of order $(n-1)/2$ also belongs to $\HT^{1}(\Rn)$, which in turn allowed Smith to recover the existing results on $L^{p}$-boundedness. 

Unfortunately, so far Smith's work has not had the same impact on the $L^{p}$-theory of oscillatory integral operators that the classical Hardy space has had on the theory of singular integral operators. One of the goals of this article is to change this status quo, by developing a full Hardy space theory for Fourier integral operators that involves techniques from modern harmonic analysis and is adapted to applications to wave equations with rough coefficients. 

We introduce a scale $\Hp$, $p\in[1,\infty]$, of spaces that are invariant under Fourier integral operators of order zero and satisfy the Sobolev embeddings
\begin{equation}\label{eq:Sobolevintro}
W^{r_{p},p}(\Rn)\subseteq \Hp\subseteq W^{-r_{p},p}(\Rn)
\end{equation}
for $1<p<\infty$ and $r_{p}=\frac{n-1}{2}|\frac{1}{p}-\frac{1}{2}|$. Moreover, $\HT^{1}_{FIO}(\Rn)$ coincides (up to a shift in differentiability) with Smith's invariant space from \cite{Smith98a}. In particular, we recover the known results on $L^{p}$-boundedness for Fourier integral operators, and in fact we extend them in several ways. First, it should be noted that having an invariant space allows one to use iterative constructions to construct parametrices for equations, whereas this is not possible using the results from \cite{SeSoSt91}. We also extend the results from both \cite{SeSoSt91} and \cite{Smith98a} by removing the assumption that the relevant Fourier integral operators have compactly supported Schwartz kernels (see \cite{IsRoSt18} for similar global results). In particular, this allows us to directly prove that the wave operators $(\cos(t\Delta))_{t\in\R}$ and $(\sin(t\Delta))_{t\in\R}$ are bounded on $\Hp$, thereby not only recovering the results of \cite{Peral80} but providing a subspace of $L^{p}(\Rn)$ on which the wave equation is well-posed. Finally, we extend the class of symbols from the Kohn-Nirenberg class $S^{0}$ to a class $S^{0}_{\frac{1}{2},\frac{1}{2},1}$ that is similar to H\"{o}rmander's $S^{0}_{\frac{1}{2},\frac{1}{2}}$ symbols, except that the symbols decay faster when differentiated in the radial direction for the fiber variable.

\subsection{Hardy spaces and tent spaces over the cosphere bundle}

We construct $\mathcal{H}^{p}_{FIO}(\R^{n})$, and prove boundedness of FIOs on this space, using \emph{tent spaces}, in a way that directly follows the pattern used in the recent developments of harmonic analysis beyond Calder\'on-Zygmund theory, originating in the work of Coifman-Meyer-Stein \cite{CoMeSt85}. The standard tent spaces $T^{p}(\R^{n})$ over $\R^n$ form a single interpolation scale that contain most of the usual function spaces of harmonic analysis \cite[Chapter IV, Section 6B]{Stein93}. They are highly useful in extending Calder\'on-Zygmund theory, because they appear to be somewhat universal, in the sense that they do not need to be adapted to a rough situation. When dealing with spaces such as the aforementioned $H^{p}_{L}$ spaces, one merely needs to change the embedding of $H^{p}_L$  into $T^{p}$, not $T^{p}$ itself. Moreover, the extrapolation theory of all the known rough generalisations of singular integral operators works directly on the entire $T^{p}$ scale. Only the embedding of  $H^{p}_{L}$ into $T^p$ produces restrictions in $p$ (see e.g. the memoir \cite{HoLuMiMiYa11}). 

This is why we construct $\mathcal{H}^{p}_{FIO}(\R^{n})$ as a subspace of an appropriate tent space. However, instead of the standard tent space $T^{p}(\R^{n})$, we use a tent space $T^{p}(\Sp)$ over the cosphere bundle $\Sp=\Rn\times S^{n-1}$ of $\R^n$. Let us describe this tent space in more detail. A fundamental idea, used by Smith in \cite{Smith98a} and crucial in our approach as well, is that FIOs acting on function spaces over $\R^{n}$ are more effectively understood when lifted to function spaces over the cosphere bundle $S^{*}(\R^{n})$. 
For this reason, we consider the tent space $T^{p}(\Sp)$ of functions $F$ defined on $\Sp \times (0,\infty)$ (which can be identified with phase space minus the zero section) such that
\[
\|F\|_{T^{p}(S^{*}(\R^{n}))}^{p} = 
\int_{S^{*}(\R^{n})} \Big(\int_{0}^{\infty}\fint_{B_{\sqrt{\sigma}}(x,\w)}|F(y,\nu,\sigma)|^{2}\ud y\ud \nu\frac{\ud \sigma}{\sigma}\Big)^{p/2} \ud x\ud\omega<\infty,
\]
for $p<\infty$. Here $B_{\sqrt{\sigma}}(x,\w)$ denotes a ball of radius $\sqrt{\sigma}$ with respect to a specific metric $d$ on $\Sp$ that is introduced in Section \ref{subsec:metric}. This metric arises naturally from contact geometry, and turns $\Sp$ into a doubling metric measure space over which tent space theory is well understood (see  \cite{Amenta14} and the references therein). 

The Hardy space $\mathcal{H}^{p}_{FIO}(\R^{n})$ is defined by lifting functions on $\R^n$ to functions on phase space via a \emph{wave packet transform}. Our wave packet transform $W$ is defined for $\sigma<1$, i.e.~in the high frequency regime, by 
\[
(Wf)(x, \omega, \sigma) = (\psi_{\omega, \sigma}(D)f)(x).
\]
Here $\psi_{\omega,\sigma}(D)$ is a Fourier multiplier with compact support in a region where $\xi$ satisfies $|\xi| \eqsim \sigma ^{-1}$ and  $|\hat{\xi}-\w|\eqsim \sigma^{\frac{1}{2}}$.  For $\sigma \geq 1$, we include `by hand' a copy of $f$ localized to low frequencies; in this regime, localization in the $\omega$ variable is not relevant. All of this is done in such a way that $W$ is an isometry on $L^2$.  The $\mathcal{H}^{p}_{FIO}(\R^{n})$ norm is then defined as 
\[
\| f \|_{\mathcal{H}^{p}_{FIO}(\R^{n})} :=   \| Wf\|_{T^{p}(S^{*}(\R^{n}))}. 
\]

This choice of wave packet transform effectively encodes a key technical aspect of the Seeger-Sogge-Stein proof from \cite{SeSoSt91}: the dyadic-parabolic decomposition. The idea, which goes back to Fefferman \cite{Fefferman73b}, is to decompose the standard dyadic annuli further, into narrow bands whose width is comparable to the square root of their length. This decomposition has proven to be an effective tool for the analysis of FIOs; for example, it yields an optimally sparse representation of FIOs as infinite matrices \cite{Candes-Demanet05}. It is also directly related to the distance $d$ on $\Sp$  mentioned above. In fact, the inverse Fourier transform of a wave packet $\psi_{\w, \sigma}$ is concentrated on the projection to $\R^n$ of a ball of radius $\sqrt{\sigma}$ in the $d$-metric, and this set is an anisotropic ball with different radii in directions parallel to and perpendicular to $\omega$. 

As in the general theory of tent spaces (see e.g. \cite{AuKrMoPo12}), boundedness of an integral operator follows from a decay property of its kernel. Here, we have to generalize the usual notion of off-diagonal decay to a notion of \emph{off-singularity decay}. This is because, while the kernel of a singular integral operator is only singular on the diagonal, the Schwartz kernel of an FIO $T$ has singularities determined by its \emph{canonical relation}. For an FIO $T$ associated with a homogeneous canonical transformation, which is the sort we consider here, this is a transformation $\hat \chi$ on the cosphere bundle, and it specifies how $T$ acts on the \emph{wavefront set} of a distribution (the wavefront set  of a distribution $u$ is a closed subset of $\Sp$ describing the location and direction of the singularities of $u$). That is, singularities `propagate', according to $\hat \chi$ under the action of an FIO, instead of remaining fixed as they are under a singular integral or pseudodifferential operator. By working over the cosphere bundle, adding this level of generalization barely affects the tent space theory. The off-singularity bounds, relative to $\hat \chi$, for a kernel $K_{\sigma, \tau}(x, \omega; y, \nu)$ acting on tent spaces over $\Sp$, are of the form 
\[
| K_{\sigma, \tau}(x, \omega; y, \nu)| \leq C \rho^{-n} \big(\max\big(\tfrac{\sigma}{\tau}, \tfrac{\tau}{\sigma}\big)\big)^{-N} \Big(1 + \frac{ d((x, \omega), \hat \chi(y, \nu))^2}{\rho} \Big)^{-N} 
\]
for $C,N\geq0$ and $\rho = \min(\sigma, \tau)$. This bears a remarkable resemblance to the off-diagonal estimates from the parabolic theory. From this point of view, FIOs, lifted via wave packet transforms, are analogous to a \emph{diffusion} on the cosphere bundle, with $\rho$, the inverse of frequency, playing the role of a time variable.

\subsection{Main results} The main results of this paper are:

\begin{itemize}
\item Theorem~\ref{thm:FIObddHardy}, which states that FIOs of order zero associated with a canonical graph are bounded on $\mathcal{H}^{p}_{FIO}(\R^{n})$ for all $p \in [1, \infty]$. The proof is divided into two parts: we first show that integral operators satisfying off-singularity bounds are bounded on the tent spaces $T^p(\Sp)$ (Theorem~\ref{thm:tentbounded}). We then show that FIOs, lifted to tent spaces via the wave packet transform, satisfy off-singularity bounds (Theorem~\ref{thm:offsingFIO}). 
\item Theorem~\ref{thm:Sobolev}, in which we prove the embeddings \eqref{eq:Sobolevintro}. 
\end{itemize}
Aside from these principal results, we prove fundamental properties of the scale of Hardy spaces for FIOs such as interpolation, duality and density results. We also obtain a similar molecular decomposition of $\HT^{1}_{FIO}(\Rn)$ as was obtained by Smith for his invariant space in \cite{Smith98a}, and we use it to argue that our $\HT^{1}_{FIO}(\Rn)$ is a Sobolev space over Smith's invariant space.

\subsection{Previous literature} 
The method of constructing appropriate Hardy spaces for extrapolation of $L^{2}$ bounds, then identifying said spaces through Littlewood-Paley estimates, has recently been successful in solving elliptic and parabolic PDE problems. The literature on the subject is already vast and expanding rapidly. We just mention a few representative results. Hardy spaces for differential forms on Riemannian manifolds are constructed in \cite{AuMcIRu08}, while their counterpart for elliptic operators with rough coefficients on $\R^{n}$ are introduced in \cite{Hofmann-Mayboroda09} (see also the earlier \cite{Duong-Yan05}). This is applied to elliptic boundary value problems in \cite{Auscher-Stahlhut16} (which gives a new approach, valid for systems, to the groundbreaking a priori estimates proved in \cite{HoKeMaPi15a} for non-symmetric equations), and to initial value problems for parabolic systems in \cite{AuMoPo15} (simultaneously extending Aronson's theory to systems, and Lions' theory to $L^p$). Beyond such results, which use the method in a direct fashion, there is a range of recent articles that exploit the same circle of ideas to go beyond Calder\'on-Zygmund theory through a combination of $T(b)$ arguments for the $L^2$ estimates, and generalised Hardy space extrapolation to extend the estimates to $L^p$. 
This includes recent breakthroughs in geometric measure theory (see e.g. the memoir \cite{HoMiMiMo17} as a starting point).

Wave packet transforms have a long tradition in harmonic and microlocal analysis (see \cite{Folland89, Cordoba-Fefferman78} and \cite[Chapter 3]{Martinez02}). They have also been used in the study of wave equations with rough coefficients by e.g.~Smith \cite{Smith98b,Smith06,Smith14} and Tataru \cite{Tataru99,Tataru00,Tataru01,Tataru02,MaMeTa08}. 
Our notion of off-singularity decay appears in the work of Smith \cite{Smith98a,Smith98b} and Geba and Tataru \cite{Geba-Tataru07}, and similar concepts have been used in the numerical analysis of wave equations \cite{Candes-Demanet03} and in the time-frequency analysis of Schr\"{o}dinger operators \cite{CoFaRo10}.

As already indicated, a major source of inspiration for us has been the visionary work of Hart Smith \cite{Smith98a}.  We find it remarkable that his impressive results have hardly been used or developed in any significant way since their publication 20 years ago. Even the author himself did not use it in his subsequent work on rough wave equations. It is our belief that Smith's work, suitably developed as in the present paper,  will furnish a powerful tool for tackling low regularity and nonlinear hyperbolic equations. We intend to develop these ideas further in future work.

\subsection{Organization of this paper} This article is organized as follows. Section \ref{sec:tentFIO} contains background knowledge for the rest of the article. In Section \ref{subsec:metric} we introduce the relevant metric structure on the cosphere bundle, and we study some of its properties. Section \ref{subsec:tent} contains the basics of tent space theory, including a distribution theory that is used in later sections. Section \ref{subsec:FIOs} contains some basics on Fourier integral operators and the $S^{0}_{\frac{1}{2},\frac{1}{2},1}$ symbol class. In Section \ref{sec:offsing} we introduce the notion of off-singularity decay, and we show that it implies boundedness on tent spaces. We also define so-called residual families; these play the same role for tent spaces that smoothing operators play for microlocal analysis on $\Rn$. Section \ref{sec:wavetransforms} deals with wave packet transforms, and in Section \ref{sec:offsingFIO} we show that the kernels of FIOs, lifted to phase space using wave packet transforms, satisfy off-singularity bounds.  This is the most technical part of the paper, involving delicate analysis and repeated integration by parts to obtain the required bounds. Armed with these powerful bounds, it becomes an easy matter to introduce the Hardy spaces and study their basic properties in Section \ref{sec:Hardy spaces}. Section \ref{sec:Sobolev} then contains the proof of the Sobolev embeddings, and Section \ref{sec:decomp} the molecular decomposition of $\HT^{1}_{FIO}(\Rn)$.

\subsection{Notation}

The natural numbers are $\N=\{1,2,\ldots\}$, and $\Z_{+}=\N\cup\{0\}$. Throughout this article we fix $n\in\N$ with $n\geq2$. The approach which we use can also be applied for $n=1$. However, in this case the theory is less involved and one simply recovers the classical $L^{p}$-spaces for $p\in(1,\infty)$, cf.~Theorem \ref{thm:Sobolev}.

For $1\leq i\leq n$ we denote the $i$-th standard basis vector of $\Rn$ by $e_{i}$. 
For $\xi,\eta\in\Rn$ we write $\lb\xi\rb=(1+|\xi|^{2})^{1/2}$ and $\lb \xi,\eta\rb=\xi\cdot\eta$, and if $\xi\neq 0$ then $\hat{\xi}=\xi/|\xi|$. We denote by $P_{\w}^{\perp}:\Rn\to\Rn$ the projection onto the hyperplane orthogonal to $\w\in S^{n-1}$. 

The spaces of Schwartz functions and tempered distributions are $\Sw(\Rn)$ and $\Sw'(\Rn)$, respectively. The duality between $f\in\Sw(\Rn)$ and $g\in\Sw'(\Rn)$ is denoted by $\lb f,g\rb$. The Fourier transform of $f\in\Sw'(\Rn)$ is denoted by $\F f$ or $\widehat{f}$, and the inverse Fourier transform by $\F^{-1}f$ or $\widecheck{f}$. If $f\in L^{1}(\Rn)$ then
\begin{align*}
\F f(\xi)=\int_{\Rn}e^{-i x\cdot \xi}f(x)\ud x\quad (\xi\in\Rn).
\end{align*}
We use multi-index notation, where $\partial^{\alpha}_{\xi}=\partial^{\alpha_{1}}_{\xi_{1}}\ldots\partial^{\alpha_{n}}_{\xi_{n}}$ and $\xi^{\alpha}=\xi_{1}^{\alpha_{1}}\ldots\xi_{n}^{\alpha_{n}}$ for $\xi=(\xi_{1},\ldots,\xi_{n})\in\Rn$ and $\alpha=(\alpha_{1},\ldots,\alpha_{n})\in\Z_{+}^{n}$. For a smooth function $\Phi\in C^{\infty}(V)$ on an open subset $V\subseteq\Rn\times\Rn$, we denote by $\partial^{2}_{x\xi}\Phi$ the mixed Hessian with respect to the first variable $x$ and the second variable $\xi$, and similarly for $\partial^{2}_{xx}\Phi$ and $\partial^{2}_{\xi\xi}\Phi$. For $m:\Rn\to\C$ a measurable function of temperate growth, $m(D)$ is the Fourier multiplier with symbol $m$. 

The volume of a measurable subset $B$ of a measure space $(\Omega,\mu)$ is $V(B)$. For an integrable $F:B\to\C$, we write
\[
\fint_{B}F(x)\ud\mu(x)=\frac{1}{V(B)}\int_{B}F(x)\ud\mu(x)
\]
if $V(B)<\infty$. If $B$ is a ball around $x\in\Rn$ with radius $r>0$, then $cB$ denotes the ball around $x$ with radius $cr$, for $c>0$.

Throughout, we let $\Upsilon:(0,\infty)\to(0,1]$ be given by
\begin{equation}\label{eq:upsilon}
\Upsilon(t):=\min(t,t^{-1})\quad(t>0).
\end{equation}
The H\"{o}lder conjugate of $p\in[1,\infty]$ is denoted by $p'$. The indicator function of a set $E$ is denoted by $\ind_{E}$. The space of continuous linear operators between Banach spaces $X$ and $Y$ is $\La(X,Y)$, and $\La(X):=\La(X,X)$. We write $f(s)\lesssim g(s)$ to indicate that $f(s)\leq Cg(s)$ for all $s$ and a constant $C\geq0$ independent of $s$, and similarly for $f(s)\gtrsim g(s)$ and $g(s)\eqsim f(s)$.

\section{Tent spaces and Fourier integral operators}\label{sec:tentFIO}

In this section we introduce a metric structure on the cosphere bundle over $\Rn$, and we collect some basics on tent spaces and on Fourier integral operators.

\subsection{A metric on the cosphere bundle}\label{subsec:metric}

Here we introduce the fundamental metric for this article. This metric arises naturally from a contact structure on the cosphere bundle, and the connection to contact geometry facilitates e.g.~the proof of Proposition \ref{prop:Lipschitz} below. However, in later sections we will only use that the resulting metric measure space is doubling, and in computations we will work with the equivalent analytic expressions in \eqref{eq:quasi-metric} and \eqref{eq:qm-equiv}. Hence readers unfamiliar with the basics of contact geometry can simply consider these analytic expressions, while keeping in mind that we work throughout on a doubling metric measure space. We also note that the same metric structure was considered in \cite{Geba-Tataru07}.

Let $\Tp$ be the cotangent bundle of $\Rn$, identified with $\Rn\times\Rn$ and endowed with the symplectic form $d\xi \cdot dx$ and the Liouville measure $\ud x\ud\xi$.  
Let $\Sp=\Rn\times S^{n-1}$ be the cosphere bundle over $\Rn$. We shall denote elements of $S^{n-1}$ by either $\hat \xi$, for $\xi \in \Rn \setminus \{ 0 \}$, or by $\omega$ or $\nu$. We also denote the standard Riemannian metric on $S^{n-1}$ by $g_{S^{n-1}}$, and the standard measure on $S^{n-1}$ by $\ud \omega$. 

The cosphere bundle $\Sp$ is a contact manifold with respect to the standard contact form $\alpha_{S^{n-1}}:= \hat \xi \cdot dx$. This form determines the contact structure, that is, the smooth distribution of codimension 1 subspaces of the tangent space $T (\Sp)$ given by the kernel of $\alpha_{S^{n-1}}$. The product metric $dx^2 + g_{S^{n-1}}$ and the contact form together determine a sub-Riemannian metric $d$ on $\Sp$, given by the formula 
\begin{equation}
d((x, \omega), (y, \nu)) = \inf_{\gamma} \int_0^1 |\gamma'(s)|\ud s
\label{d-defn}\end{equation}
for $(x,\w),(y,\nu)\in\Sp$. Here the infimum is taken over all piecewise $C^1$ curves $\gamma : [0,1] \to \Sp$ such that $\gamma(0) = (x, \omega)$, $\gamma(1) = (y, \nu)$, $\gamma$ is horizontal in the sense that $\alpha_{S^{n-1}}(\gamma'(s)) = 0$ for almost all $s \in [0,1]$, and $|\gamma'(s)|$ is the length of the velocity vector $\gamma'(s)$ with respect to the product metric $dx^2 + g_{S^{n-1}}$. We can also write 
(see \cite[Lemma 1.4.2]{Jost11})
\begin{equation}\label{eq:d2}
d((x, \omega), (y, \nu))^2 = \inf_{\gamma} \int_0^1 |\gamma'(s)|^2\ud s. 
\end{equation}
It is often more convenient to work with the quasi-metric $\wt{d}$ given by
\begin{equation}\label{eq:quasi-metric}
\wt{d}((x,\w),(y,\nu)):=\big(|\lb\w,x-y\rb|+|\lb\nu,x-y\rb|+|x-y|^{2}+|\w-\nu|^{2}\big)^{1/2}
\end{equation}
for $(x,\w),(y,\nu)\in\Sp$, or the expression
\begin{equation}\label{eq:qm-equiv}
\utilde{d}((x,\w),(y,\nu)):=\big(|\lb\w,x-y\rb|+|x-y|^{2}+|\w-\nu|^{2}\big)^{1/2}.
\end{equation}
The following lemma essentially already appeared as \cite[Lemma 7.1]{Geba-Tataru07}. 

\begin{lemma}\label{lem:metric}
The metric $d$ is equivalent to the quasi-metric $\wt{d}$, and to $\utilde{d}$. 
\end{lemma}

\begin{proof} 
We first note that $\wt{d}$ is equivalent to $\utilde{d}$. To see this, simply observe that 
\[
|\lb\nu,x-y\rb|\leq |\lb\w,x-y\rb|+|\lb\w-\nu,x-y\rb|\leq|\lb\w,x-y\rb|+|x-y|^{2}+|\w-\nu|^{2}.
\] 
Hence it suffices to prove that 
\begin{equation}
d((x, \omega), (y, \nu))^2 \eqsim |\lb\w,x-y\rb|+|x-y|^{2}+|\w-\nu|^{2}
\label{eq:d-approx}\end{equation}
for all $(x,\w),(y,\nu)\in\Sp$.

To show that 
\begin{equation}\label{eq:dgtrsim}
d((x, \omega), (y, \nu))^2 \gtrsim |\lb\w,x-y\rb|+|x-y|^{2}+|\w-\nu|^{2},
\end{equation}
first note that $d((x, \omega), (y, \nu))^2 \geq |x-y|^{2}+|\w-\nu|^{2}$, since $d$ is bounded from below by the Riemannian distance with respect to $dx^2 + g_{S^{n-1}}$, which in turn is bounded from below by the Euclidean distance on $\R^{2n}$ restricted to $\Sp$. On the other hand, consider a piecewise $C^{1}$ horizontal curve $\gamma$ from $(x,\w)$ to $(y,\nu)$. Write $\gamma(s) = (x(s), \hat \xi(s))$ for $s \in [0,1]$. Since $\gamma$ is horizontal, we have $\hat \xi(s) \cdot x'(s) = 0$ for almost all $0\leq s\leq 1$. Then
\begin{align*}
|\lb\w,x-y\rb | &\leq \Big| \int_0^1  \omega \cdot x'(s)\ud s \Big|=\Big|  \int_0^1 (\hat \xi(0) - \hat \xi(s)) \cdot x'(s) \ud s \Big| \\
&= \Big| \int_0^1 \Big( \int_0^s \hat \xi'(t) \ud t \Big) x'(s) \ud s \Big|\leq \int_0^1 \int_0^1 |x'(s)| |\hat \xi'(t)| \ud t\ud s\\
&\leq \frac1{2} \int_0^1 \int_0^1 \big( |x'(s)|^2 +  |\hat \xi'(t)|^2  \big)\ud t\ud s= \frac1{2} \int_0^1 |\gamma'(s)|^2 \ud s. 
\end{align*}
Taking the infimum over all such $\gamma$, and using \eqref{eq:d2}, we obtain \eqref{eq:dgtrsim}. 

For the other inequality in \eqref{eq:d-approx}, by \eqref{eq:d2} we need only exhibit a particular piecewise $C^{1}$ horizontal curve $\gamma$ connecting $(x, \omega)$ and $(y, \nu)$ such that
\begin{equation}\label{eq:dsmaller}
\int_0^1 |\gamma'(s)|^2\ud s\lesssim |\lb\w,x-y\rb|+|x-y|^{2}+|\w-\nu|^{2}.
\end{equation}
This is straightforward when the right-hand side of \eqref{eq:dsmaller} is bounded away from zero by a fixed constant $c>0$, to be chosen later. Indeed, moving at constant speed in time intervals of equal size, first move from $\w$ to an $\w'\in S^{n-1}$ orthogonal to $y-x$ while keeping $x$ fixed, then move from $x$ to $y$ in a straight line while keeping $\w'$ fixed, and finally move from $\w'$ to $\nu$ while keeping $y$ fixed. 

On the other hand, suppose that the right-hand side of \eqref{eq:dsmaller} is smaller than $c$, say
\[
\veps^{2}:=|\lb\w,x-y\rb|+|x-y|^{2}+|\w-\nu|^{2}
\]
for some $\veps<c$. This implies in particular that $|x-y| \leq \veps$, that $|\omega - \nu| \leq \veps$ and that $|\lb\w,x-y\rb| \leq \veps^2$. 
Consider the following five moves:
\begin{itemize}
\item First, move from $\w$ to an $\w'\in S^{n-1}$ in the $xy\w$-plane with $|\w-\w'|=\veps$, while keeping $x$ fixed; 
\item Then move from $x$, while keeping $\w'$ fixed, in a straight line orthogonal to $\omega'$ such that $|\lb x'(s), \omega \rb| \eqsim \sin(\veps)$, to a point $z$ at a distance of at most $\veps^2/\sin(\veps)$ of $x$ and such that $\lb\w,y-z\rb= 0$; 
\item Then move from $\w'$ back to $\omega$ while keeping $z$ fixed; 
\item Next, move in a straight line from $z$ to $y$ while keeping $\w$ fixed. Note that the distance from $z$ to $y$ is at most $\veps + \veps^2/\sin(\veps)$, by the triangle inequality; 
\item Finally, move from $\omega$ to $\nu$ while keeping $x$ fixed, a distance of at most $\veps$. 
\end{itemize}
The second step is possible since $|\lb \omega, x-y \rb| \leq \veps^2$. By performing these moves at constant speed in equal time intervals in $[0,1]$, this procedure yields a horizontal curve $\gamma$ such that
\[
\int_0^1 |\gamma'(s)|^2\ud s\lesssim \veps^{2}+\tfrac{\veps^{4}}{\sin(\veps)^{2}}\lesssim \veps^{2},
\]
for $c>0$ sufficiently small. 
\end{proof}

\begin{remark}\label{rem:Hartsdistance}
It should be noted that the quasi-metric in \eqref{eq:quasi-metric} is closely related to the square root of the quasi-metric from \cite{Smith98a}, which is given by
\[
d^{*}((x,\w),(y,\nu)):=|\lb \w,x-y\rb|+|\lb \nu,x-y\rb|+\min(|x-y|,|x-y|^{2})+|\w-\nu|^{2}
\]
for $(x,\w),(y,\nu)\in\Sp$. Whether one works with $d^{*}$ or its square root is inconsequential; it merely affects the scaling in some definitions and proofs. On the other hand, the factor $\min(|x-y|,|x-y|^{2})$ only differs from $|x-y|^{2}$ at large distances, and it results in a quasi-metric that is homogeneous with respect to the volume of balls. More precisely, the volume of a ball of radius $\tau>0$ with respect to the square root of $d^{*}$ is approximately $\tau^{2n}$, whereas $d$ only has this property for small $\tau$ and it differs fundamentally on large scales (see Lemma \ref{lem:doubling} below). 

Nonetheless, $d$, $d^{*}$ and its square root all have the doubling property, and it is this property that is crucial for the tent space theory in this article. Moreover, $d$ is equivalent to the square root of $d^{*}$ at small scales, while large scales are less relevant for most of the theory. We work with the metric $d$ because we believe that it leads to a more natural and robust metric structure on $\Sp$. 
\end{remark}

Throughout, we let $B_{\tau}(x,\w)\subseteq\Sp$ be the open ball around $(x,\w)\in\Sp$ of radius $\tau>0$ with respect to $d$, and similarly for the balls $B_{\tau}(x)\subseteq\Rn$ and $B_{\tau}(\w)\subseteq S^{n-1}$ with respect to the standard Euclidean metric. Clearly, the volume $V(B_{\tau}(x,\w))$ of the ball $B_{\tau}(x,\w)$ depends only on $\tau$, and not on $(x,\w)$. 

The following lemma is straightforward to prove using e.g.~Lemma \ref{lem:metric}, keeping in mind that $S^{n-1}$ is compact.

\begin{lemma}\label{lem:doubling}
There exists a $C>0$ such that, for all $(x,\w)\in\Sp$, one has
\[
\frac{1}{C}\tau^{2n}\leq V(B_{\tau}(x,\w))\leq C\tau^{2n}
\]
if $\tau\in(0,1)$, and
\[
\frac{1}{C}\tau^{n}\leq V(B_{\tau}(x,\w))\leq C\tau^{n}
\]
if $\tau\geq 1$. In particular, 
\[
V(B_{c\tau}(x,\w))\leq C^{2}c^{2n}V(B_{\tau}(x,\w))
\]
for all $\tau>0$ and $c\geq1$, and $(\Sp,d,\ud x\ud\w)$ is a doubling metric measure space. 
\end{lemma}

Recall that a conic subset of $\Tp$ is a subset $U\subseteq\Tp\setminus o$, where $o:=\Rn\times\{0\}$ is the zero section, that is closed under positive dilations in the fiber variables. The \emph{base} of a conic subset $U$ is the projection of $U$ to $\Sp$. 
A \emph{homogeneous canonical transformation} is a diffeomorphism between open conic subsets of $\Tp$ that preserves the symplectic form and commutes with positive dilations in the fiber variable, and a \emph{contact transformation} $\hchi$ on $\Sp$ is a diffeomorphism between open subsets of $\Sp$ such that the tangent map $T(\hchi)$ preserves the horizontal distribution, or equivalently, such that $\hchi^{*}\alpha_{S^{n-1}}=f\alpha_{S^{n-1}}$ for a nowhere vanishing function $f$ on $\Sp$. It is straightforward to check that each homogeneous canonical transformation $\chi$ on $\Tp$ induces a contact transformation $\hchi$ on $\Sp$ by projection. 

A consequence of the contact nature of the sub-Riemannian metric $d$ is the following proposition, similar to \cite[Lemma 2.3]{Smith98a}. One difference is that the constant of equivalence below depends only on the first derivatives of $\hchi$, whereas the proof of \cite[Lemma 2.3]{Smith98a} also uses bounds for its second derivatives. We write $|T_{(x,\w)}(\hchi)|$ for the operator norm, with respect to $dx^{2}+g_{S^{n-1}}$, of the tangent map $T_{(x,\w)}(\hat{\chi})$ of $\hat{\chi}$ at $(x,\w)\in\dom(\hat{\chi})$. 

\begin{proposition}\label{prop:Lipschitz}
Let $\hchi$ be a contact transformation on $\Sp$. Suppose that there exists a $C>0$ such that $|T_{(x, \omega)}(\hchi)|\leq C$ for all $(x,\w)\in\dom(\hat{\chi})$ and $|T_{(x, \omega)}(\hchi)^{-1}|\leq C$ for all $(x,\w)\in\ran(\hchi)$. Let $U\subseteq\dom(\hchi)$ be such that either $U=\dom(\hchi)=\ran(\hchi)=\Sp$, or $U$ is compact. Then $\hchi:U\to\hchi(U)$ is bi-Lipschitz with respect to $d$, and there exists a $C'>0$ depending only on $C$ and $U$ such that  
\begin{equation}
\frac{1}{C'}d((x, \omega), (y, \nu)) \leq d(\hchi(x, \omega), \hchi(y, \nu)) \leq C' d((x, \omega), (y, \nu))
\label{eq:biLip}\end{equation}
for all $(x,\w)\in U$.
\end{proposition}
\begin{proof} 
If $U=\dom(\hchi)=\ran(\hchi)=\Sp$ then the proof is almost trivial. Composition with $\hchi$ sets up an isomorphism between the horizontal curves joining $(x, \omega)\in\Sp$ with $(y, \nu)\in\Sp$, and the horizontal curves joining $\hchi(x, \omega)$ with $\hchi(y, \nu)$. Then the assumed bound on the operator norm of the tangent map of $\hchi$ and its inverse, together with the definition \eqref{d-defn} of the distance $d$, immediately yields \eqref{eq:biLip}. 

If $\dom(\hchi)\neq \Sp$ then the same argument does not work, since a curve connecting points in the domain of $\hchi$ may leave the domain of $\hchi$. However, if $U$ is compact then there exists an $\veps>0$ such that $B_{\veps}(x,\w)\subseteq \dom(\hchi)$ for all $(x,\w)\in U$. Then the same argument as before does show that the second inequality in \eqref{eq:biLip} holds if $d((x,\w),(y,\nu))<\veps$. On the other hand, the second inequality is trivial if $d((x,\w),(y,\nu))\geq \veps$, since $\hat{\chi}(U)$ is compact. By symmetry, this suffices.
\end{proof}

From here on, unless indicated otherwise, a Lipschitz map on $\Sp$ is a map that is Lipschitz with respect to the metric $d$.

\subsection{Tent spaces}\label{subsec:tent}

For the basics of tent space theory, including many of the statements below, see e.g.~\cite{CoMeSt85, Amenta14}. The corresponding results for the parabolic tent spaces which we use follow from a straightforward change of variables (see also \cite{AuKrMoPo12}).

Let $\Spp:=\Sp\times(0,\infty)$, endowed with the measure $\ud x\ud\w\frac{\ud\sigma}{\sigma}$. Set
\[
\Gamma(x,\w):=\{(y,\nu,\sigma)\in\Spp\mid (y,\nu)\in B_{\sqrt{\sigma}}(x,\w)\}
\]
for $(x,\w)\in\Sp$, and $\Gamma(U):=\cup_{(x,\w)\in U}\Gamma(x,\w)$ for $U\subseteq\Sp$. Recall that $B_{\sqrt{\sigma}}(x,\w)$ is the open ball around $(x,\w)$ of radius $\sqrt{\sigma}$ with respect to the metric $d$ from the previous section. If $U$ is open, then the tent over $U$ is $T(U):=\Spp\setminus \Gamma(U^{c})$. Note that $(x,\w,\sigma)\in T(U)$ if and only if $d((x,\w),U^{c})\geq \sqrt{\sigma}$.

For $F:\Spp\to\C$ measurable and $(x,\w)\in\Sp$, set
\[
\A F(x,\w):=\Big(\int_{0}^{\infty}\fint_{B_{\sqrt{\sigma}}(x,\w)}|F(y,\nu,\sigma)|^{2}\ud y\ud \nu\frac{\ud \sigma}{\sigma}\Big)^{1/2}\in[0,\infty]
\]
and
\[
\mathcal{C}F(x,\w):=\sup_{B}\Big(\frac{1}{V(B)}\int_{T(B)}|F(y,\nu,\sigma)|^{2}\ud y\ud \nu\frac{\ud \sigma}{\sigma}\Big)^{1/2}\in[0,\infty],
\]
where the supremum is taken over all balls $B\subseteq \Sp$ containing $(x,\w)$. Then $\A F$ and $\mathcal{C}F$ are lower semi-continuous as maps from $\Sp$ to $[0,\infty]$, and in particular measurable.

\begin{definition}\label{def:tentspace}
For $p\in[1,\infty)$, the \emph{tent space} $T^{p}(\Sp)$ consists of all $F\in L^{2}_{\loc}(\Spp)$ such that $\A F\in L^{p}(\Sp)$, endowed with the norm
\[
\|F\|_{T^{p}(\Sp)}:=\|\A F\|_{L^{p}(\Sp)}.
\]
Also, $T^{\infty}(\Sp)$ consists of all $F\in L^{2}_{\loc}(\Spp)$ such that $\mathcal{C}F\in L^{\infty}(\Sp)$, with 
\[
\|F\|_{T^{\infty}(\Sp)}:=\|\mathcal{C}F\|_{L^{\infty}(\Sp)}.
\]
\end{definition}

For all $p\in[1,\infty]$, the tent space $T^{p}(\Sp)$ is a Banach space, and $T^{p}(\Sp)\cap T^{2}(\Sp)$ is dense in $T^{p}(\Sp)$ for $p<\infty$. Also, one has 
\[
T^{2}(\Sp)=L^{2}(\Spp):=L^{2}\big(\Spp,\ud x\ud\w\tfrac{\ud\sigma}{\sigma}\big)
\]
with equality of norms, as a straightforward application of Fubini's theorem shows. 

It is instructive to compare the `conical' square function in the $T^{p}(\Sp)$-norm to its `vertical' analogue. By \cite[Proposition 2.1 and Remark 2.2]{AuHoMa12}, one has 
\begin{equation}\label{eq:vertical}
\Big(\int_{\Sp}\Big(\int_{0}^{\infty}|F(x,\w,\sigma)|^{2}\frac{\ud\sigma}{\sigma}\Big)^{p/2}\ud x\ud\w\Big)^{1/p}\lesssim \|F\|_{T^{p}(\Sp)}
\end{equation}
for all $F\in T^{p}(\Sp)$ and $p\in[1,2]$, but the reverse inequality does not hold for $p<2$. On the other hand, the reverse inequality does hold for all $p\in[2,\infty)$, while \eqref{eq:vertical} itself fails for $p>2$. 

We include the following results on complex interpolation, duality and atomic decompositions of tent spaces for later use. In the setting of general spaces of homogeneous type, the following two lemmas can be found in \cite{Amenta14}.

\begin{lemma}\label{lem:tentint}
Let $p_{1},p_{2}\in[1,\infty]$ and $\theta\in[0,1]$, and let $p\in[1,\infty]$ be such that $\frac{1}{p}=\frac{1-\theta}{p_{1}}+\frac{\theta}{p_{2}}$. Then 
\[
[T^{p_{1}}(\Sp),T^{p_{2}}(\Sp)]_{\theta}=T^{p}(\Sp),
\]
with equivalent norms.
\end{lemma}

\begin{lemma}\label{lem:tentdual}
Let $p\in[1,\infty)$. Then $T^{p'}(\Sp)=(T^{p}(\Sp))^{*}$, with equivalent norms, where the duality pairing is given by
\[
(F,G)\mapsto\int_{\Spp}F(x,\w,\sigma)\overline{G(x,\w,\sigma)}\ud x\ud \w\frac{\ud \sigma}{\sigma}
\]
for $F\in T^{p}(\Sp)$ and $G\in T^{p'}(\Sp)$.  
\end{lemma}

A measurable function $A:\Spp\to\C$ is a \emph{$T^{1}(\Sp)$-atom} if there exists an open ball $B\subseteq\Sp$ such that $\supp(A)\subseteq T(B)$ and $\|A\|_{L^{2}(\Spp)}\leq V(B)^{-\frac{1}{2}}$.
The collection of $T^{1}(\Sp)$-atoms is a uniformly bounded subset of $T^{1}(\Sp)$.

\begin{lemma}\label{lem:atomictent}
There exists a $C>0$ such that the following holds. For all $F\in T^{1}(\Sp)$, there exists a sequence $(A_{k})_{k=1}^{\infty}$ of $T^{1}(\Sp)$-atoms and a sequence $(\alpha_{k})_{k=1}^{\infty}\subseteq \C$ such that $F=\sum_{k=1}^{\infty}\alpha_{k}A_{k}$ and
\[
\frac{1}{C}\|F\|_{T^{1}(\Sp)}\leq\sum_{k=1}^{\infty}|\alpha_{k}|\leq C\|F\|_{T^{1}(\Sp)}.
\]
If $F\in T^{1}(\Sp)\cap T^{p}(\Sp)$ for $p\in(1,\infty)$, then $\sum_{k=1}^{\infty}\alpha_{k}A_{k}$ also converges to $F$ in $T^{p}(\Sp)$. 

Let $R\in\La(T^{2}(\Sp))$ be such that $\|R(A)\|_{T^{1}(\Sp)}\leq C'$ for all $T^{1}(\Sp)$-atoms $A$ and a $C'\geq0$ independent of $A$. Then $R$ has a unique bounded extension from $T^{1}(\Sp)\cap T^{2}(\Sp)$ to $T^{1}(\Sp)$.
\end{lemma}
\begin{proof}
The atomic decomposition itself, on general spaces of homogeneous type, can be found in \cite{Russ07}. The second statement follows from an application of the dominated convergence theorem in the proof (see e.g.~\cite[Theorem 3.6]{CaMcMo13}).

For the final statement, it suffices to show that $\|R(F)\|_{T^{1}(\Sp)}\lesssim \|F\|_{T^{1}(\Sp)}$ for all $F\in T^{1}(\Sp)\cap T^{2}(\Sp)$, since $T^{1}(\Sp)\cap T^{2}(\Sp)$ is dense in $T^{1}(\Sp)$. Let $(A_{k})_{k\in\N}$ and $(\alpha_{k})_{k\in\N}\subseteq\C$ be as in the first part of the lemma. Then $\sum_{k=1}^{\infty}\alpha_{k}R(A_{k})$ converges in $T^{2}(\Sp)$ to $R(F)$, by the second part of the lemma. Also, $(\sum_{k=1}^{m}\alpha_{k}R(A_{k}))_{m=1}^{\infty}$ is a Cauchy sequence in $T^{1}(\Sp)$. Hence $\sum_{k=1}^{\infty}\alpha_{k}R( A_{k})$ converges in $T^{1}(\Sp)$ as well, and $\|R(F)\|_{T^{1}(\Sp)}\lesssim \|F\|_{T^{1}(\Sp)}$. 
\end{proof}

\begin{remark}\label{rem:localatoms}
In Lemma \ref{lem:atomictent}, if $F(x,\w,\sigma)=0$ for all $(x,\w,\sigma)\in\Spp$ with $\sigma\geq1$, then each $A_{k}$ can be chosen to be associated with a ball of radius at most $2$, cf.~\cite[Theorem 3.6]{CaMcMo13}.
\end{remark}

We now introduce classes of test functions on $\Sp$ and $\Spp$, and the corresponding distributions. Let $\Da(\Sp)$ consist of those $f\in L^{\infty}(\Sp)$ such that $[(x,\w)\mapsto (1+|x|)^{N}f(x,\w)]\in L^{\infty}(\Sp)$ for all $N\geq0$, endowed with the topology generated by the corresponding weighted $L^{\infty}$-norms. Similarly, $\Da(\Spp)$ consists of all $F\in L^{\infty}(\Spp)$ such that 
\[
(x,\w,\sigma)\mapsto (1+|x|+\Upsilon(\sigma)^{-1})^{N}F(x,\w,\sigma)
\]
is an element of $L^{\infty}(\Spp)$ for all $N\geq0$, with the topology generated by the corresponding weighted $L^{\infty}$-norms. Here $\Upsilon$ is as in \eqref{eq:upsilon}. Let $\Da'(\Sp)$ be the space of continuous linear $g:\Da(\Sp)\to \C$, endowed with the topology induced by $\Da(\Sp)$, and similarly for $\Da'(\Spp)$. To avoid confusion, we denote the duality between $f\in\Da(\Sp)$ and $g\in \Da'(\Sp)$ by $\lb f,g\rb_{\Sp}$, and the duality between $F\in\Da(\Spp)$ and $G\in\Da'(\Spp)$ by $\lb F,G\rb_{\Spp}$. If $G\in L^{1}_{\loc}(\Spp)$ is such that 
\[
F\mapsto \int_{\Spp}F(x,\w,\sigma)\overline{G(x,\w,\sigma)}\ud x\ud\w\frac{\ud\sigma}{\sigma}
\]
defines an element of $\Da'(\Spp)$, then we write $G\in\Da'(\Spp)$. Note in particular that 
\[
L^{1}\big(\Spp,(1+|x|+\Upsilon(\sigma)^{-1})^{-N}\ud x\ud\w\frac{\ud\sigma}{\sigma}\big)\subseteq\Da'(\Spp)
\]
for all $N\geq0$.

The following lemma allows one to apply this distribution theory to tent spaces.

\begin{lemma}\label{lem:distributions}
For all $p\in[1,\infty]$ one has
\[
\Da(\Spp)\subseteq T^{p}(\Sp)\subseteq\Da'(\Spp)
\]
continuously, where the first embedding is dense if $p<\infty$.
\end{lemma}
\begin{proof}
For $p<\infty$, $F\in\Da(\Spp)$ and $\alpha>0$ large enough one has, by Lemmas \ref{lem:metric} and \ref{lem:doubling},
\begin{align*}
&\int_{\Sp}\Big(\int_{0}^{\infty}\fint_{B_{\sqrt{\sigma}}(x,\w)}|F(y,\nu,\sigma)|^{2}\ud y\ud\nu\frac{\ud\sigma}{\sigma}\Big)^{\frac{p}{2}}\ud x\ud\w\\
&\lesssim \!\int_{\Sp}\!\Big(\!\int_{0}^{\infty}\!\int_{B_{\sqrt{\sigma}}(x,\w)}\!\frac{(1+|y|+\sqrt{\sigma})^{2\alpha/ p}}{\min(\sigma^{n},\sigma^{n/2})}|F(y,\nu,\sigma)|^{2}\ud y\ud\nu\frac{\ud\sigma}{\sigma}\!\Big)^{\frac{p}{2}}\!\frac{\ud x\ud\w}{(1+|x|)^{\alpha}}<\infty.
\end{align*}
Hence $\Da(\Spp)\subseteq T^{p}(\Sp)$. Moreover, the collection of $F\in L^{2}(\Spp)$ with compact support in $\Spp$ is dense in $T^{p}(\Sp)$ for $p<\infty$ (see \cite[Proposition 3.5]{Amenta14}). Since each such function can be approximated in $L^{2}(\Spp)$ by compactly supported simple functions, and the $T^{p}(\Sp)$-norm is equivalent to the $L^{2}(\Spp)$-norm on compact subsets of $\Spp$ (see \cite[Lemma 3.3]{Amenta14}), it follows that $\Da(\Spp)\subseteq T^{p}(\Sp)$ is dense for $p<\infty$. It now also follows from Lemma \ref{lem:tentdual} that $T^{p}(\Sp)\subseteq\Da'(\Spp)$ for all $p\in(1,\infty]$. 

Next, for $G\in T^{1}(\Sp)$ and $N$ large enough one has
\begin{align*}
&\int_{\Spp}(1+|x|+\Upsilon(\sigma)^{-1})^{-N}|G(x,\w,\sigma)|\ud x\ud\w\frac{\ud\sigma}{\sigma}\\
&=\int_{\Spp}\!(1+|x|+\Upsilon(\sigma)^{-1})^{-N}|G(x,\w,\sigma)|\fint_{B_{\sqrt{\sigma}}(x,\w)}\ud y\ud\nu\ud x\ud\w\frac{\ud\sigma}{\sigma}\\
&=\int_{\Sp}\int_{0}^{\infty}\fint_{B_{\sqrt{\sigma}}(y,\nu)}(1+|x|+\Upsilon(\sigma)^{-1})^{-N}|G(x,\w,\sigma)|\ud x\ud\w\frac{\ud\sigma}{\sigma}\ud y\ud\nu\\
&\lesssim \int_{\Sp}\Big(\int_{0}^{\infty}\fint_{B_{\sqrt{\sigma}}(y,\nu)}|G(x,\w,\sigma)|^{2}\ud x\ud \w\frac{\ud\sigma}{\sigma}\Big)^{1/2}\ud y\ud\nu=\|G\|_{T^{1}(\Sp)}.
\end{align*}
Hence 
\begin{equation}\label{eq:weightedL1}
T^{1}(\Sp)\subseteq L^{1}\big(\Spp,(1+|x|+\Upsilon(\sigma)^{-1})^{-N}\ud x\ud\w\frac{\ud\sigma}{\sigma}\big)\subseteq\Da'(\Spp).
\end{equation}
Finally, by Lemma \ref{lem:tentdual} one also obtains that $\Da(\Spp)\subseteq T^{\infty}(\Sp)$. 
\end{proof}

\subsection{Fourier integral operators}\label{subsec:FIOs}

In this section we collect some background material on oscillatory integrals and Fourier integral operators.

We first introduce suitable spaces of symbols. Let $a\in C^\infty(\R^{2n})$ and $\rho\in[\frac{1}{2},1]$. We say that $a$ is a symbol of order $m\in\R$ and type $(\rho, 1-\rho, 1)$ if, for all $\alpha,\beta\in\Z_{+}^{n}$ and $\gamma\in\Z_{+}$, there exists a $C_{\alpha,\beta,\gamma}\geq0$ such that
\begin{equation}
\big|\partial^{\alpha}_{x}\partial_{\eta}^{\beta}\lb\hat{\eta},\nabla_{\eta}\rb^{\gamma}a(x,\eta) \big|\leq C_{\alpha, \beta, \gamma} \lb \eta\rb^{m + (1 - \rho)|\alpha| - \rho|\beta| - \gamma}
\label{symbol-est}\end{equation}
for all $x,\eta\in\Rn$ with $\eta\neq 0$. We denote the space of all symbols of order $m$ and type $(\rho,1-\rho,1)$ by $S^m_{\rho, 1-\rho, 1}(\R^{2n})$. It is a Frech\'et space with the seminorms given by the minimal constants $C_{\alpha,\beta,\gamma}$ in \eqref{symbol-est}. Note that, for $\rho=1$, this is just the classical Kohn--Nirenberg class $S^{m}(\R^{2n})$. For $1/2 \leq \rho < 1$, it is almost the H\"ormander class $S^{m}_{\rho,1-\rho}(\R^{2n})$, but there is extra decay when differentiating in the radial direction in $\eta$. The \emph{cone support} of $a$ is the conic set in $\R^{2n}$ generated by $\supp(a)\setminus o$, and the base of the cone support, i.e.~the projection of the cone support of $a$ onto $\R^{n}\times S^{n}$, is denoted by $\text{base}(a)$.

Also, for $N\in\N$, $S^{m}_{\rho,1-\rho,1}(\R^{2n}\times\R^{N})$ consists of all $a\in C^{\infty}(\R^{2n}\times\R^{N})$ such that,
for all $\alpha,\delta\in\Z_{+}^{n}$, $\beta\in\Z_{+}^{N}$ and $\gamma\in\Z_{+}$, there exists a $C_{\alpha,\beta,\gamma,\delta}\geq0$ such that
\[
\big|\partial^{\alpha}_{x}\partial_{y}^{\delta}\partial_{\theta}^{\beta}\lb\hat{\theta},\nabla_{\theta}\rb^{\gamma}a(x,y,\theta) \big|\leq C_{\alpha, \beta, \gamma,\delta} \lb \theta\rb^{m + (1 - \rho)(|\alpha|+|\delta|) - \rho|\beta| - \gamma}
\]
for all $x,y\in\Rn$ and $\theta\in\R^{N}$ with $\theta\neq 0$. The cone support of such an $a$ is the conic set in $\R^{2n}\times\R^{N}$ generated by $\{(x,y,\theta)\in \supp(a)\mid \theta\neq0\}$, and the base  of the cone support is the projection of the cone support onto $\R^{2n}\times S^{N}$.

We now define the specific Fourier integral operators that we will work with for most of the article. Throughout, for $\veps>0$, we say that a set $\widetilde{U}\subseteq\Sp$ is an \emph{$\veps$-neighborhood} of a set $U\subseteq\Sp$ if $B_{\veps}(x,\w)\subseteq \widetilde{U}$ for all $(x,\w)\in U$.
 
\begin{definition}\label{def:operator}
Let $a\in L^{\infty}(\R^{2n})$ and $\Phi:V\to\R$ be given, where $V\subseteq\R^{2n}$ is an open conic neighborhood of the cone support of $a$. Set
\begin{equation}
Tf(x):=\int_{\Rn}e^{i\Phi(x,\eta)}a(x,\eta)\wh{f}(\eta)\ud\eta
\label{eq:normal-oscint}\end{equation}
for $f\in\Sw(\Rn)$ and $x\in \Rn$. We call $T$ a \emph{normal oscillatory integral operator} of order $m\in\R$ and type $(\rho, 1-\rho, 1)$, for $\rho\in[\frac{1}{2},1]$, with phase function $\Phi$ and symbol $a$, if
\begin{enumerate}
\item\label{it:phase1} $\Phi\in C^{\infty}(V)$ and $(x,\eta)\mapsto \Phi(x,\eta)$ is homogeneous of degree $1$ in the $\eta$-variable;
\item\label{it:phase2} $\sup_{(x,\eta)\in V}|\partial_{x}^{\alpha}\partial_{\eta}^{\beta}\Phi(x,\hat{\eta})|<\infty$ for all $\alpha,\beta\in\Z_{+}^{n}$ with $|\alpha|+|\beta|\geq 2$;
\item\label{it:phase3} $\inf_{(x,\eta)\in V}| \det \partial^2_{x \eta} \Phi (x,\eta)|>0$;
\item\label{it:phase4} The map $(\nabla_{\eta}\Phi(x,\eta),\eta)\mapsto (x,\nabla_{x}\Phi(x,\eta))$ defines a homogeneous canonical transformation $\chi$ with domain $\dom(\chi)=\{(\nabla_{\eta}\Phi(x,\eta),\eta)\mid (x,\eta)\in V\}$, and $\dom(\chi)\cap \Sp$ is an $\veps$-neighborhood of $\{(\nabla_{\eta}\Phi(x,\eta),\eta)\mid (x,\eta)\in \text{base}(a)\}$ for some $\veps>0$;
\item\label{it:symbol} $a\in S^{m}_{\rho,1-\rho,1}(\R^{2n})$. 
\end{enumerate}
\end{definition}

\begin{remark}
For general Fourier integral operators, as defined below, one considers local versions of these assumptions; the point here is that they should hold uniformly on the domain of $\Phi$. This will allow us to obtain results for Fourier integral operators whose Schwartz kernel is not compactly supported. 

The main point of condition \eqref{it:phase4} is that the map $(\nabla_{\eta}\Phi(x,\eta),\eta)\mapsto (x,\nabla_{x}\Phi(x,\eta))$ is well defined and bijective from its domain to its range. If this is true, then it is automatically a homogeneous canonical transformation. And if either $\Sp\subseteq\dom(\chi)$ or $\text{base}(a)$ is compact, then $\dom(\chi)\cap \Sp$ is an $\veps$-neighborhood of $\{(\nabla_{\eta}\Phi(x,\eta),\eta)\mid (x,\eta)\in \text{base}(a)\}$.

For $n\geq3$, the global inverse function theorem (see \cite[Theorems 6.2.4 and 6.2.8]{Krantz-Parks13}) shows that \eqref{it:phase4} in fact follows from \eqref{it:phase1}, \eqref{it:phase2} and \eqref{it:phase3} if $\dom(\Phi)=\Rn\times (\Rn\setminus \{0\})$, and that one then has $\dom(\chi)=\ran(\chi)=\Tp\setminus o$. Here the condition $n\geq3$ ensures that $\Rn\setminus \{0\}$ is simply connected when solving the equation $\xi=\nabla_{x}\Phi(x,\eta)$ for $\eta$.
\end{remark}

As noted in Section \ref{subsec:metric}, the map $\chi$ in \eqref{it:phase4} induces a contact transformation $\hchi$ on $\Sp$: 
\begin{equation}
\hchi(\nabla_{\eta}\Phi(x,\eta),\eta):=\big(x,\tfrac{\nabla_{x}\Phi(x,\eta)}{|\nabla_{x}\Phi(x,\eta)|}\big)
\label{hatchi}\end{equation}
for $(x,\eta)\in V$ with $|\eta|=1$. We call $\hchi$ the contact transformation induced by $\Phi$. 

\begin{lemma}\label{lem:philip}
Let $\Phi:V\to\R$ satisfy \eqref{it:phase1} -- \eqref{it:phase4} in Definition \ref{def:operator}, let $\hchi$ be the contact transformation induced by $\Phi$, and let $U\subseteq\dom(\hchi)$. Suppose that either $U=\dom(\hchi)=\ran(\hchi)=\Sp$, or $U$ is compact. Then $\hchi:U\to\hchi(U)$ is bi-Lipschitz.
\end{lemma}
\begin{proof} 
By Proposition \ref{prop:Lipschitz}, we only have to show that $\hchi$ and $\hchi^{-1}$ have uniformly bounded first derivatives. We first claim that $|\nabla_{x}\Phi(x, \hat\eta)|$ is uniformly bounded in $(x,\eta) \in V$, both from above and away from zero. Indeed, the homogeneity of $\Phi$ yields 
\[
\nabla_{x}  \Phi(x, \hat\eta)  = \partial^2_{x \eta} \Phi(x,\hat{\eta})\hat{\eta},
\] 
and the claim then follows from the bounds in \eqref{it:phase2} and \eqref{it:phase3} of Definition \ref{def:operator}. Combined with \eqref{it:phase4}, this also shows that $\hchi$ is well-defined and bijective from its domain to its range.

By \eqref{hatchi}, the claim reduces the proof to showing that $\chi$ and $\chi^{-1}$ have uniformly bounded first derivatives on $\Sp$. To this end, we express the derivatives of $\chi$ using the implicit function theorem. Set 
\[
F((y, \eta),(x, \xi)) :=  \begin{pmatrix} y - \nabla_{\eta}\Phi(x, \eta) \\ \xi - \nabla_{x}\Phi(x, \eta) \end{pmatrix}
\]
for $(y,\eta),(x,\xi)\in \Tp$ with $(x,\eta)\in V$. Then $\chi$ is defined implicitly by $(x,\xi) = \chi(y, \eta) \Leftrightarrow F((y, \eta), (x, \xi)) = 0$. Hence
\[
\frac{\partial(x, \xi)}{\partial(y, \eta)} = \begin{pmatrix} \partial_{\eta x}^{2}\Phi & 0 \\  \partial^{2}_{xx}\Phi & -I \end{pmatrix}^{-1} \begin{pmatrix} I & -\partial^{2}_{\eta\eta}\Phi \\ 0 & -\partial^{2}_{x\eta}\Phi \end{pmatrix}= \begin{pmatrix} (\partial_{x\eta}^{2}\Phi)^{-1} & 0 \\ (\partial^{2}_{x\eta}\Phi)^{-1}\partial^{2}_{xx}\Phi & -I \end{pmatrix}  \begin{pmatrix} I & -\partial_{\eta\eta}^{2}\Phi \\ 0 & -\partial_{x\eta}^{2}\Phi \end{pmatrix}.
\]
The first derivatives of $\chi$ are therefore bounded in terms of the second derivatives of $\Phi$ and the lower bound for $\partial_{x\eta}^{2}\Phi$ from \eqref{it:phase3}. The first derivatives of $\chi^{-1}$ are estimated in the exact same way.
\end{proof}

We now collect some basics on the invariant theory of Fourier integral operators, as developed by H\"{o}rmander in \cite{Hormander71}. For most of the results in later sections it suffices to consider the normal oscillatory integral operators defined above, and readers only interested in those operators may skip the rest of this section.

In H\"{o}rmander's theory, there is a class of Fourier integral operators associated with certain conic Lagrangian submanifolds of $T^*(X) \times T^*(Y)$, where $X$ and $Y$ are manifolds. 
In this article we shall only be concerned with the local theory of these operators, and with the case where $\dim X = \dim Y$, so we may assume that $X = Y = \R^n$. For $N\in\N$, a \emph{Langrangian submanifold} of $T^{*}(\R^{N})$ is a submanifold of dimension $N$ on which the natural symplectic form on $T^{*}(\R^{N})$ vanishes. A \emph{homogeneous canonical relation} is a submanifold $\Gamma\subseteq(\Tp\setminus o)\times(\Tp\setminus o)$ which is closed in $T^{*}(\R^{2n})\setminus o$ and is invariant under positive dilations in the fiber variables, such that
\[
\Gamma':=\{(x,\xi,y,-\eta)\mid (x,\xi,y,\eta)\in\Gamma\}
\]
is a Lagrangian submanifold of $T^{*}(\R^{2n})$. One says that $\Gamma$ is a \emph{local canonical graph} if it is locally the graph of a homogeneous canonical transformation. A \emph{local parametrization} of $\Gamma$ near a point $(x_0, \xi_0, y_0, \eta_0) \in \Gamma$ is a smooth function $\Psi:U\to\R$, where $U\subseteq \R^{2n}\times\R^{N}$ is an open conic neighborhood of $(x_{0},y_{0},\theta_{0})$ for some $N\in\N$ and $\theta_{0}\in\R^{N}$, such that
\begin{itemize}
\item $\Psi$ is homogeneous of degree one in the $\theta$-variable;
\item the differentials  $d(\partial_{\theta_j} \Psi)$, $j\in\{1,\ldots,N\}$, are linearly independent near $(x_0, y_0, \theta_0)$;
\item $\nabla_\theta \Psi(x_0, y_0, \theta_0) = 0$, $\nabla_x \Psi(x_0, y_0, \theta_0) = \xi_0$ and $\nabla_y \Psi(x_0, y_0, \theta_0) = \eta_0$;
\item locally near $(x_0, y_0, \theta_0)$, one has
$$
\Gamma' = \{ (x, \nabla_x \Psi(x,y,\theta), y, \nabla_y \Psi(x,y,\theta)) \mid \nabla_\theta \Psi(x,y,\theta) = 0 \}.
$$
\end{itemize}
A \emph{smoothing operator} is an operator $R:\Sw(\Rn)\to\Sw(\Rn)$ with Schwartz kernel in $\mathcal{S}(\R^{2n})$. Let $\rho\in[\frac{1}{2},1]$. Then a \emph{Fourier integral operator} of order $m$ and type $(\rho, 1-\rho, 1)$ associated with a canonical relation $\Gamma$ is an operator $T:\Sw(\Rn)\to\Sw'(\Rn)$ for which there exists a smoothing operator $R$ such that the Schwartz kernel of $T-R$ is a locally finite sum of oscillatory integrals of the form 
\begin{equation}\label{eq:oscillatory}
(x,y)\mapsto (2\pi)^{-(n+N)/2} \int_{\R^N} e^{i\Psi(x, y, \theta)} a(x, y, \theta) \ud\theta,
\end{equation}
where $\Psi$ is a local parametrization of $\Gamma$ defined on a neighbourhood of the cone support of $a$, and where $a\in S^{m + (n-N)/2}_{\rho, 1-\rho, 1}(\R^{2n}\times\R^{N})$. The integral in \eqref{eq:oscillatory} is a formal expression which need not be absolutely convergent; for a proper definition one can use smooth cutoffs, integrate by parts and take a limit (see e.g.~\cite[Section 0.5]{Sogge17}).

We now explain the link with our normal oscillatory integral operators. First, let $T$ be a normal oscillatory integral operator as in \eqref{eq:normal-oscint}, with associated homogeneous canonical transformation $\chi$. Then the phase function $(x,y,\xi)\mapsto \Psi(x, y, \xi) := \Phi(x, \xi) - y \cdot \xi$, associated with the Schwartz kernel of $T$, parametrizes the relation $\{ \chi(y, \xi),(y, \xi)) \mid (y, \xi) \in \dom(\chi) \}$. This relation, the graph of $\chi$, is a homogeneous canonical relation. Hence each normal oscillatory integral operator is a Fourier integral operator. 

The following proposition provides a partial converse.

\begin{proposition}\label{prop:FIOnormal}
Let $\rho \in (1/2, 1]$, and let $T$ be a Fourier integral operator of order $0$ and type $(\rho, 1-\rho,1)$ associated with a local canonical graph, with compactly supported Schwartz kernel. Then there exist a smoothing operator $R$, an $m\in\N$ and sequences $(T_{j,1})_{j=1}^{m}$ and $(T_{j,2})_{j=1}^{m}$ of normal oscillatory integral operators of order $0$ and type $(\rho,1-\rho,1)$ such that 
\[
T=\sum_{j=1}^{m}T_{j,1}T_{j,2}+R.
\]
Moreover, for all $j\in\{1,\ldots,m\}$ and $k\in\{1,2\}$, the Schwartz kernel of $T_{j,k}$ is compactly supported, and the symbol $a_{j,k}$ of $T_{j,k}$ is such that $a_{j,k}(x,\eta)=0$ for all $(x,\eta)\in\R^{2n}$ with $|\eta|<1$. 
\end{proposition}
\begin{proof}
By assumption, there exists a smoothing operator $R_{1}$ such that the Schwartz kernel of $T-R_{1}$ is a finite sum of oscillatory integral expressions as in \eqref{eq:oscillatory}, where each amplitude $a$ is an $S^{(n-N)/2}_{\rho,1-\rho,1}(\R^{2n}\times\R^{N})$ symbol such that the cone support of $a$ has a compact base. 
So without loss generality we may assume that the Schwartz kernel of $T$ is as in \eqref{eq:oscillatory}, where the $(x,y)$-support of $a$ is compact. Let $\Gamma$ denote the canonical relation of $T$. We may also assume that the projection $\hat{\Gamma}$ of $\Gamma$ to $\Sp\times\Sp$ is compact, and that $\Gamma$ is the graph of a homogeneous canonical transformation $\chi$.

We first find a change of $y$ coordinates that will allow us to construct a  phase function $\Phi$ satisfying properties \eqref{it:phase1} -- \eqref{it:phase4}.  
Fix a point $(x_0, \xi_0, y_0, \eta_0) \in \Gamma$. Then  
\[
L := \{ (y, \eta) \in \Tp \mid \exists \xi \text{ s.t.~}(x_0, \xi, y, \eta) \in \Gamma \}
\]
is a Lagrangian submanifold of $\Tp$, being the image under $\chi^{-1}$ of the Lagrangian submanifold $T^*_{x_0}(\R^n)\subseteq \Tp$. Since $q_0 :=(y_{0},\eta_{0})\in L$, it follows from \cite[Theorem 3.1.3]{Hormander71} that one can make a change of $y$ coordinates so that the dual $\eta$ coordinate on $\Tp$ is a coordinate on $L$, locally near $q_0$. This coordinate change is itself a normal oscillatory integral operator. In fact, if the coordinate change is $y \mapsto \tilde y = F(y)$, defined on a small neighbourhood $V$ of a compact set $K$, then it is implemented by the operator with Schwartz kernel
$$
(\tilde{y},y)\mapsto (2\pi)^{-n} \int e^{i(F^{-1}(\tilde y) - y) \cdot \eta} \phi(F^{-1}(\tilde y)) \, d\eta
$$
where $\phi\in C^{\infty}_{c}(\Rn)$ is supported on $F(V)$ and identically equal to one on $F(K)$, and we interpret $\phi(F^{-1}(x)) = 0$ if $x \notin F(V)$. This is a normal oscillatory integral operator, and we can cut off the symbol in a neighbourhood of $\eta = 0$ as this only changes the kernel by a smoothing term. This gives us the $T_{j,2}$ factor in the statement of the proposition.

Abusing notation somewhat, we continue to use $y$ (instead of $\tilde y$) to denote these new coordinates, and $\eta$ for the new dual coordinates.  Therefore, both $\xi$ and $\eta$ furnish coordinates on $L$, implying that  the Jacobian $\partial \eta/ \partial \xi$ is nonzero on $L$, which may be identified with $\Gamma \cap \{ x = x_0\}$. By continuity, the Jacobian $\partial \eta/ \partial \xi$ is nonzero on $L = \Gamma \cap \{ x = x_1\}$ for $x_1$ in a neighbourhood of $x_0$. By the assumption that $\Gamma$ is a local canonical graph, $(x, \xi)$ form local coordinates on $\Gamma$ near $(x_0, \xi_0, y_0, \eta_0)$. Because of the Jacobian property just shown,  we can use $(x, \eta)$ instead of $(x, \xi)$ as local coordinates on $\Gamma$. Thus we can write the other coordinates locally as functions of $(x, \eta)$: that is, $y = Y(x, \eta)$ and $\xi = \Xi(x, \eta)$. Then, as shown in 
\cite[Theorem 21.2.18]{Hormander07}, the function
$$
\Psi(x, y, \eta) = (Y(x, \eta) - y) \cdot \eta = \Phi(x, \eta) - y \cdot \eta, \quad \Phi(x, \eta) := Y(x, \eta) \cdot \eta
$$
locally parametrizes $\Gamma$. To verify this, one uses identities derived from the Lagrangian nature of $\Gamma'$, which implies that 
\begin{equation}
d\eta \cdot dY = d\Xi \cdot dx.
\label{symp}\end{equation}

The results in \cite{Hormander71} about invariance under change of phase function show that, at least for $\rho > 1/2$, the original FIO can be written in terms of this new phase function. That is, its kernel takes the form 
\begin{equation}\label{eq:oscillatory2}
(x,y)\mapsto (2\pi)^{-n} \int_{\R^N} e^{i(\Phi(x, \eta) - y \cdot \eta)} \tilde a(x, y, \eta) \ud\eta
\end{equation}
with $\tilde a \in S^{0}_{\rho,1-\rho,1}(\R^{2n}\times\R^{n})$. This is not quite of the right form, as the amplitude depends on both $x$ and $y$. However, as is well known, up to a smooth kernel $R_{2}$, this oscillatory integral is equal to a similar expression with $\tilde a$ replaced by a function of $x$ and $\xi$ alone (and compactly supported in $x$). This is done essentially by expanding $\tilde a$ in a Taylor series about the stationary point $y = d_\eta \Phi(x, \eta)$, expressing $y - d_\eta \Phi$ as the $\eta$-derivative of the exponential, and then integrating by parts. The kernel $(x,y)\mapsto R_{2}(x,y)$ is no longer compactly supported in $y$. However, integrating by parts in $\eta$ shows that $R_{2}$, together with all its partial derivatives in $x$ and $y$, is rapidly decreasing as $|x-y|$ tends to infinity. Together with the compact support in $x$, this shows that $R_{2}$ has a kernel that is a Schwartz function, and therefore $R_{2}$ is a smoothing operator. In a similar way, up to another smoothing operator we can assume that the amplitude $\tilde a$ is supported away from $\eta = 0$. The operator so obtained is the $T_{j,1}$ in the statement of the proposition. 

The phase function $\Phi(x, \eta) - y \cdot \eta$ now has the correct form as in \eqref{eq:normal-oscint}. Moreover, $\Phi$ has the properties \eqref{it:phase1} -- \eqref{it:phase4}.  These are all clear except for \eqref{it:phase3}. To see this property, we note that writing out the two-form identity \eqref{symp} in local coordinates $(x, \eta)$  shows, by equating the coefficient of $d\eta_i d\eta_j$ on both sides, 
$$
\frac{\partial Y_i}{\partial \eta_j} = \frac{\partial Y_j}{\partial \eta_i}.   
$$
We then find that 
$$
\sum_i \eta_i \frac{\partial Y_i}{\partial \eta_j} = \sum_i \eta_i \frac{\partial Y_j}{\partial \eta_i} = 0
$$
using the homogeneity of degree zero of the functions $Y_j$. It follows that
$$
d_{\eta_j} \Phi = Y_j + \sum_i \eta_i \frac{\partial Y_i}{\partial \eta_j} = Y_j, \text{ and  therefore } d_{x_i \eta_j} \Phi = d_{x_i} Y_j.
$$
Therefore \eqref{it:phase3} follows from the observation that the Jacobian $\partial Y/ \partial x$ is nonzero,  which is due the fact that both $(x, \eta)$ and $(y, \eta)$ furnish coordinates locally on $\Gamma$, together with the assumed compactness of $\hat \Gamma$. 
\end{proof}

\begin{remark}
In \cite{Hormander71}, H\"ormander only considered symbols of type $(\rho, 1 - \rho)$, for $\rho > 1/2$. We have adjusted the definition in the obvious way to treat
our anisotropic symbols of type $(\rho, 1- \rho, 1)$, and it is straightforward to adapt the arguments of H\"{o}rmander's paper to this class for $\rho > 1/2$. From the point of view of wave packet transforms, as in Section~\ref{sec:offsingFIO}, the case $\rho = 1/2$ is very natural to consider. However, it is not clear to the authors whether the statements in \cite{Hormander71} are also valid for symbols of type $(\frac{1}{2},\frac{1}{2},1)$. As this point is not relevant to the rest of this article, we do not pursue this question further. For most readers, the classical case where $\rho = 1$ will be sufficient. 

Also, H\"ormander defined FIOs acting on half-densities, which facilitates defining an invariant
principal symbol. We ignore the half-density factors here. 
\end{remark}

\section{Off-singularity bounds}\label{sec:offsing}

Throughout this article we work with operator families which act on functions on the cosphere bundle and whose kernels decay in a suitable sense off singular sets. In this section we introduce this off-singularity decay and we study some of its basic properties, including the boundedness on tent spaces of associated operators.

\subsection{Definitions and basic properties}\label{subsec:defoffsing}

The notion of off-singularity decay is defined using maps $\hchi$ on $\Sp$. In the rest of this article $\hchi$ generally arises by projection from a map $\chi$ on $\Tp$, but for the results in this section this is not relevant. Recall from \eqref{eq:upsilon} the definition of $\Upsilon$.

\begin{definition}\label{def:pointwise offsing} 
Let $(T_{\sigma,\tau})_{\sigma,\tau>0}$ be a collection of operators such that each $T_{\sigma,\tau}:\Da(\Sp)\to\Da'(\Sp)$ has kernel $K_{\sigma,\tau}:\Sp\times\Sp\to\C$, and let $U\subseteq\Sp$, $\hchi:U\to\Sp$ and $N,C\geq0$ be given. Suppose that 
\[
((x,\w,\sigma),(y,\nu,\tau))\mapsto K_{\sigma,\tau}((x,\w),(y,\nu))
\]
is measurable on $\Spp\times\Spp$, and that $K_{\sigma,\tau}((x,\w),(y,\nu))=0$ if $(x,\w)\notin\hchi(U)$ or $(y,\nu)\notin U$. We say that $(T_{\sigma,\tau})_{\sigma,\tau>0}$ satisfies \emph{off-singularity bounds of type $(\hchi,N,C)$} if for all $\sigma,\tau>0$, $(x,\w)\in\hchi(U)$ and $(y,\nu)\in U$,
\[
|K_{\sigma,\tau}((x,\w),(y,\nu))|\leq C\rho^{-n}\Upsilon(\tfrac{\sigma}{\tau})^{N}(1+\rho^{-1}d((x,\w),\hchi(y,\nu))^{2})^{-N},
\]
where $\rho:=\min(\sigma,\tau)$. We denote by $OS(\hchi,N,C)$ the class of operator families satisfying off-singularity bounds of type $(\hchi,N,C)$.
\end{definition}

For the boundedness results in the next subsection, it is more convenient to work with a notion of $L^{2}$ off-singularity decay, similar to $(L^{2},L^{2})$ off-diagonal decay (see e.g.~\cite{Auscher-Martell07}).

\begin{proposition}\label{prop:offsing}
There exists an $M=M(n)\geq0$ such that the following holds. Let $\hchi:U\to\Sp$ be injective and such that $\hchi^{-1}$ is Lipschitz with constant $C_{\hchi^{-1}}\geq0$, where $U\subseteq\Sp$. Let $N>M$, $C\geq0$, and $(T_{\sigma,\tau})_{\sigma,\tau>0}\in OS(\hchi,N,C)$. Then there exists a $C'=C'(n,N,M,C)$ such that, for all measurable $E\subseteq \hchi(U)$ and $F\subseteq U$ and all $\sigma>0$, one has
\[
\|\ind_{E}T_{\sigma,\tau}\ind_{F}\|_{\La(L^{2}(\Sp))}\leq C'\Upsilon(\tfrac{\sigma}{\tau})^{N}(1+\min(\sigma,\tau)^{-1}d(E,\hchi(F))^{2})^{-(N-M)}.
\]
Moreover, 
\[
(T_{\sigma,\tau}^{*})_{\sigma,\tau>0}\in OS(\hchi^{-1},N, C\max(C_{\hchi^{-1}}^{2N},1)).
\]
\end{proposition}
\begin{proof}
First note that, since $(\Sp,d)$ is a space of homogeneous type, by \cite[Lemma III.1.1]{Coifman-Weiss71} there exists an $L=L(n)\in\N$ such that, for all $r>0$ and $j\in\N$, each ball $B\subseteq\Sp$ with radius $r$ contains at most $L^{j}$ points $((x_{k},\w_{k}))_{k=1}^{L^{j}}$ such that $d((x_{k},\w_{k}),(x_{l},\w_{l}))>r/2^{j}$ for $k\neq l$. Now let $M\in\N$ be such that $M>n+\log_{2}(L)$, fix $\sigma,\tau>0$ and set $\rho:=\min(\sigma,\tau)$, and let $((x_{k},\w_{k}))_{k\in \N}\subseteq\Sp$ be such that $d((x_{k},\w_{k}),(x_{l},\w_{l}))>\sqrt{\rho}/2$ for $k\neq l$ and such that $\Sp=\cup_{k=1}^{\infty}B_{\sqrt{\rho}}(x_{k},\w_{k})$. For $k\in \N$, set $B_{k}:=B_{\sqrt{\rho}}(x_{k},\w_{k})\cap \hchi(U)$, $C_{1,k}:=\hchi^{-1}(4B_{k})$, and $C_{j,k}:=\hchi^{-1}(2^{j+1}B_{k}\setminus 2^{j}B_{k})$ for $j\geq2$. Using Schur's test, the assumption that $\hchi^{-1}$ is Lipschitz, and Lemma \ref{lem:doubling}, one has
\begin{align*}
&\|\ind_{B_{k}\cap E}T_{\sigma,\tau}\ind_{C_{j,k}\cap F}\|_{\La(L^{2}(\Sp))}\\
&\leq C\Upsilon(\tfrac{\sigma}{\tau})^{N}\frac{\sqrt{V(B_{k})V(C_{j,k})}}{\rho^{n}(1+\rho^{-1}d(B_{k},2^{j+1}B_{k}\setminus 2^{j}B_{k})^{2})^{M}(1+\rho^{-1}d(E,\hchi(F))^{2})^{N-M}}\\
&\lesssim 2^{j(n-2M)}\Upsilon(\tfrac{\sigma}{\tau})^{N}(1+\rho^{-1}d(E,\hchi(F))^{2})^{-(N-M)}
\end{align*}
for $k\in\N$ and $j\geq2$, and similarly 
\[
\|\ind_{B_{k}\cap E}T_{\sigma,\tau}\ind_{C_{1,k}\cap F}\|_{\La(L^{2}(\Sp))}\lesssim 2^{n-2M}\Upsilon(\tfrac{\sigma}{\tau})^{N}(1+\rho^{-1}d(E,\hchi(F))^{2})^{-(N-M)}.
\]
Let $f\in L^{2}(\Sp)$. Then
\begin{align*}
&\Upsilon(\tfrac{\sigma}{\tau})^{-2N}(1+\rho^{-1}d(E,\hchi(F))^{2})^{2(N-M)}\|\ind_{E}T_{\sigma,\tau}\ind_{F}f\|_{L^{2}(\Sp)}^{2}\\
&\leq \Upsilon(\tfrac{\sigma}{\tau})^{-2N}(1+\rho^{-1}d(E,\hchi(F))^{2})^{2(N-M)}\sum_{k=1}^{\infty}\|\ind_{B_{k}\cap E}T_{\sigma,\tau}\ind_{F}f\|_{L^{2}(\Sp)}^{2}\\
&\leq\sum_{k=1}^{\infty}\Big(\Upsilon(\tfrac{\sigma}{\tau})^{-N}(1+\rho^{-1}d(E,\hchi(F))^{2})^{N-M}\sum_{j=1}^{\infty}\|\ind_{B_{k}\cap E}T_{\sigma,\tau}\ind_{C_{j,k}\cap F}f\|_{L^{2}(\Sp)}\Big)^{2}\\
&\lesssim \sum_{k=1}^{\infty}\Big(\sum_{j=1}^{\infty}2^{j(n-2M)}\|\ind_{C_{j,k}\cap F}f\|_{L^{2}(\Sp)}\Big)^{2}\\
&\leq \sum_{k=1}^{\infty}\Big(\sum_{j=1}^{\infty}2^{j(n-2M)}\Big)\Big(\sum_{j=1}^{\infty}2^{j(n-2M)}\|\ind_{C_{j,k}\cap F}f\|_{L^{2}(\Sp)}^{2}\Big)\\
&\lesssim \sum_{j=1}^{\infty}2^{j(n-2M)}\sum_{k=1}^{\infty}\int_{\Sp}\ind_{C_{j,k}\cap F}(x,\w)|f(x,\w)|^{2}\ud x\ud\w\\
&\leq \sum_{j=1}^{\infty}2^{j(n-2M)}L^{j+2}\|\ind_{F}f\|_{L^{2}(\Sp)}^{2}\lesssim \|\ind_{F}f\|_{L^{2}(\Sp)}^{2}.
\end{align*}
This proves the first statement. The second statement follows by observing that, if $K_{\sigma,\tau}:\Sp\times\Sp\to\C$ is the kernel of $T_{\sigma,\tau}$, then the kernel $\wt{K}_{\sigma,\tau}$ of $T_{\sigma,\tau}^{*}$ is given by 
\[
\wt{K}_{\sigma,\tau}((x,\w),(y,\nu))=K_{\sigma,\tau}((y,\nu),(x,\w))
\]
for $(x,\w),(y,\nu)\in\Sp$.
\end{proof}

We will also consider residual families of operators, defined as follows.

\begin{definition}\label{def:residual}
Let $(T_{\sigma,\tau})_{\sigma,\tau>0}$ be a collection of operators such that each $T_{\sigma,\tau}:\Da(\Sp)\to\Da'(\Sp)$ has kernel $K_{\sigma,\tau}:\Sp\times\Sp\to\C$. Suppose that 
\[
((x,\w,\sigma),(y,\nu,\tau))\mapsto K_{\sigma,\tau}((x,\w),(y,\nu))
\]
is measurable on $\Spp\times\Spp$. We say that $(T_{\sigma,\tau})_{\sigma,\tau>0}$ is \emph{residual} if for each $N\geq0$ there exists a $C_{N}\geq0$ such that, for all $\sigma,\tau>0$ and $(x,\w),(y,\nu)\in\Sp$,
\begin{equation}\label{eq:residual}
|K_{\sigma,\tau}((x,\w),(y,\nu))|\leq C_{N}(1+|x|+|y|+\Upsilon(\sigma)^{-1}+\Upsilon(\tau)^{-1})^{-N}.
\end{equation}
We denote by $\mathcal{R}$ the class of residual families $(T_{\sigma,\tau})_{\sigma,\tau>0}$.
\end{definition}

The following lemma is a version of Proposition \ref{prop:offsing} for residual families. 

\begin{lemma}\label{lem:residualadjoint}
Let $(T_{\sigma,\tau})_{\sigma,\tau>0}\in \Ra$. Then $\sup_{\sigma,\tau>0}\|T_{\sigma,\tau}\|_{L^{2}(\Sp)}<\infty$ and $(T_{\sigma,\tau}^{*})_{\sigma,\tau>0}\in\Ra$.
\end{lemma}
\begin{proof}
The first statement follows from Schur's test, and the second statement from the fact that \eqref{eq:residual} is symmetric in $(x,\w)$ and $(y,\nu)$.
\end{proof}

\subsection{Boundedness on tent spaces}

In this subsection we obtain boundedness results on tent spaces for families of operators satisfying off-singularity bounds, and for residual families. The relevant operators are given by
\begin{equation}\label{eq:defR}
RF(x,\w,\sigma):=\int_{0}^{\infty}T_{\sigma,\tau}F(\cdot,\cdot,\tau)(x,\w)\frac{\ud\tau}{\tau}\quad((x,\w,\sigma)\in\Spp)
\end{equation}
for suitable operator families $(T_{\sigma,\tau})_{\sigma,\tau>0}$ and functions $F$ on $\Spp$.

\begin{remark}\label{rem:p=infty}
In this section, and in later sections as well, we consider the boundedness of operators $R$ as in \eqref{eq:defR} on the tent spaces $T^{p}(\Sp)$ from Definition \ref{def:tentspace}, for all $p\in[1,\infty]$. The relevant operators are initially defined on $T^{2}(\Sp)$, and $R:T^{2}(\Sp)\cap T^{p}(\Sp)\to T^{p}(\Sp)$ is bounded for all $p\in[1,\infty]$. Since $T^{2}(\Sp)\cap T^{\infty}(\Sp)\subseteq T^{\infty}(\Sp)$ is not dense for $p<\infty$, a bounded extension of $R$ to $T^{\infty}(\Sp)$ is not unique. Throughout, the extension that we consider is given by the adjoint action $\lb RF,G\rb_{\Spp}=\lb F,R^{*}G\rb_{\Spp}$ for $F\in T^{\infty}(\Sp)$ and $G\in T^{1}(\Sp)$.
\end{remark}

We first consider the simpler case of the boundedness of residual families.

\begin{proposition}\label{prop:residualbounded}
Let $(T_{\sigma,\tau})_{\sigma,\tau>0}\in\mathcal{R}$. For $N\geq0$ and $F\in L^{1}\big(\Spp,(1+|x|+\Upsilon(\sigma)^{-1})^{-N}\ud x\ud\w\frac{\ud\sigma}{\sigma}\big)$, let $RF$ be as in \eqref{eq:defR}. Then $RF\in \Da(\Spp)$. In particular, $R\in \La(T^{p}(\Sp))$ for all $p\in[1,\infty]$.
\end{proposition}
\begin{proof}
For $\sigma,\tau>0$, let $K_{\sigma,\tau}$ be the kernel of $T_{\sigma,\tau}$. Then for all $F\in L^{1}\big(\Spp,(1+|x|+\Upsilon(\sigma)^{-1})^{-N}\ud x\ud\w\frac{\ud\sigma}{\sigma}\big)$, $M\geq 0$ and $(x,\w,\sigma)\in\Spp$ one has
\begin{align*}
&(1+|x|+\Upsilon(\sigma)^{-1})^{M}|RF(x,\w,\sigma)|\\
&\leq \int_{\Spp}(1+|x|+\Upsilon(\sigma)^{-1})^{M}|K_{\sigma,\tau}((x,\w),(y,\nu))F(y,\nu,\tau)|\ud y\ud\nu\frac{\ud\tau}{\tau}\\
&\lesssim \int_{\Spp}(1+|y|+\Upsilon(\tau)^{-1})^{-N}|F(y,\nu,\tau)|\ud y\ud\nu\frac{\ud\tau}{\tau}<\infty.
\end{align*}
This proves the first statement. For the second statement, recall from Lemma \ref{lem:distributions} and \eqref{eq:weightedL1} that 
\[
\Da(\Spp)\subseteq T^{1}(\Sp)\subseteq L^{1}\big(\Spp,(1+|x|+\Upsilon(\sigma)^{-1})^{-M}\ud x\ud\w\frac{\ud\sigma}{\sigma}\big)
\]
for $M\geq0$ large enough. Hence $R\in\La(T^{1}(\Sp))$, by what we have already shown. Next, note that 
\[
R^{*}F(x,\w,\sigma)=\int_{0}^{\infty}T_{\tau,\sigma}^{*}F(\cdot,\cdot,\tau)(x,\w)\frac{\ud\tau}{\tau}
\]
for $F\in L^{1}\big(\Spp,(1+|x|+\Upsilon(\sigma)^{-1})^{-N}\ud x\ud\w\frac{\ud\sigma}{\sigma}\big)$. By Lemma \ref{lem:residualadjoint} and by what we have already shown, one has $R^{*}\in \La(T^{1}(\Sp))$. Now Lemma \ref{lem:tentdual} implies that $R: T^{\infty}(\Sp)\to T^{\infty}(\Sp)$ is well-defined and bounded. Finally, Lemma \ref{lem:tentint} concludes the proof.
\end{proof}

Next, we state and prove our main theorem on the connection between off-singularity bounds and boundedness on the tent spaces $T^{p}(\Sp)$, $p\in[1,\infty]$.

\begin{theorem}\label{thm:tentbounded}
There exists an $N>0$ such that the following holds. Let $\hchi:U\to\hchi(U)\subseteq\Sp$ be bi-Lipschitz, where $U\subseteq\Sp$, and let $C\geq0$ and $(T_{\sigma,\tau})_{\sigma,\tau>0}\in OS(\hchi,N,C)$. For $F\in T^{2}(\Spp)$, let $RF$ be as in \eqref{eq:defR}. Then $R\in\La(T^{p}(\Sp))$ for all $p\in[1,\infty]$.  
\end{theorem}

It follows from the proof that $\|R\|_{\La(T^{p}(\Sp))}$ is bounded by a constant which depends only on $n$, $N$, $C$, and on the Lipschitz constants of $\hchi$ and $\hchi^{-1}$. 

\begin{proof}
The proof is similar to that of \cite[Theorem 4.9]{AuMcIRu08}. Throughout, let $M=M(n)$ be as in Proposition \ref{prop:offsing}, and fix $N>M+\frac{n}{2}$ and $\delta\in(0,N-M-\frac{n}{2})$. For $\sigma,\tau>0$, set $\overline{T}_{\sigma,\tau}:=\Upsilon(\frac{\sigma}{\tau})^{-N}T_{\sigma,\tau}$. 

We first consider the case where $p=2$. Let $F\in T^{2}(\Sp)=L^{2}(\Sp)$. By Proposition \ref{prop:offsing}, one has $\sup_{\sigma,\tau>0}\|\overline{T}_{\sigma,\tau}\|_{\La(L^{2}(\Sp))}<\infty$. Hence 
\begin{align*}
&\|R(F)\|_{L^{2}(\Spp)}^{2}=\int_{\Spp}\Big|\int_{0}^{\infty}\Upsilon(\tfrac{\sigma}{\tau})^{N}\overline{T}_{\sigma,\tau}F(\cdot,\cdot,\tau)(x,\w)\frac{\ud\tau}{\tau}\Big|^{2}\ud x\ud\w\frac{\ud \sigma}{\sigma}\\
&\leq \int_{\Spp}\Big(\int_{0}^{\infty}\Upsilon(\tfrac{\sigma}{\tau})^{\delta}\frac{\ud\tau}{\tau}\Big)\Big(\int_{0}^{\infty}\Upsilon(\tfrac{\sigma}{\tau})^{2N-\delta}|\overline{T}_{\sigma,\tau}F(\cdot,\cdot,\tau)(x,\w)|^{2}\frac{\ud\tau}{\tau}\Big)\ud x\ud\w\frac{\ud \sigma}{\sigma}\\
&\lesssim \int_{0}^{\infty}\int_{0}^{\infty}\Upsilon(\tfrac{\sigma}{\tau})^{2N-\delta}\|\overline{T}_{\sigma,\tau}F(\cdot,\cdot,\tau)\|_{L^{2}(\Sp)}^{2}\frac{\ud \tau}{\tau}\frac{\ud \sigma}{\sigma}\\
&\lesssim \int_{0}^{\infty}\Big(\int_{0}^{\infty}\Upsilon(\tfrac{\sigma}{\tau})^{2N-\delta}\frac{\ud \sigma}{\sigma}\Big)\|F(\cdot,\cdot,\tau)\|_{L^{2}(\Sp)}^{2}\frac{\ud \tau}{\tau}\lesssim \|F\|_{L^{2}(\Spp)}^{2},
\end{align*}
for implicit constants independent of $F$. In particular, $RF(x,\w)$ is well defined and finite for almost all $(x,\w)\in\Sp$, and $R\in\La(T^{2}(\Sp))$.

Next, we show that $R\in \La(T^{1}(\Sp))$. By Lemma \ref{lem:atomictent}, it suffices to prove that 
\begin{equation}\label{eq:unifboundatom}
\sup_{A}\|R(A)\|_{T^{1}(\Sp)}<\infty,
\end{equation}
where the supremum is taken over all $T^{1}(\Sp)$-atoms. Let $A$ be such an atom, associated with a ball $B_{A}\subseteq\Sp$. Since $\hchi$ is Lipschitz, there exists a ball $B\subseteq\Sp$ such that $\hchi(B_{A}\cap U)\subseteq B$ and $V(B)\lesssim V(B_{A})$. Set $A_{1}:=\ind_{T(4B)}R(A)$ and, for $k\geq2$, $A_{k}:=\ind_{T(2^{k+1}B)\setminus T(2^{k}B)}R(A)$. We show that there exist $C',\veps>0$, independent of $A$, such that $\|A_{k}\|_{L^{2}(\Spp)}\leq C'2^{-k\veps}V(2^{k+1}B)^{-1/2}$ for all $k\in\N$. Then \eqref{eq:unifboundatom} follows from the fact that $A=\sum_{k=1}^{\infty}A_{k}$, that $\frac{1}{C'}2^{k\veps}A_{k}$ is a $T^{1}(\Sp)$-atom associated with $2^{k+1}B$ for all $k\in\N$, and that the collection of $T^{1}(\Sp)$-atoms is uniformly bounded in $T^{1}(\Sp)$.

For $k=1$, since we have just shown that $R\in\La(L^{2}(\Spp))$, one obtains from Lemma \ref{lem:doubling} that
\[
\|A_{1}\|_{L^{2}(\Spp)}\lesssim \|A\|_{L^{2}(\Spp)}\leq V(B_{A})^{-1/2}\lesssim V(B)^{-1/2}\lesssim V(4B)^{-1/2}.
\] 
Next, let $k\geq 2$, and let $r>0$ be the radius of $B$. Then
\[
A_{k}(x,\w,\sigma)=(\ind_{T(2^{k+1}B)\setminus T(2^{k}B)}R(A))(x,\w,\sigma)=\int_{0}^{r^{2}}\Upsilon(\tfrac{\sigma}{\tau})^{N}\overline{T}_{\sigma,\tau}A(\cdot,\cdot,\tau)(x,\w)\frac{\ud\tau}{\tau}
\]
for $(x,\w,\sigma)\in T(2^{k+1}B)\setminus T(2^{k}B)$, and $A_{k}(x,\w,\sigma)=0$ otherwise. For $(x,\w,\sigma)\in T(2^{k+1}B)\setminus T(2^{k}B)$, the Cauchy-Schwarz inequality yields
\begin{align*}
|A_{k}(x,\w,\sigma)|^{2}&\leq \Big(\int_{0}^{\infty}\Upsilon(\tfrac{\sigma}{\tau})^{\delta}\frac{\ud\tau}{\tau}\Big)\Big(\int_{0}^{r^{2}}\Upsilon(\tfrac{\sigma}{\tau})^{2N-\delta}|\overline{T}_{\sigma,\tau}A(\cdot,\cdot,\tau)|^{2}(x,\w)\frac{\ud\tau}{\tau}\Big)\\
&\lesssim \int_{0}^{r^{2}}\Upsilon(\tfrac{\sigma}{\tau})^{2N-\delta}|\overline{T}_{\sigma,\tau}A(\cdot,\cdot,\tau)|^{2}(x,\w)\frac{\ud\tau}{\tau}.
\end{align*}
Moreover, for $(x,\w,\sigma)\in T(2^{k+1}B)\setminus T(2^{k}B)$ one has $\sigma\leq 2^{2k+2}r^{2}$ and $(x,\w)\in 2^{k+1}B$, and if $\sigma\leq 2^{2k-2}r^{2}$ then in fact $(x,\w)\in 2^{k+1}B\setminus 2^{k-1}B$, where we use that $d(2^{k-1}B,(2^{k}B)^{c})\geq 2^{k-1}r$. Hence
\begin{align}
&\|A_{k}\|_{L^{2}(\Sp)}^{2}=\int_{\Spp}|A_{k}(x,\w,\sigma)|^{2}\ud x\ud \w\frac{\ud \sigma}{\sigma}\nonumber\\
&\lesssim\label{eq:bigest} \int_{0}^{2^{2k-2}r^{2}}\!\int_{0}^{r^{2}}\Upsilon(\tfrac{\sigma}{\tau})^{2N-\delta}\|\ind_{2^{k+1}B\setminus 2^{k-1}B}\overline{T}_{\sigma,\tau}A(\cdot,\cdot,\tau)\|_{L^{2}(\Sp)}^{2}\frac{\ud\tau}{\tau}\frac{\ud \sigma}{\sigma}\\
&+\int_{2^{2k-2}r^{2}}^{2^{2k+2}r^{2}}\!\int_{0}^{r^{2}}\Upsilon(\tfrac{\sigma}{\tau})^{2N-\delta}\|\overline{T}_{\sigma,\tau}A(\cdot,\cdot,\tau)\|_{L^{2}(\Sp))}^{2}\frac{\ud\tau}{\tau}\frac{\ud \sigma}{\sigma}\nonumber.
\end{align}
For the final term in \eqref{eq:bigest}, recall that $\sup_{\sigma,\tau>0}\|\overline{T}_{\sigma,\tau}\|_{\La(L^{2}(\Sp))}<\infty$, by Proposition \ref{prop:offsing}. Hence Lemma \ref{lem:doubling} yields
\begin{align*}
&\int_{2^{2k-2}r^{2}}^{2^{2k+2}r^{2}}\!\int_{0}^{r^{2}}\Upsilon(\tfrac{\sigma}{\tau})^{2N-\delta}\|\overline{T}_{\sigma,\tau}A(\cdot,\cdot,\tau)\|_{L^{2}(\Sp))}^{2}\frac{\ud\tau}{\tau}\frac{\ud \sigma}{\sigma}\\
&\lesssim \int_{0}^{r^{2}}\Big(\frac{\tau}{2^{2k}r^{2}}\Big)^{2N-\delta}\|A(\cdot,\cdot,\tau)\|_{L^{2}(\Sp)}^{2}\frac{\ud\tau}{\tau}\leq 2^{-2k(2N-\delta)}V(B_{A})^{-1}\\
&\lesssim 2^{-2k(2N-\delta)}V(B)^{-1}\lesssim 2^{-2k(2N-\delta-n)}V(2^{k+1}B)^{-1}.
\end{align*}
This suffices, since $2N-\delta-n>0$. 

For the middle term in \eqref{eq:bigest}, note that $2N-\delta>1$. Proposition \ref{prop:offsing} yields, for $\tau\in(0,r^{2})$,
\begin{align*}
&\int_{0}^{2^{2k-2}r^{2}}\Upsilon(\tfrac{\sigma}{\tau})^{2N-\delta}\|\ind_{2^{k+1}B\setminus 2^{k-1}B}\overline{T}_{\sigma,\tau}\ind_{B_{A}\cap U}\|_{\La(L^{2}(\Sp))}^{2}\frac{\ud\sigma}{\sigma}\\
&\lesssim \int_{0}^{2^{2k-2}r^{2}}\Upsilon(\tfrac{\sigma}{\tau})^{2N-\delta}(1+\sigma^{-1}d(2^{k+1}B\setminus 2^{k-1}B,B)^{2})^{-2(N-M)}\frac{\ud\sigma}{\sigma}\\
&\lesssim \int_{0}^{\tau}\frac{\sigma}{\tau}\Big(\frac{\tau}{2^{2k}r^{2}}\Big)^{2(N-M)}\frac{\ud\sigma}{\sigma}+\int_{\tau}^{2^{2k}r^{2}}\Big(\frac{\tau}{\sigma}\Big)^{2(N-M)-\delta}\Big(\frac{\sigma}{2^{2k}r^{2}}\Big)^{2(N-M)-2\delta}\frac{\ud\sigma}{\sigma}\\
&\lesssim \Big(\frac{\tau}{2^{2k}r^{2}}\Big)^{2(N-M)}+\Big(\frac{\tau}{2^{2k}r^{2}}\Big)^{2(N-M)-2\delta}\lesssim 2^{-2k(2N-2M-2\delta)}.
\end{align*}  
Also, recall that the kernel $K_{\sigma,\tau}$ of $\overline{T}_{\sigma,\tau}$ satisfies $K_{\sigma,\tau}((x,\w),(y,\nu))=0$ for $(y,\nu)\notin U$, and that $\supp(A)\subseteq T(B_{A})$. Hence, using Lemma \ref{lem:doubling} again,
\begin{align*}
&\int_{0}^{2^{2k-2}r^{2}}\int_{0}^{r^{2}}\Upsilon(\tfrac{\sigma}{\tau})^{2N-\delta}\|\ind_{2^{k+1}B\setminus 2^{k-1}B}\overline{T}_{\sigma,\tau}A(\cdot,\cdot,\tau)\|_{L^{2}(\Sp)}^{2}\frac{\ud\tau}{\tau}\frac{\ud \sigma}{\sigma}\\
&\lesssim 2^{-2k(2N-2M-2\delta)}\int_{0}^{r^{2}}\|A(\cdot,\cdot,\tau)\|_{L^{2}(\Sp)}^{2}\frac{\ud\tau}{\tau}\leq2^{-2k(2N-2M-2\delta)}V(B_{A})^{-1}\\
&\lesssim 2^{-2k(2N-2M-2\delta)}V(B)^{-1}\lesssim 2^{-2k(2N-2M-2\delta-n)}V(2^{k+1}B)^{-1}.
\end{align*}
Since $2N-2M-n-2\delta>0$, this proves \eqref{eq:unifboundatom} and concludes the proof for $p=1$.

Next, recall that 
\begin{equation}\label{eq:adjoint}
R^{*}F(x,\w,\tau)=\int_{0}^{\infty}T_{\tau,\sigma}^{*}F(\cdot,\cdot,\sigma)(x,\w)\frac{\ud \sigma}{\sigma}
\end{equation}
for $F\in T^{1}(\Sp)\cap T^{2}(\Sp)$ and $(x,\w,\tau)\in\Spp$. Hence, by Proposition \ref{prop:offsing} and by what we have already shown, $R^{*}\in\La(T^{1}(\Sp))$. Now Lemma \ref{lem:tentdual} yields $R\in \La(T^{\infty}(\Sp))$, and Lemma \ref{lem:tentint} shows that $R\in\La(T^{p}(\Sp))$ for all $p\in(1,\infty)$ as well. 
\end{proof}

\section{Wave packet transforms}\label{sec:wavetransforms}

In this section we introduce the wave packet transforms that will be used throughout this article, and we derive some of their basic properties. 

\subsection{Wave packets}\label{subsec:wave packets}

We first introduce the specific wave packets which we will work with. Related wave packets were considered in \cite{Smith98a}, \cite{Candes-Demanet03} and \cite{Geba-Tataru07}.

Let $\ph\in\Sw(\Rn)$ be real-valued and such that $\ph(\zeta)=1$ if $|\zeta|\leq \frac{1}{2}$, and $\ph(\zeta)=0$ if $|\zeta|\geq2$. For $\w\in S^{n-1}$, $\sigma>0$ and $\zeta\in\Rn\setminus\{0\}$, set $\ph_{\w,\sigma}(\zeta):=c_{\sigma}\ph\big(\tfrac{\hat{\zeta}-\w}{\sqrt{\sigma}}\big)$, where $c_{\sigma}:=\big(\int_{S^{n-1}}\ph\big(\tfrac{e_{1}-\nu}{\sqrt{\sigma}}\big)^{2}\ud\nu\big)^{-1/2}$. Also set $\ph_{\w,\sigma}(0):=0$. Next, let $\Psi\in\Sw(\Rn)$ be real-valued and radial, with $\supp(\Psi)\subseteq\{\zeta\in\Rn\mid |\zeta|\in[\frac{1}{2},2]\}$, and such that 
\[
\int_{0}^{\infty}\Psi(\sigma\zeta)^{2}\frac{\ud \sigma}{\sigma}=1\quad(\zeta\neq0).
\]
Set $\Psi_{\sigma}(\zeta):=\Psi(\sigma\zeta)$ for $\sigma>0$ and $\zeta\in\Rn$. Finally, for $\w\in S^{n-1}$ and $\sigma>0$, set $\psi_{\w,\sigma}:=\Psi_{\sigma}\ph_{\w,\sigma}$. The following lemma collects some basic estimates for these wave packets.

\begin{lemma}\label{lem:packetbounds}
For all $\w\in S^{n-1}$ and $\sigma>0$ one has $\psi_{\w,\sigma}\in C^{\infty}_{c}(\Rn)$. Each $\zeta\in\supp(\psi_{\w,\sigma})$ satisfies $\frac{1}{2}\sigma^{-1}\leq |\zeta|\leq 2\sigma^{-1}$ and $|\hat{\zeta}-\w|\leq 2\sqrt{\sigma}$. Moreover,
\[
\int_{0}^{\infty}\int_{S^{n-1}}\psi_{\w,\sigma}(\zeta)^{2}\ud\w\frac{\ud\sigma}{\sigma}=1
\]
for all $\zeta\in\Rn\setminus\{0\}$. For all $\alpha\in\Z_{+}^{n}$ and $\beta\in\Z_{+}$ there exists a constant $C=C(\alpha,\beta)\geq0$ such that
\begin{equation}\label{eq:boundspsi}
|\lb\w,\nabla_{\zeta}\rb^{\beta}\partial_{\zeta}^{\alpha}\psi_{\w,\sigma}(\zeta)|\leq C\sigma^{-\frac{n-1}{4}+\frac{|\alpha|}{2}+\beta}
\end{equation}
for all $\w\in S^{n-1}$, $\sigma>0$ and $\zeta\in\Rn$. Also, for each $N\geq0$ there exists a $C_{N}\geq0$ such that, for all $\w\in S^{n-1}$, $\sigma>0$ and $x\in\Rn$,
\begin{equation}\label{eq:boundspsiinverse}
|\F^{-1}(\psi_{\w,\sigma})(x)|\leq C_{N}\sigma^{-\frac{3n+1}{4}}(1+\sigma^{-1}|x|^{2}+\sigma^{-2}\lb\w,x\rb^{2})^{-N}.
\end{equation}
In particular, $\{\sigma^{\frac{n-1}{4}}\F^{-1}(\psi_{\w,\sigma})\mid \w\in S^{n-1},\sigma>0\}\subseteq L^{1}(\Rn)$ is uniformly bounded.
\end{lemma}
\begin{proof}
The first two statements follow from the support properties of the Schwartz functions $\ph$ and $\Psi$. Next, for each $\zeta\neq0$ one has
\begin{align*}
\int_{0}^{\infty}\int_{S^{n-1}}\psi_{\w,\sigma}(\zeta)^{2}\ud\w\frac{\ud\sigma}{\sigma}=\int_{0}^{\infty}\Psi(\sigma \zeta)^{2}\frac{\int_{S^{n-1}}\ph\big(\frac{\hat{\zeta}-\w}{\sqrt{\sigma}}\big)^{2}\ud\w}{\int_{S^{n-1}}\ph\big(\frac{e_{1}-\nu}{\sqrt{\sigma}}\big)^{2}\ud\nu}\frac{\ud\sigma}{\sigma}=1.
\end{align*}
Next, set $E_{\sigma}:=\{\nu\in S^{n-1}\mid |e_{1}-\nu|\leq \frac{\sqrt{\sigma}}{2}\}$ and $F_{\sigma}:=\{\nu\in S^{n-1}\mid |e_{1}-\nu|\leq 2\sqrt{\sigma}\}$ for $\sigma>0$. Then 
\begin{equation}\label{eq:csigma}
\sigma^{\frac{n-1}{2}}\eqsim \int_{E_{\sigma}}\ud\nu\leq \int_{S^{n-1}}\ph\big(\tfrac{e_{1}-\nu}{\sqrt{\sigma}}\big)^{2}\ud\nu\lesssim\int_{F_{\sigma}}\ud\nu\eqsim \sigma^{\frac{n-1}{2}}.
\end{equation} 
Now, for \eqref{eq:boundspsi}, first consider the derivatives of $\Psi_{\sigma}$. One has
\begin{equation}\label{eq:boundsPsi}
|\lb\w,\nabla_{\zeta}\rb^{\beta_{1}}\partial_{\zeta}^{\alpha_{1}}\Psi_{\sigma}(\zeta)|=\sigma^{|\alpha_{1}|+\beta_{1}}|(\lb\w,\nabla_{\zeta}\rb^{\beta_{1}}\partial_{\zeta}^{\alpha_{1}}\Psi)(\sigma\zeta)|\lesssim \sigma^{|\alpha_{1}|+\beta_{1}}
\end{equation}
for all $\alpha_{1}\in\Z_{+}^{n}$ and $\beta_{1}\in\Z_{+}$ and an implicit constant independent of $\w\in S^{n-1}$, $\sigma>0$ and $\zeta\in\Rn$. On the other hand, for the derivatives of $\ph_{\w,\sigma}$, note that $\lb\w,\nabla_{\zeta}\rb=\lb\w-\hat{\zeta},\nabla_{\zeta}\rb+\lb\hat{\zeta},\nabla_{\zeta}\rb$. Since $\ph_{\w,\sigma}$ is positively homogeneous of degree $0$, one has $\lb\hat{\zeta},\nabla_{\zeta}\rb\ph_{\w,\sigma}(\zeta)=0$ for all $\zeta\neq0$. Also, $|\w-\hat{\zeta}|\leq 2\sqrt{\sigma}$ and $|\zeta|^{-1}\leq2\sigma$ for all $\zeta\in\supp(\psi_{\w,\sigma})$. Hence, using \eqref{eq:csigma}, for all $\alpha_{2}\in\Z_{+}^{n}$ and $\beta_{2}\in\Z_{+}$ one obtains
\begin{align*}
|\lb\w,\nabla_{\zeta}\rb^{\beta_{2}}\partial_{\zeta}^{\alpha_{2}}\ph_{\w,\sigma}(\zeta)|&=c_{\sigma}\sigma^{-\frac{|\alpha_{2}|+\beta_{2}}{2}}\big|\big(\lb\w-\hat{\zeta},\nabla_{\zeta}\rb^{\beta_{2}}\partial_{\zeta}^{\alpha_{2}}\ph\big)\big(\tfrac{\hat{\zeta}-\w}{\sqrt{\sigma}}\big)\big|\,|\zeta|^{-|\alpha_{2}|-\beta_{2}}\\
&\lesssim \sigma^{-\frac{n-1}{4}+\frac{|\alpha_{2}|}{2}+\beta_{2}}.
\end{align*}
Combined with \eqref{eq:boundsPsi}, this proves \eqref{eq:boundspsi}.

Finally, for \eqref{eq:boundspsiinverse}, fix $N\in\Z_{+}$, $\w\in S^{n-1}$ and $\sigma>0$ and set
\[
w(x):=1+\sigma^{-1}|x|^{2}+\sigma^{-2}\lb\w,x\rb^{2}\quad(x\in\Rn).
\]
Consider the self-adjoint differential operator
\[
L:=w(x)^{-1}(1-\sigma^{-1}\Delta_{\zeta}-\sigma^{-2}\lb\w,\nabla_{\zeta}\rb).
\]
By writing $e^{ix\cdot\zeta}=L^{N}(e^{ix\cdot\zeta})$ and integrating by parts, one expresses 
\[
\F^{-1}(\psi_{\w,\sigma})(x)=(2\pi)^{-n}\int_{\Rn}e^{i x\cdot\zeta}\psi_{\w,\sigma}(\zeta)\ud\zeta
\]
as a finite linear combination of terms of the form
\begin{equation}\label{eq:boundsinverseterms}
w(x)^{-N}\sigma^{-\frac{|\alpha|}{2}-\beta}\int_{\supp(\psi_{\w,\sigma})}\!e^{ix\cdot\zeta}\lb\w,\nabla_{\zeta}\rb^{\beta}\partial_{\zeta}^{\alpha}\psi_{\w,\sigma}(\zeta)\ud\zeta
\end{equation}
for $\alpha\in\Z_{+}^{n}$ and $\beta\in\Z_{+}$ with $|\alpha|+\beta\leq 2N$. By the support properties of $\psi_{\w,\sigma}$, $V(\supp(\psi_{\w,\sigma}))\lesssim \sigma^{-\frac{n+1}{2}}$. Now \eqref{eq:boundspsi} allows one to bound the absolute value of each term as in \eqref{eq:boundsinverseterms} by a constant multiple of $w(x)^{-N}\sigma^{-\frac{3n+1}{4}}$, which proves \eqref{eq:boundspsiinverse}. The final statement follows from these bounds by applying a suitable anisotropic substitution.
\end{proof}

Next, set
\[
r(\zeta):=\Big(\int_{1}^{\infty}\Psi_{\sigma}(\zeta)^{2}\frac{\ud\sigma}{\sigma}\Big)^{1/2}\quad(\zeta\neq0)
\]
and $r(0):=1$. It is straightforward to prove that $r\in C^{\infty}_{c}(\Rn)$, using that $r$ is radial, by showing that all derivatives of $r$ vanish where $r(\zeta)=0$.

\subsection{Wave packet transforms}\label{subsec:wavepackettransforms}

For $f\in \Sw'(\Rn)$, $(x,\w)\in\Sp$ and $\sigma>0$, set
\[
W_{\sigma}f(x,\w):=\begin{cases}\psi_{\w,\sigma}(D)f(x)&\text{if }\sigma\in(0,1),\\
V(S^{n-1})^{-1/2}\ind_{[1,e]}(\sigma)r(D)f(x)&\text{if }\sigma\geq1.\end{cases}
\]
Note that $W_{\sigma}f$ is well-defined, since $\psi_{\w,\sigma},r\in \Sw(\Rn)$. The following proposition contains some basic mapping properties of these wave packet transforms.

\begin{proposition}\label{prop:bddwavetransforms}
The following maps are continuous for all $\sigma>0$:
\begin{enumerate}
\item\label{it:wave1} $W_{\sigma}:L^{2}(\Rn)\to L^{2}(\Sp)$;
\item\label{it:wave2} $W_{\sigma}:\Sw(\Rn)\to \Da(\Sp)$;
\item\label{it:wave3} $W_{\sigma}^{*}:\Da(\Sp)\to \Sw(\Rn)$;
\item\label{it:wave4} $W_{\sigma}:\Sw'(\Rn)\to \Da'(\Sp)$.
\end{enumerate}
\end{proposition}
\begin{proof}
\eqref{it:wave1} and \eqref{it:wave2} follow directly from the fact that each of the Schwartz seminorms of $\psi_{\w,\sigma}$ is uniformly bounded in $\w\in S^{n-1}$. The same holds for \eqref{it:wave3}, upon observing in addition that, for $\sigma\in(0,1)$, one has
\[
\lb x\rb^{N}\partial_{x}^{\alpha}W_{\sigma}^{*}g(x)=\int_{S^{n-1}}\lb x\rb^{N}(\partial_{x}^{\alpha}\widecheck{\psi_{\w,\sigma}})\ast g(\cdot,\w)(x)\ud\w
\]
for all $N\geq0$, $\alpha\in \Z_{+}^{n}$, $g\in \Da(\Sp)$ and $x\in\Rn$. A similar identity holds for $\sigma\geq1$. For \eqref{it:wave4}, note that
\begin{equation}\label{eq:duality}
\begin{aligned}
\lb W_{\sigma}f,g\rb_{\Sp}&=\int_{\Sp}W_{\sigma}f(x)\overline{g(x,\w)}\ud x\ud\w\\
&=\int_{S^{n-1}}\lb f,\psi_{\w,\sigma}(D)g(\cdot,\w)\rb\ud\w=\lb f,W_{\sigma}^{*}g\rb
\end{aligned}
\end{equation}
for all $f\in\Sw'(\Rn)$ and $g\in \Da(\Sp)$ if $\sigma\in(0,1)$, and similarly for $\sigma\geq1$.
\end{proof}

We now combine the collection $(W_{\sigma})_{\sigma>0}$ into a single transform $W$. For $f\in \Sw'(\Rn)$ and $(x,\w,\sigma)\in\Spp$, set 
\[
Wf(x,\w,\sigma):=(W_{\sigma}f)(x,\w).
\]
This transform has similar mapping properties. 

\begin{proposition}\label{prop:bddwavetransform}
The following statements hold:
\begin{enumerate}
\item\label{it:W1} $W:L^{2}(\Rn)\to L^{2}(\Spp)$ is an isometry;
\item\label{it:W2} $W:\Sw(\Rn)\to \Da(\Spp)$ is continuous;
\item\label{it:W3} $W^{*}:\Da(\Spp)\to \Sw(\Rn)$ is continuous;
\item\label{it:W4} $W:\Sw'(\Rn)\to \Da'(\Spp)$ is continuous.
\end{enumerate}
\end{proposition}
\begin{proof}
By Lemma \ref{lem:packetbounds}, for all $f\in L^{2}(\Rn)$ one has
\begin{align*}
&\|Wf\|_{L^{2}(\Spp)}^{2}=\int_{\Spp}|Wf(x,\w,\sigma)|^{2}\ud x\ud\w\frac{\ud\sigma}{\sigma}\\
&=\int_{0}^{1}\int_{\Sp}|\psi_{\w,\sigma}(D)f(x)|^{2}\ud x\ud\w\frac{\ud\sigma}{\sigma}+\frac{1}{V(S^{n-1})}\int_{1}^{e}\int_{\Sp}|r(D)f(x)|^{2}\ud x\ud\w\frac{\ud\sigma}{\sigma}\\
&=(2\pi)^{-n}\int_{0}^{1}\int_{S^{n-1}}\int_{\Rn}|\psi_{\w,\sigma}(\zeta)\wh{f}(\zeta)|^{2}\ud \zeta\ud\w\frac{\ud\sigma}{\sigma}+(2\pi)^{-n}\int_{\Rn}|r(\zeta)\wh{f}(\zeta)|^{2}\ud\zeta\\
&=(2\pi)^{-n}\int_{\Rn}\Big(\int_{0}^{\infty}\Psi_{\sigma}(\zeta)^{2}\frac{\ud\sigma}{\sigma}\Big)|\wh{f}(\zeta)|^{2}\ud\zeta= \|f\|^{2}_{L^{2}(\Rn)}.
\end{align*}

Next, suppose that $f\in \Sw(\Rn)$. Since $Wf(x,\w,\sigma)=0$ if $\sigma>e$, and because $r\in \Sw(\Rn)$, to prove that $Wf\in\Da(\Spp)$ it suffices to show that for all $N\in\N$ there exists a $C_{N}\geq0$ such that
\[
(1+|x|+\sigma^{-1})^{N}|Wf(x,\w,\sigma)|\leq C_{N}
\]
for all $(x,\w,\sigma)\in\Spp$ with $\sigma<1$. We first consider the case where $|x|\leq 1$. Recall that $V(\supp(\psi_{\w,\sigma}))\lesssim \sigma^{-\frac{n+1}{2}}$, by Lemma \ref{lem:packetbounds}. Another application of that lemma yields
\begin{align*}
&(1+|x|+\sigma^{-1})^{N}|Wf(x,\w,\sigma)|\lesssim \sigma^{-N}|\psi_{\w,\sigma}(D)f(x)|\\
&=\sigma^{-N}(2\pi)^{-n}\Big|\int_{\Rn}e^{ix\cdot\zeta}\psi_{\w,\sigma}(\zeta)\wh{f}(\zeta)\ud\zeta\Big|\\
&\lesssim \sigma^{-N}\int_{\supp(\psi_{\w,\sigma})}\sigma^{-\frac{n-1}{4}}|\zeta|^{-(N+\frac{n-1}{4}+\frac{n+1}{2})}\ud\zeta\lesssim 1,
\end{align*}
where we used that $\wh{f}\in\Sw(\Rn)$. In the same manner, for $|x|\geq1$ we integrate by parts and use Lemma \ref{lem:packetbounds} to write
\begin{align*}
&(1+|x|+\sigma^{-1})^{N}|Wf(x,\w,\sigma)|\lesssim |x|^{2N}\sigma^{-N}\Big|\int_{\Rn}e^{ix\cdot\zeta}\psi_{\w,\sigma}(\zeta)\wh{f}(\zeta)\ud\zeta\Big|\\
&=\sigma^{-N}\Big|\int_{\Rn}e^{ix\cdot\zeta}\Delta_{\zeta}^{N}(\psi_{\w,\sigma}\wh{f}\,)(\zeta)\ud\zeta\Big|\leq \sigma^{-N}\int_{\supp(\psi_{\w,\sigma})}|\Delta_{\zeta}^{N}(\psi_{\w,\sigma}\wh{f}\,)(\zeta)|\ud\zeta\\
&\lesssim \sigma^{-N-\frac{n-1}{4}}\int_{\supp(\psi_{\w,\sigma})}|\zeta|^{-(N+\frac{n-1}{4}+\frac{n+1}{2})}\ud\zeta\lesssim 1.
\end{align*}
This proves \eqref{it:W2}.

For \eqref{it:W3}, let $F\in \Da(\Spp)$ and note that
\begin{align*}
&W^{*}F(x)=\int_{0}^{\infty}W_{\sigma}^{*}F(\cdot,\cdot,\sigma)(x)\frac{\ud\sigma}{\sigma}\\
&=\!\int_{0}^{1}\int_{S^{n-1}}\psi_{\w,\sigma}(D)F(\cdot,\w,\sigma)(x)\ud\w\frac{\ud\sigma}{\sigma}+\int_{1}^{e}\int_{S^{n-1}}r(D)F(\cdot,\w,\sigma)(x)\frac{\ud\w}{V(S^{n-1})^{\frac{1}{2}}}\frac{\ud\sigma}{\sigma}
\end{align*}
for all $x\in\Rn$. We show that $\F(W^{*}F)\in\Sw(\Rn)$. To this end, first note that $\F(F(\cdot,\w,\sigma))\in C^{\infty}(\Rn)$ for all $\w\in S^{n-1}$ and $\sigma>0$, and that $\F(F(\cdot,\w,\sigma))$ and all its derivatives are bounded uniformly in $\w$ and $\sigma$. It follows that $\F(r(D)F(\cdot,\w,\sigma))\in\Sw(\Rn)$, with Schwartz seminorms which are uniformly bounded in $\w$ and $\sigma$. Hence 
\[
\F\Big(\int_{1}^{e}\int_{S^{n-1}}r(D)F(\cdot,\w,\sigma)\frac{\ud\w}{V(S^{n-1})^{\frac{1}{2}}}\frac{\ud\sigma}{\sigma}\Big)\in\Sw(\Rn).
\]
Similarly, Lemma \ref{lem:packetbounds} and the decay assumption on $F$ imply that for each $N\in\N$ and $\alpha\in\Z_{+}^{n}$ one has
\[
|\partial_{\zeta}^{\alpha}(\psi_{\w,\sigma}\F(F(\cdot,\w,\sigma)))(\zeta)|\lesssim \sigma^{N+1}\eqsim \sigma\lb\zeta\rb^{-N},
\]
for an implicit constant independent of $\w\in S^{n-1}$, $\sigma\in(0,1)$ and $\zeta\in\supp(\psi_{\w,\sigma})$. It follows that $\F(\psi_{\w,\sigma}(D)F(\cdot,\w,\sigma))\in\Sw(\Rn)$ and 
\[
\F\Big(\int_{0}^{1}\int_{S^{n-1}}\psi_{\w,\sigma}(D)F(\cdot,\w,\sigma)\ud\w\frac{\ud\sigma}{\sigma}\Big)\in\Sw(\Rn),
\]
as required. 

Finally, for \eqref{it:W4} one shows as in \eqref{eq:duality} that $\lb Wf,F\rb_{\Spp}=\lb f,W^{*}F\rb$ for all $f\in\Sw'(\Rn)$ and $F\in \Da(\Spp)$.
\end{proof}

It follows from Proposition \ref{prop:bddwavetransforms} that we may extend $W^{*}$ to a continuous map $W^{*}:\Da'(\Spp)\to \Sw'(\Rn)$ by $\lb W^{*}F,f\rb:=\lb F,Wf\rb_{\Spp}$ for $F\in\Da'(\Spp)$ and $f\in\Sw(\Rn)$, and then
\begin{equation}\label{eq:identity}
W^{*}Wf=f\quad(f\in\Sw'(\Rn)).
\end{equation}

\begin{remark}\label{rem:othertransforms}
Wave packet transforms closely related to $W$ were considered in \cite{Smith98a} and \cite{Geba-Tataru07}. There are some differences between those transforms and the one in this article, such as the use of slightly different wave packets. Moreover, wave packet transforms were only used implicitly in \cite{Smith98a}, through a reproducing formula, whereas one of the goals of this article is to explicitly combine wave packet transforms with function spaces on phase space to study Fourier integral operators. On the other hand, the wave packet transform in \cite{Geba-Tataru07} is an isometry between $L^{2}(\Rn)$ and $L^{2}(\Tp)$, where $\Tp$ is equipped with the Liouville measure $\ud x\ud\xi$, as well as a continuous map between $\Sw(\Rn)$ and $\Sw(\Tp)=\Sw(\R^{2n})$. For the $L^{2}$-theory of Fourier integral operators such mapping properties are natural, but when using square function norms and tent spaces for the $L^{p}$-theory it appears to be more natural to identify $\Tp\setminus o$ with $\Spp=\Sp\times(0,\infty)$ and to equip the latter with the measure $\ud x\ud\w\frac{\ud\sigma}{\sigma}$. This then leads one to consider a wave packet transform with mapping properties as in Proposition \ref{prop:bddwavetransform}. 
\end{remark}

\section{Off-singularity bounds for Fourier integral operators}\label{sec:offsingFIO}

In this section we obtain off-singularity bounds for Fourier integral operators, using the wave packet transforms from the previous section.

\subsection{The main theorem and its corollaries}

Throughout this section, fix $\wt{\ph}$, $\wt{\Psi}$ and $\wt{r}$ with the same properties as the functions $\ph$, $\Psi$ and $r$ from the previous section. More precisely, $\wt{\ph}\in\Sw(\Rn)$ is real-valued and such that $\wt{\ph}(\zeta)=1$ if $|\zeta|\leq \frac{1}{2}$, and $\wt{\ph}(\zeta)=0$ if $|\zeta|\geq2$. Also, $\wt{\Psi}\in\Sw(\Rn)$ is real-valued and radial, with $\supp(\wt{\Psi})\subseteq\{\zeta\in\Rn\mid |\zeta|\in[\frac{1}{2},2]\}$, and such that 
\[
\int_{0}^{\infty}\wt{\Psi}(\sigma\zeta)^{2}\frac{\ud \sigma}{\sigma}=1\quad(\zeta\neq0).
\]
Let
\[
\wt{r}(\zeta):=\Big(\int_{1}^{\infty}\wt{\Psi}(\sigma\zeta)^{2}\frac{\ud\sigma}{\sigma}\Big)^{1/2}\quad(\zeta\neq0)
\]
and $\wt{r}(0):=1$. For $\w\in S^{n-1}$, $\sigma>0$ and $\zeta\in\Rn$, set $\wt{\ph}_{\w,\sigma}(\zeta):=\wt{c}_{\sigma}\wt{\ph}\big(\tfrac{\hat{\zeta}-\w}{\sqrt{\sigma}}\big)$ if $\zeta\neq0$, where $\wt{c}_{\sigma}:=\big(\int_{S^{n-1}}\wt{\ph}\big(\tfrac{e_{1}-\nu}{\sqrt{\sigma}}\big)^{2}\ud\nu\big)^{-1/2}$, and $\wt{\ph}_{\w,\sigma}(0):=0$. Also set $\wt{\Psi}_{\sigma}(\zeta):=\wt{\Psi}(\sigma\zeta)$ and $\wt{\psi}_{\w,\sigma}:=\wt{\Psi}_{\sigma}\wt{\ph}_{\w,\sigma}$. Next, for $f\in \Sw'(\Rn)$ and $(y,\nu,\tau)\in\Spp$, set 
\[
V_{\tau}f(y,\nu):=\begin{cases}
\wt{\psi}_{\nu,\tau}(D)f(y)&\text{if }\tau\in(0,1),\\
V(S^{n-1})^{-1/2}\ind_{[1,e]}(\tau)\wt{r}(D)f(y)&\text{if }\tau\geq1,
\end{cases}
\]
and $Vf(y,\nu,\tau):=V_{\tau}f(y,\nu)$. 

Recall from Definition \ref{def:operator} the definition of a normal oscillatory integral operator. The main technical result of this section is as follows. 

\begin{theorem}\label{thm:offsingFIO}
Let $T$ be a normal oscillatory integral operator of order $0$ and type $(\frac{1}{2},\frac{1}{2},1)$ with symbol $a$ and phase function $\Phi$. Let $\hchi$ be the contact transformation induced by $\Phi$, and suppose that $\hchi$ is bi-Lipschitz. Suppose also that either $(z,\theta)\mapsto \Phi(z,\theta)$ is linear in $\theta$, or that there exists an $\veps>0$ such that $a(z,\theta)=0$ for all $(z,\theta)\in \R^{2n}$ with $|\theta|<\veps$. For $\sigma,\tau>0$ let $K_{\sigma,\tau}$ be the kernel of $W_{\sigma}TV_{\tau}^{*}$, and set $\rho:=\min(\sigma,\tau)$. Then for each $N\geq0$ there exists a $C\geq0$ such that, for all $(x,\w),(y,\nu)\in\Sp$, one has
\begin{equation}\label{eq:mainestimate1}
|K_{\sigma,\tau}((x,\w),(y,\nu))|\leq C\Upsilon(\tfrac{\sigma}{\tau})^{N}\rho^{-n}\big(1+\rho^{-1}d((x,\w),\hchi(y,\nu))^2 \big)^{-N}
\end{equation}
if $(y,\nu)\in\dom(\hchi)$, and 
\[
|K_{\sigma,\tau}((x,\w),(y,\nu))|\leq C \big(1 + \rho^{-1} d((x, \omega), \hchi(A))^2 \big)^{-N}  \big(1 + \rho^{-1} d((y, \nu), A)^2 \big)^{-N}
\]
if $(y,\nu)\notin\dom(\hchi)$ or $(x,\w)\notin\ran(\hchi)$. Here $A:=\{(\nabla_{\theta}\Phi(z,\theta),\theta)\mid (z,\theta)\in\mathrm{base}(a)\}$.
\end{theorem}

In the next section we will mostly use the following corollary of Theorem \ref{thm:offsingFIO}, concerning the boundedness of normal oscillatory integral operators, conjugated with wave packet transforms, on the tent spaces $T^{p}(\Sp)$, $p\in[1,\infty]$.

\begin{corollary}\label{cor:normalbounds}
Let $T$ be a normal oscillatory integral operator of order $0$ and type $(\frac{1}{2},\frac{1}{2},1)$ with symbol $a$ and phase function $\Phi$. Let $\hchi$ be the contact transformation induced by $\Phi$, and suppose that $\dom(\hchi)=\ran(\hchi)=\Sp$. Suppose also that either $(z,\theta)\mapsto \Phi(z,\theta)$ is linear in $\theta$, or there exists an $\veps>0$ such that $a(z,\theta)=0$ for all $(z,\theta)\in \R^{2n}$ with $|\theta|<\veps$. Then for each $N\geq0$ there exists a $C\geq0$ such that
\[
(W_{\sigma}TV_{\tau}^{*})_{\sigma,\tau>0}\in OS(\hchi,N,C).
\]
In particular, $WTV^{*}\in \La(T^{p}(\Sp))$ for all $p\in[1,\infty]$.
\end{corollary}
\begin{proof}
Firstly, by writing out the kernel $K_{\sigma,\tau}$ of $W_{\sigma}TV_{\tau}^{*}$ (see e.g.~\eqref{eq:kernelrep} below), one sees that $((x,\w,\sigma),(y,\nu,\tau))\mapsto K_{\sigma,\tau}((x,\w),(y,\nu))$ is measurable on $\Spp\times\Spp$. Moreover, by Lemma \ref{lem:philip}, $\hchi:\Sp\to\Sp$ is bi-Lipschitz. Hence \eqref{eq:mainestimate1} holds for all $(y,\nu)\in\Sp$, which yields the first statement. The second statement now follows directly from Theorem \ref{thm:tentbounded}, since
\[
WTV^{*}F(x,\w,\sigma)=\int_{0}^{\infty}W_{\sigma}TV^{*}_{\tau}F(\cdot,\cdot,\tau)(x,\w)\frac{\ud\tau}{\tau}
\]
for all $F\in \Da(\Spp)$ and $(x,\w,\sigma)\in\Spp$.
\end{proof}

For our next corollary we need a proposition on the boundedness of smoothing operators, conjugated with wave packet transforms, on tent spaces.

\begin{proposition}\label{prop:smoothingbounds}
Let $R$ be a smoothing operator. Then $(W_{\sigma}RV_{\tau}^{*})_{\sigma,\tau>0}\in\Ra$ and $WRV^{*}\in\La(T^{p}(\Sp))$ for all $p\in[1,\infty]$.
\end{proposition}
\begin{proof}
For $\sigma,\tau>0$ let $K_{\sigma,\tau}$ be the kernel of $W_{\sigma}RV_{\tau}^{*}$, and fix $N\in\N$. The measurability condition is straightforward to check, and since $K_{\sigma,\tau}=0$ if $\sigma>e$ or $\tau>e$, we only have to show that 
\begin{equation}\label{eq:mainresidual}
|K_{\sigma,\tau}((x,\w),(y,\nu))|\lesssim (1+|x|+|y|+\sigma^{-1}+\tau^{-1})^{-N}
\end{equation}
for $(x,\w),(y,\nu)\in\Sp$. Let $K\in \Sw(\Rn\times\Rn)$ be the kernel of $R$. 

Suppose that $\sigma,\tau<1$, with the proof being similar if $\max(\sigma,\tau)\in[1,e]$. Then 
\[
K_{\sigma,\tau}((x,\w),(y,\nu))=\frac{1}{(2\pi)^{2n}}\int_{\R^{4n}}\!e^{i((x-z)\cdot\zeta+(u-y)\cdot\theta)}\psi_{\w,\sigma}(\zeta)K(z,u)\wt{\psi}_{\nu,\tau}(\theta)\ud\theta\ud u\ud \zeta \ud z
\]
for all $(x,\w),(y,\nu)\in\Sp$. Consider the following differential operators:
\begin{align*}
D_{1}:=-|\zeta|^{-2}\Delta_{z}&,\quad D_{2}:=-|\theta|^{-2}\Delta_{u},\\
D_{3}:=(1+|x-z|^{2})^{-1}(1-\Delta_{\zeta})&,\quad D_{4}:=(1+|u-y|^{2})^{-1}(1-\Delta_{\theta}).
\end{align*}
Note that each of these leaves the exponential in the representation for $K_{\sigma,\tau}$ invariant. Now, first integrate by parts sufficiently many times with respect to $D_{1}$ and $D_{2}$, and then with respect to $D_{3}$ and $D_{4}$, and use the properties of $\psi_{\w,\sigma}$ and $\wt{\psi}_{\nu,\tau}$ from Lemma \ref{lem:packetbounds} and the assumption that $K\in\Sw(\R^{2n})$, to write
\begin{align*}
&|K_{\sigma,\tau}((x,\w),(y,\nu))|\\
&\lesssim (\sigma\tau)^{-\frac{n-1}{4}}\int_{\R^{4n}}\big((1+|x-z|+|z|+|u-y|+|u|)|\zeta||\theta|\big)^{-M}\ud u\ud z\ud\theta\ud \zeta\\
&\lesssim (\sigma\tau)^{M-\frac{n-1}{4}}(1+|x|+|y|)^{4n-M}\int_{\supp(\psi_{\w,\sigma})}\int_{\supp(\wt{\psi}_{\nu,\tau})}\ud\theta\ud \zeta\\
&\lesssim (1+|x|+|y|+\sigma^{-1}+\tau^{-1})^{-N}
\end{align*}
for $M\geq0$ large enough, where we used Lemma \ref{lem:packetbounds} to see that $V(\supp(\psi_{\w,\sigma}))\lesssim \sigma^{\frac{n+1}{2}}$ and $V(\supp(\wt{\psi}_{\nu,\tau}))\lesssim \tau^{\frac{n+1}{2}}$.

We have now shown that $(W_{\sigma}RV_{\tau}^{*})_{\sigma,\tau>0}\in\Ra$. By Proposition \ref{prop:residualbounded}, this in turn implies that $WRV^{*}\in\La(T^{p}(\Sp))$ for all $p\in[1,\infty]$.
\end{proof}

\begin{corollary}\label{cor:FIObounds}
Let $T$ be a Fourier integral operator of order $0$ and type $(\rho,1-\rho,1)$, for $\rho\in(\frac{1}{2},1]$, associated with a local canonical graph. Suppose that the Schwartz kernel of $T$ has compact support. Then $WTV^{*}\in \La(T^{p}(\Sp))$ for all $p\in[1,\infty]$.
\end{corollary}
\begin{proof}
Fix $p\in[1,\infty]$. Let $R$ be a smoothing operator and let $T_{1,1},\ldots, T_{m,1}$ and $T_{1,2},\ldots, T_{m,2}$, for some $m\in\N$, be normal oscillatory integral operators of order $0$ and type $(\rho,1-\rho,1)$ such that $T=\sum_{j=1}^{m}T_{j,1}T_{j,2}+R$, as in Proposition \ref{prop:FIOnormal}. By Proposition \ref{prop:smoothingbounds}, one has $WRV^{*}\in \La(T^{p}(\Sp))$. Hence it suffices to show that $WT_{j,1}T_{j,2}V^{*}\in \La(T^{p}(\Sp))$ for all $1\leq j\leq m$. And, by using that $W$ is an isometry to write 
\[
WT_{j,1}T_{j,2}V^{*}=(WT_{j,1}W^{*})(WT_{j,2}V^{*}),
\]
it suffices to show that $WT_{j,k}W^{*}\in\La(T^{p}(\Sp))$ for all $j\in\{1,\ldots,m\}$ and $k\in\{1,2\}$. 

Fix $j\in\{1,\ldots,m\}$ and $k\in\{1,2\}$, and let $a_{j,k}$ and $\Phi_{j,k}$ be the symbol and phase function of $T_{j,k}$, respectively. Note that the cone support of $a_{j,k}$ has a compact base, since the Schwartz kernel of $T_{j,k}$ has compact support. Hence 
\[
A_{j,k}:=\{(\nabla_{\theta}\Phi_{j,k}(z,\theta),\theta)\mid (z,\theta)\in\text{base}(a_{j,k})\}\subseteq\dom(\hchi_{j,k})
\]
is compact as well, where $\hchi_{j,k}$ is the contact transformation induced by $\Phi_{j,k}$. By Lemma \ref{lem:philip}, we may now assume that $\hchi_{j,k}$ is bi-Lipschitz. Also, one has $a_{j,k}(z,\theta)=0$ for $(z,\theta)\in\R^{2n}$ with $|\theta|<1$, by Proposition \ref{prop:FIOnormal}. Now set $S_{j,k}:=\ind_{\ran(\hchi_{j,k})}T_{j,k}\ind_{\dom(\hchi_{j,k})}$. It follows from Theorem \ref{thm:offsingFIO} that, for all $N\geq0$, there exists a $C\geq0$ such that 
\[
(W_{\sigma}S_{j,k}V_{\tau}^{*})_{\sigma,\tau>0}\in OS(\hchi_{j,k},N,C).
\]
Similarly, one has 
\[
(W_{\sigma}(T_{j,k}-S_{j,k})V_{\tau}^{*})_{\sigma,\tau>0}\in \mathcal{R},
\]
where we used that $A_{j,k}$ and $\hchi_{j,k}(A_{j,k})$ are compact, and that $\dom(\hchi_{j,k})$ and $\ran(\hchi_{j,k})$ are $\epsilon$-neighborhoods of $A_{j,k}$ and $\hchi_{j,k}(A_{j,k})$, respectively, for some $\veps>0$. Finally, Theorem \ref{thm:tentbounded} and Proposition \ref{prop:residualbounded} conclude the proof.
\end{proof}

\subsection{Proof of Theorem \ref{thm:offsingFIO}}\label{subsec:proof}

The proof is somewhat similar in spirit to that of Proposition \ref{prop:smoothingbounds}, in that we will integrate by parts repeatedly with respect to several differential operators to introduce suitable decay factors. However, for a normal oscillatory integral operator the analysis is much more delicate. In particular, apart from having to split into several cases when $\dom(\hchi)\neq \Sp$ or $\ran(\hchi)\neq \Sp$, one has to take care when integrating by parts using anisotropic differential operators, to ensure that the corresponding anisotropic weights do not blow up after application of these operators.

\subsubsection{Preliminary work}

We assume that $a$ is compactly supported in the first variable, to ensure that the relevant integrals converge absolutely. The bounds which we obtain depend only on $n$, $N$, $\ph$, $\Psi$, the Lipschitz constants of $\hchi$ and $\hchi^{-1}$, finitely many of the $S^{0}_{\frac{1}{2},\frac{1}{2},1}$ seminorms of $a$, and on the following bounds for $\Phi$:
\begin{equation}\label{eq:phibounds}
\sup_{(z,\theta)\in \dom(\Phi)}|\partial_{z}^{\alpha}\partial_{\theta}^{\beta}\Phi(z,\hat{\theta})|<\infty
\end{equation}
for $\alpha,\beta\in\Z_{+}^{n}$ with $|\alpha|+|\beta|\geq 2$, and
\begin{equation}\label{eq:phibounds2}
\inf_{(z,\theta)\in \dom(\Phi)}|\text{det}(\partial^{2}_{z\theta}\Phi(z,\theta))|>0.
\end{equation}
For general $a$ one then multiplies by smooth cut-offs and applies the dominated convergence theorem to a suitable expression for $K_{\sigma,\tau}$.

Fix $x,y\in\Rn$, $\w,\nu\in S^{n-1}$ and $\sigma,\tau\in(0,e]$, where the restriction on $\sigma$ and $\tau$ is allowed since $K_{\sigma,\tau}=0$ otherwise. Let $\rho\in C^{\infty}(\R^{2n})$, $|\rho|\leq 1$, have bounded derivatives and be such that $\rho(z,\theta)=1$ if $\big|\tfrac{\nabla_{z}\Phi(z,\hat{\theta})}{|\nabla_{z}\Phi(z,\hat{\theta})|}-\w\big|\leq \tfrac{1}{2}$ or $|\theta|\leq \tfrac{1}{4}$, and $\rho(z,\theta)= 0$ if $\big|\tfrac{\nabla_{z}\Phi(z,\hat{\theta})}{|\nabla_{z}\Phi(z,\hat{\theta})|}-\w\big|\geq 1$ and $|\theta|\geq\frac{1}{2}$. We may also suppose that $\rho$ is homogeneous of degree zero in the $\theta$-variable for $|\theta|\geq\tfrac{1}{2}$. For $z,\zeta,\theta\in\Rn$, set 
\[
a_{1}(z,\zeta,\theta):=(\sigma\tau)^{\frac{n-1}{4}}\psi_{\w,\sigma}(\zeta)\rho(z,\theta)a(z,\theta)\wt{\psi}_{\nu,\tau}(\theta)
\]
and 
\[
a_{2}(z,\zeta,\theta):=(\sigma\tau)^{\frac{n-1}{4}}\psi_{\w,\sigma}(\zeta)(1-\rho(z,\theta))a(z,\theta)\wt{\psi}_{\nu,\tau}(\theta)
\]
if $\sigma,\tau<1$. For $\sigma\in[1,e]$ one replaces $\psi_{\w,\sigma}$ by $V(S^{n-1})^{-1/2}r$, and for $\tau\in[1,e]$ one replaces $\wt{\psi}_{\nu,\tau}$ by $V(S^{n-1})^{-1/2}\wt{r}$. Note that $\rho$ can be chosen such that its derivatives are bounded independently of $\w$, and these bounds depend on $\Phi$ only through \eqref{eq:phibounds} and \eqref{eq:phibounds2}. Hence it follows from the properties of $\psi_{\w,\sigma}$ and $\wt{\psi}_{\nu,\tau}$ from Lemma \ref{lem:packetbounds} that, for all $j\in\{1,2\}$, $\alpha,\beta,\delta\in\Z_{+}^{n}$ and $\gamma,\veps\in\Z_{+}$, the quantity
\begin{equation}\label{eq:Shalf}
\sup_{z,\zeta\neq0,\theta\neq0} \lb\theta\rb^{-\frac{1}{2}|\alpha|+\frac{1}{2}|\beta|+ \gamma}\lb \zeta\rb^{\frac{1}{2}|\delta|+\veps}\big|\partial^{\alpha}_{z}\partial_{\theta}^{\beta}\lb\hat{\theta},\nabla_{\theta}\rb^{\gamma}\partial_{\zeta}^{\delta}\lb\hat{\zeta},\nabla_{\zeta}\rb^{\veps}a_{j}(z,\zeta,\theta)|
\end{equation}
is bounded in terms of the $S^{0}_{\frac{1}{2},\frac{1}{2},1}$ seminorms of $a$, and \eqref{eq:phibounds} and \eqref{eq:phibounds2}. Let $T_{j}$ be the normal oscillatory integral operator with phase $\Phi$ and symbol $(\sigma\tau)^{-\frac{n-1}{4}}a_{j}$, and let $K^{j}_{\sigma,\tau}$ be the Schwartz kernel of $T_{j}$. It then suffices to obtain the required bounds with $K_{\sigma,\tau}((x,\w),(y,\nu))$ replaced by $K^{j}_{\sigma,\tau}(x,y)$, in terms of \eqref{eq:phibounds} -- \eqref{eq:Shalf}. 

Throughout, we will use the following representation for $j=1,2$:
\begin{equation}\label{eq:kernelrep}
K_{\sigma,\tau}^{j}(x,y)=\frac{(\sigma\tau)^{-\frac{n-1}{4}}}{(2\pi)^{n}}\int_{\R^{3n}}e^{i((x-z)\cdot\zeta+\Phi(z,\theta)-y\cdot\theta)}a_{j}(z,\zeta,\theta)\ud\theta\ud z\ud\zeta.
\end{equation}
We use differential operators adapted to the phase function in this representation:
\[
\Pi(z,\zeta,\theta):=(x-z)\cdot\zeta+\Phi(z,\theta)-y\cdot\theta
\]
for $\zeta\in\Rn$ and $(z,\theta)\in \supp(a)\setminus o$. These operators are as follows:
\begin{align*}
D_1 &:= \big( 1 + \tau^{-1} |x-z|^2 \big)^{-1} \big( 1 -i \tau^{-\frac{1}{2}} (x-z) \cdot \tau^{-\frac{1}{2}} \nabla_\zeta \big) , 
\\
D_2 &:= \big( 1 + \tau^{-1} |\nabla_{\theta}\Phi(z, \theta) - y|^2 \big)^{-1} \big( 1 -i\tau^{-\frac{1}{2}} (\nabla_{\theta}\Phi(z, \theta) - y) \cdot \tau^{-\frac{1}{2}} \nabla_\theta \big) , 
\\
D_3 &:= \big( 1 + \tau^{-1} | P^\perp_{\w} (\nabla_{z}\Phi(z, \hat\theta) - \tfrac{\zeta}{|\theta|} )|^2 \big)^{-1} \big( 1 -i\frac{\tau^{-\frac{3}{2}}}{|\theta|}  P^\perp_{\w}\big(\nabla_{z}\Phi(z, \hat\theta) - \tfrac{\zeta}{|\theta|}\big)  \cdot \tau^{\frac{1}{2}} \nabla_z \big),\\
D_4 &:= \big( 1 + \tau^{-2} \lb \w,x-z\rb^2 \big)^{-1} \big( 1 -i\tau^{-1}\lb \w,x-z\rb   \tau^{-1}\w\cdot \nabla_{\zeta}\big) , \\
D_5 &:= \big( 1 + (\tau|\theta|)^{-2} \Pi(z,\zeta,\theta)^2  \big)^{-1} \big( 1-i(\tau|\theta|)^{-1} \Pi(z,\zeta,\theta) \tau^{-1} ( \hat \theta \cdot \nabla_\theta + |\theta|^{-1} \zeta \cdot \nabla_\zeta) \big),\\
D_6 &:= \big( 1 + \tau |\nabla_{z}\Phi(z, \theta) - \zeta |^2 \big)^{-1} \big( 1 -i\tau^{\frac{1}{2}} (\nabla_{z}\Phi(z, \theta) - \zeta)  \cdot \tau^{\frac{1}{2}} \nabla_z \big).
\end{align*}
Note that $D_{j}(e^{i\Pi})=e^{i\Pi}$ for $j\in\{1,\ldots,6\}$. We can write these differential operators in the following form: 
\begin{align*}
D_1 &= f_{1,1}^2 -i f_{1,1} f_{1,2}\cdot\tau^{-1/2} \nabla_\zeta, \\
D_2 &= f_{2,1}^2 -i f_{2,1} f_{2,2}\cdot\tau^{-1/2} \nabla_\theta, \\
D_3 &= f_{3,1}^2 -i f_{3,1} f_{3,2}\cdot  \tau^{1/2} P_{\w}^{\perp}\nabla_z, \\
D_4 &= f_{4,1}^2 -i f_{4,1} f_{4,2}\, \tau^{-1} \w\cdot\nabla_{\zeta}, \\
D_5 &= f_{5,1}^2 -i f_{5,1} f_{5,2}\,\tau^{-1} (\hat \theta \cdot \nabla_\theta + |\theta|^{-1} \zeta \cdot \nabla_\zeta),\\
D_6 &= f_{6,1}^2 -i f_{6,1} f_{6,2}\cdot\tau^{1/2} \nabla_z,
\end{align*}
where each $f_{j, 1}$ is the square root of the first factor in parentheses in the definition of $D_j$. For example, 
\[
f_{3,1} =f_{3,1}(z,\zeta,\theta)= \big( 1 + \tau^{-1} \big| P^\perp_{\w} \big(\nabla_{z}\Phi(z, \hat\theta) - \tfrac{\zeta}{|\theta|} \big)   \big|^2 \big)^{-1/2}
\]
and
\[
f_{3,2} =f_{3,2}(z,\zeta,\theta)= \frac{\tau^{-3/2}|\theta|^{-1}P^\perp_{\w}\big(\nabla_{z}\Phi(z, \hat\theta) - \tfrac{\zeta}{|\theta|}\big)}{\big( 1 + \tau^{-1} \big| P^\perp_{\w} \big(\nabla_{z}\Phi(z, \hat\theta) - \tfrac{\zeta}{|\theta|} \big)\big|^2 \big)^{1/2}},
\]
and similarly for other $j\in\{1,\ldots,6
\}$.

\subsubsection{Technical lemma}

For the rest of the proof, let $c,c'\in(0,1)$ be such that, if either $\sigma<c\tau$ or $\sigma>\frac{1}{c}\tau$, then
\begin{equation}\label{eq:defc}
|\nabla_{z}\Phi(z,\theta)-\zeta|\geq c'\max(\sigma^{-1},\tau^{-1})
\end{equation}
for all $(z,\zeta,\theta)\in\supp(a_{1})\cup\supp(a_{2})$ with $\theta\neq0$. Such $c$ and $c'$ exist because of the bounds on $\Phi$ and the support properties of $\psi_{\w,\sigma}$ and $\wt{\psi}_{\nu,\tau}$. We will not use the specific properties of $c$ and $c'$ for now, but we note that $c$ and $c'$ can be chosen dependent only on the bounds for $\Phi$ from \eqref{eq:phibounds} and \eqref{eq:phibounds2}. 

We now present a lemma that will allow us to deal with terms that arise by integrating by parts. Set $\w_{1}:=\w$, and let $\w_{2},\ldots, \w_{n}\in S^{n-1}$ be such that $\{\w_{1},\w_{2},\ldots,\w_{n}\}$ is a basis of $\Rn$. The lemma considers the following differential operators, for $2\leq k\leq n$ and $1\leq l\leq n$:
\begin{equation}\label{eq:operators1}
\tau^{-\frac{1}{2}}\w_{k}\cdot \nabla_{\zeta},\quad\tau^{-1}\w\cdot\nabla_{\zeta},\quad\tau^{-\frac{1}{2}}\w_{l}\cdot\nabla_{\theta},\quad\tau^{\frac{1}{2}}\w_{k}\cdot\nabla_{z},\quad \tau^{-1}(\hat{\theta}\cdot\nabla_{\theta}+|\theta|^{-1}\zeta\cdot\nabla_{\zeta}),
\end{equation}
as well as 
\begin{equation}\label{eq:operators2}
\tau^{-\frac{1}{2}}\w_{l}\cdot\nabla_{\zeta},\quad\tau^{-\frac{1}{2}}\w_{l}\cdot\nabla_{\theta},\quad\tau^{\frac{1}{2}}\w_{l}\cdot\nabla_{z}
\end{equation}
and
\begin{equation}\label{eq:operators3}
\sigma^{-\frac{1}{2}}\w_{l}\cdot\nabla_{\zeta},\quad\tau^{-\frac{1}{2}}\w_{l}\cdot\nabla_{\theta},\quad\tau^{\frac{1}{2}}\w_{l}\cdot\nabla_{z}.
\end{equation}
Throughout, we omit the dependence on $(z,\zeta,\theta)$ in the functions $f_{j,1}$ and $f_{j,2}$.

\begin{lemma}\label{lem:fij}
For each $M\in\N$ there exists a constant $C'\geq0$, dependent on $M$, $a$ and $\Phi$, such that the following assertions hold for all $(z,\zeta,\theta)\in\supp(a_{1})\cup\supp(a_{2})$ with $\theta\neq0$. 
\begin{enumerate}[(1)]
\item\label{it:fij1} Suppose that $c\tau\leq \sigma\leq\frac{1}{c}\tau$ and $\sigma,\tau<1$, and let $L_{1},\ldots, L_{M}$ be operators as in \eqref{eq:operators1}.\,Then $|L_{1}\ldots L_{M}(a_{1})|+|L_{1}\ldots L_{M}(a_{2})|\leq C'$ and $|f_{j,2}|+|L_{1}\ldots L_{M}(f_{j,2})|\leq C'$ for $j\in\{1,2,3,4\}$. Let $1\leq m\leq M$ and $f\in\{f_{1,1},f_{2,1},f_{3,1},f_{4,1}\}$. Then $L_{m}(f)=gf$ for some function $g$ such that $|g|+|L_{1}\ldots L_{M}(g)|\leq C'$. Moreover, $|f_{5,2}|\leq C'$ and $(\hat{\theta}\cdot\nabla_{\theta}+|\theta|^{-1}\zeta\cdot\nabla_{\zeta})f_{5,1}=(\hat{\theta}\cdot\nabla_{\theta}+|\theta|^{-1}\zeta\cdot\nabla_{\zeta})f_{5,2}=0$.
\item\label{it:fij2} Suppose that $\sigma\leq\frac{1}{c}\tau$ and $\sigma,\tau<1$, and let $L_{1},\ldots, L_{M}$ be operators as in \eqref{eq:operators2}. Then $|L_{1}\ldots L_{M}(a_{1})|+|L_{1}\ldots L_{M}(a_{2})|\leq C'$ and $|f_{j,2}|+|L_{1}\ldots L_{M}(f_{j,2})|\leq C'$ for $j\in\{1,2,3,6\}$. Let $1\leq m\leq M$ and $f\in \{f_{1,1},f_{2,1},f_{3,1},f_{6,1}\}$. Then $L_{m}(f)=gf$ for some function $g$ such that $|g|+|L_{1}\ldots L_{M}(g)|\leq C'$. 
\item\label{it:fij3} Suppose that $\sigma\leq\frac{1}{c}\tau$ and $\max(\sigma,\tau)\in[1,e]$, and let $L_{1},\ldots, L_{M}$ be operators as in \eqref{eq:operators2}. Then $|L_{1}\ldots L_{M}(a_{1})|+|L_{1}\ldots L_{M}(a_{2})|\leq C'$ and $|L_{1}\ldots L_{M}(f_{j,2})|\leq C'$ for $j\in\{1,2,6\}$. Let $1\leq m\leq M$ and $f\in \{f_{1,1},f_{2,1},f_{6,1}\}$. Then $L_{m}(f)=gf$ for some function $g$ such that $|g|+|L_{1}\ldots L_{M}(g)|\leq C'$.
\item\label{it:fij4} Suppose that $\sigma>\frac{1}{c}\tau$, and let $L_{1},\ldots, L_{M}$ be operators as in \eqref{eq:operators3}. Then $|L_{1}\ldots L_{M}(a_{1})|+|L_{1}\ldots L_{M}(a_{2})|\leq C'$ and $|L_{1}\ldots L_{M}(f_{j,2})|\leq C'$ for $j\in\{1,2,6\}$. Let $1\leq m\leq M$ and $f\in \{f_{1,1},f_{2,1},f_{6,1}\}$. Then $L_{m}(f)=gf$ for some function $g$ such that $|g|+|L_{1}\ldots L_{M}(g)|\leq C'$.
\end{enumerate}
\end{lemma}
\begin{proof} 
Using induction, the proof is a straightforward but cumbersome computation. We just draw attention to a few key points.

Throughout, one uses the properties of $\psi_{\w,\sigma}$ and $\wt{\psi}_{\nu,\tau}$ from Lemma \ref{lem:packetbounds}. Combined with the assumption that $\rho$ is homogeneous of degree zero in the $\theta$ variable for $|\theta|\geq \frac{1}{2}$, these properties imply in particular that $a_{1}$ and $a_{2}$ are bounded independently of $\sigma$ and $\tau$ under applications of $\tau^{-1}\w\cdot\nabla_{\zeta}$ and $\tau^{-1}(\hat{\theta}\cdot\nabla_{\theta}+|\theta|^{-1}\zeta\cdot\nabla_{\zeta})$ if $\sigma,\tau<1$. For the corresponding statements in \eqref{it:fij2} and \eqref{it:fij3} one uses that $\tau^{-1}\lesssim \sigma^{-1}$. For $\sigma\geq 1$ or $\tau\geq 1$ one also uses that $\rho(z,\theta)=1$ for $|\theta|$ small, so that $\theta$-derivatives of $\rho$ vanish near zero.

It is useful to note that $f_{j,1}$, for $1\leq j\leq 4$, is a negative power of a function that is homogeneous of order zero in the joint variable $(\zeta,\theta)$, and partial derivatives of such a function are homogeneous of the appropriate negative order. Moreover, $(\hat \theta \cdot \nabla_\theta + |\theta|^{-1} \zeta \cdot \nabla_\zeta)f_{j,1}=0$ for $1\leq j\leq 5$.

When dealing with $f_{j,1}$ and $f_{j,2}$ for $j\in\{2,3,6\}$, one uses the bounds for $\Phi$ from \eqref{eq:phibounds}. Note in particular that one only uses bounds for partial derivatives of $\Phi$ of order at least two. 

The restriction that $k\neq 1$ in $\sqrt{\tau}\w_{k}\cdot\nabla_{z}$ is relevant only for $f_{4,1}$ and $f_{4,2}$, and ensures that 
\[
\sqrt{\tau}\w_{k}\cdot\nabla_{z}f_{4,1}=\sqrt{\tau}\w_{k}\cdot\nabla_{z}f_{4,2}=0.
\]
The projection onto the hyperplane orthogonal to $\w$ in $f_{3,1}$ and $f_{3,2}$ ensures that 
\[
\tau^{-1}\w\cdot\nabla_{\zeta}f_{3,1}=\tau^{-1}\w\cdot\nabla_{\zeta}f_{3,2}=0.
\]

For \eqref{it:fij3} one has $\tau\in[c,e]$, so that the factors of $\tau$ in $f_{j,1}$ and $f_{j,2}$ are irrelevant. Moreover, here $\theta$-derivatives of $f_{2,1}$, $f_{2,2}$, $f_{6,1}$ and $f_{6,2}$ are bounded by the assumption that either $\partial^{2}_{\theta\theta}\Phi=0$ or that $a(z,\theta)=0$ for $|\theta|$ small.
\end{proof}

We are now ready to prove the required estimates, which we split into several cases.

\subsubsection{Case 1} Here we we assume that $c\tau\leq \sigma\leq\frac{1}{c}\tau$, $\sigma,\tau<1$ and $(y, \nu)\in\dom(\hchi)$, and we consider $K^{1}_{\sigma,\tau}(x,y)$. 

We apply (in any order) $D_1$ $2N+2n$ times, $D_2$ and $D_3$ each $2N$ times, and $D_4$ $2N+4n$ times to $e^{i\Pi}$, and then we integrate by parts in the representation for $K^{1}_{\sigma,\tau}(x,y)$ from \eqref{eq:kernelrep}. Each time we apply $D_j$, we gain a decaying factor $f_{j,1}$, and this factor persists under repeated integration by parts, due to Lemma \ref{lem:fij}. \emph{After} we have applied these operators, we then apply $D_5$ sufficiently many times to $e^{i\Pi}$ and integrate by parts. Since $D_{5}f_{5,j}=f_{5,1}^{2}f_{5,j}$ for $j\in\{1,2\}$, and because we do not apply the other operators to $f_{5,1}$ or $f_{5,2}$, this allows us to also incorporate the decaying factors $f_{5,1}$. More precisely, using part \eqref{it:fij1} of Lemma \ref{lem:fij}, we can express $K^{1}_{\sigma,\tau}(x,y)$ as a sum of terms of the form
\begin{equation}\label{eq:resultingkernel}
(\sigma\tau)^{-\frac{n-1}{4}}\int_{\R^{3n}}e^{i\Pi(z,\zeta,\theta)}(f_{1,1} f_{2,1} f_{3,1} f_{4,1} f_{5,1} \big)^{2N}f_{1,1}^{2n}f_{4,1}^{4n} \wt{a_{1}}(z,\zeta,\theta)\ud\theta\ud z\ud\zeta,
\end{equation}
where $\wt{a_{1}}$ is a bounded function such that $\supp(\wt{a}_{1})\subseteq\supp(a_{1})$, with a uniform bound that depends only on $n$, $N$, $a$ and $\Phi$. 

We now claim that 
\begin{equation}\label{eq:algebraic1}
1 + \tau^{-1} d((x,\w),\hchi(y,\nu))^{2} \lesssim\big( f_{1,1} f_{2,1} f_{3,1} f_{4,1} f_{5,1}\big)^{-2}(z,\zeta,\theta)
\end{equation}
for all $(z,\zeta,\theta)\in\supp(a_{1})$. Taking the claim for granted momentarily, we can conclude the proof. Recall that $\zeta\in\supp(\psi_{\w,\sigma})$ and $\theta\in\supp(\wt{\psi}_{\nu,\tau})$ for $(z,\zeta,\theta)\in\supp(a_{1})$, and that $V(\supp(\psi_{\w,\sigma}))\lesssim \sigma^{-\frac{n+1}{2}}\eqsim \tau^{-\frac{n+1}{2}}$ and $V(\supp(\wt{\psi}_{\nu,\tau}))\lesssim \tau^{-\frac{n+1}{2}}$. Hence, using \eqref{eq:algebraic1} and a substitution, we can bound each term as in \eqref{eq:resultingkernel} in absolute value by a multiple of 
\begin{align*}
&(\sigma\tau)^{-\frac{n-1}{4}}(1+\tau^{-1}d((x,\w),\hchi(y,\nu))^{2})^{-N}\int_{\supp(a_{1})}f_{1,1}^{2n}f_{4,1}^{4n}\ud\theta\ud\zeta\ud z\\
&\lesssim \tau^{-\frac{3n+1}{2}}(1+\tau^{-1}d((x,\w),\hchi(y,\nu))^{2})^{-N}\int_{\Rn}(1+\tau^{-1}(|\lb\w,x-z\rb|+|x-z|^{2}))^{-2n}\ud z\\
&\lesssim \tau^{-n}(1+\tau^{-1}d((x,\w),\hchi(y,\nu))^{2})^{-N},
\end{align*}
which suffices since $\Upsilon(\tfrac{\sigma}{\tau})^{-1}\lesssim 1$ by assumption.

It remains to prove \eqref{eq:algebraic1}. By assumption, $\hchi^{-1}$ is Lipschitz. Hence
\begin{align*}
&\tau^{-1}d((x,\w),\chi(y,\nu))^{2}\lesssim \tau^{-1}\big(d\big((x,\w),\big(z,\tfrac{\nabla_{z}\Phi(z,\hat{\theta})}{|\nabla_{z}\Phi(z,\hat{\theta})|}\big)\big)^{2}\!+d\big(\big(z,\tfrac{\nabla_{z}\Phi(z,\hat{\theta})}{|\nabla_{z}\Phi(z,\hat{\theta})|}\big),\hchi(y,\nu)\big)\big)^{2}\\
&\lesssim \tau^{-1}d\big((x,\w),\big(z,\tfrac{\nabla_{z}\Phi(z,\hat{\theta})}{|\nabla_{z}\Phi(z,\hat{\theta})|}\big)\big)^{2}+\tau^{-1}d((\nabla_{\theta}\Phi(z,\hat{\theta}),\hat{\theta}),(y,\nu))^{2}.
\end{align*}
Since $f_{j,1}\leq1$ for all $1\leq j\leq 5$, we need only bound both terms on the second line by a multiple of $\big( f_{1,1} f_{2,1} f_{3,1} f_{4,1} f_{5,1} \big)^{-2}$. 

Using the equivalent expression for $d$ from \eqref{eq:qm-equiv}, one has
\[
\tau^{-1}d\big((x,\w),\big(z,\tfrac{\nabla_{z}\Phi(z,\hat{\theta})}{|\nabla_{z}\Phi(z,\hat{\theta})|}\big)\big)^{2}\lesssim \tau^{-1}\big(|x-z|^2   + |\langle \omega, x-z \rangle| + \big|\tfrac{\nabla_z \Phi(z, \hat\theta)}{|\nabla_z \Phi(z, \hat\theta)|}-\w\big|^2\big).
\]
The first and the second term on the right-hand side are bounded by $f_{1,1}^{-2}$ and $f_{4,1}^{-2}$, respectively. For the third term, note that
\[
\big|\tfrac{\nabla_{z}\Phi(z,\hat{\theta})}{|\nabla_{z}\Phi(z,\hat{\theta})|}-\w\big|\lesssim \big|P_{\w}^{\perp}\big(\tfrac{\nabla_{z}\Phi(z,\hat{\theta})}{|\nabla_{z}\Phi(z,\hat{\theta})|}-\w\big)\big|=\frac{|P_{\w}^{\perp}(\nabla_{z}\Phi(z,\hat{\theta}))|}{|\nabla_{z}\Phi(z,\hat{\theta})|},
\]
because $\big|\tfrac{\nabla_{z}\Phi(z,\hat{\theta})}{|\nabla_{z}\Phi(z,\hat{\theta})|}-\w\big|\leq 1$ on $\supp(a_{1})$. Moreover, the homogeneity and boundedness assumptions on $\Phi$ show that $|\nabla_{z}\Phi(z,\hat{\theta})|=|\partial_{z\theta}^{2}\Phi(z,\hat{\theta})\cdot \hat{\theta}|\eqsim 1$. Since $\zeta\in\supp(\psi_{\w,\sigma})$ and $\theta\in\supp(\wt{\psi}_{\nu,\tau})$, one obtains
\begin{align*}
&\tau^{-1}\big|\tfrac{\nabla_{z}\Phi(z,\hat{\theta})}{|\nabla_{z}\Phi(z,\hat{\theta})|}-\w\big|^{2}\lesssim \tau^{-1}|P_{\w}^{\perp}(\nabla_{z}\Phi(z,\hat{\theta}))|^{2}\\
&\lesssim \tau^{-1}\big(\big|P_{\w}^{\perp}(\nabla_{z}\Phi(z,\hat{\theta}))-\tfrac{\zeta}{|\theta|}\big|^{2}+\big(\tfrac{|\zeta|}{|\theta|}\big)^{2}|P_{\w}^{\perp}(\hat{\zeta})|^{2}\big)\\
&\lesssim 1+\tau^{-1}\big|P_{\w}^{\perp}(\nabla_{z}\Phi(z,\hat{\theta}))-\tfrac{\zeta}{|\theta|}\big|^{2}=f_{3,1}^{-2}
\end{align*}
from the support properties of $\psi_{\w,\sigma}$ and $\wt{\psi}_{\nu,\tau}$. 

Next, we again use the equivalent expression for $d$:
\[
\tau^{-1}d((\nabla_{\theta}\Phi(z,\hat{\theta}),\hat{\theta}),(y,\nu))^{2}\lesssim  \tau^{-1}\big(| \nabla_\theta \Phi(z, \hat\theta)-y|^2 + |\lb\nu,\nabla_\theta \Phi(z, \hat\theta)-y\rb | + |\nu - \hat \theta|^2\big).
\]
The first term on the right-hand side is bounded by $f_{2, 1}^{-2}$, and the third term by a multiple of $1$ due to the support properties of $\wt{\psi}_{\tau, \nu}$. For the second term, write
\[
\lb \nu,\nabla_\theta \Phi(z, \hat\theta)-y\rb =  \lb\nu - \hat \theta,\nabla_\theta \Phi(z, \hat\theta)-y\rb + \frac{\Pi(z, \zeta, \theta)}{|\theta|} 
- (x-z) \cdot \frac{\zeta}{|\theta|}.
\]
Then $\tau^{-1}|\lb \nu,\nabla_\theta \Phi(z, \hat\theta)-y\rb|$ is bounded from above by 
\begin{align*}
&\tau^{-1}\Big(|\nu - \hat \theta|\,|\nabla_\theta \Phi(z, \hat\theta)-y| +\frac{|\Pi(z, \zeta, \theta)|}{|\theta|}+\frac{|\zeta|}{|\theta|}|\lb\w,x-z\rb| + \frac{|\zeta|}{|\theta|}|\hat{\zeta} - \omega|\,|x-z|\Big)\\
&\lesssim \tau^{-1/2}|\nabla_\theta \Phi(z, \hat\theta)-y|+(\tau|\theta|)^{-1}|\Pi(z,\zeta,\theta)|+\tau^{-1}|\lb\w,x-z\rb|+\tau^{-1/2}|x-z|\\
&\lesssim (f_{2,1}f_{5,1}f_{4,1}f_{1,1})^{-1}, 
\end{align*}
where we again used the support properties of $\psi_{\w,\sigma}$ and $\wt{\psi}_{\nu,\tau}$. This proves \eqref{eq:algebraic1} and concludes the proof of case 1.

\subsubsection{Case 2}

Here we we assume that $c\tau\leq \sigma\leq\frac{1}{c}\tau$, $\sigma,\tau<1$ and $(y,\nu)\in\dom(\hchi)$, and we consider $K^{2}_{\sigma,\tau}(x,y)$. 

We proceed in an analogous manner as before, except that we need to use different operators to obtain the required decay factors. Indeed, $f_{3,1}$ does not yield the necessary decay, given that $\nabla_{z}\Phi(z,\hat{\theta})$ may be a negative multiple of $\w$ on $\supp(a_{2})$. Note that, since $(y,\nu)\in \dom(\hchi)$, there exists a $\wt{y}\in\Rn$ such that $(\wt{y},\nu)\in\dom(\Phi)$ and $(y,\nu)=(\nabla_{\theta}\Phi(\wt{y},\nu),\nu)$. 

We first suppose additionally that $\big|\frac{\nabla_{z}\Phi(\wt{y},\nu)}{|\nabla_{z}\Phi(\wt{y},\nu)|}-\w\big|\geq \frac{1}{4}$. Then one has
\[
d((x,\w),\hchi(y,\nu))=d\big((x,\w),\big(\wt{y},\tfrac{\nabla_{z}\Phi(\wt{y},\nu)}{|\nabla_{z}\Phi(\wt{y},\nu)|}\big)\big)\gtrsim 1.
\]
Since $\hchi^{-1}$ is Lipschitz, one now obtains 
\begin{align*}
&1\lesssim d((x,\w),\hchi(y,\nu))^{2}\lesssim d\big((x,\w),\big(z,\tfrac{\nabla_{z}\Phi(z,\hat{\theta})}{|\nabla_{z}\Phi(z,\hat{\theta})|}\big)\big)^{2}+d\big(\big(z,\tfrac{\nabla_{z}\Phi(z,\hat{\theta})}{|\nabla_{z}\Phi(z,\hat{\theta})|}\big),\hchi(y,\nu)\big)^{2}\\
&\lesssim d\big((x,\w),\big(z,\tfrac{\nabla_{z}\Phi(z,\hat{\theta})}{|\nabla_{z}\Phi(z,\hat{\theta})|}\big)\big)^{2}+d(\nabla_{\theta}\Phi(z,\hat{\theta}),\hat{\theta}),(y,\nu)\big)^{2}
\end{align*}
for all $(z,\zeta,\theta)\in\supp(a_{2})$. In particular, the quantity on the second line is uniformly bounded from below by some $\delta>0$, and therefore one may ignore the anisotropic terms in the distance:
\begin{equation}\label{eq:dbound}
\begin{aligned}
&d((x,\w),\hchi(y,\nu))^{2}\lesssim d\big((x,\w),\big(z,\tfrac{\nabla_{z}\Phi(z,\hat{\theta})}{|\nabla_{z}\Phi(z,\hat{\theta})|}\big)\big)^{2}+d(\nabla_{\theta}\Phi(z,\hat{\theta}),\hat{\theta}),(y,\nu)\big)^{2}\\
&\lesssim |x-z|^{2}+\big|\w-\tfrac{\nabla_{z}\Phi(z,\hat{\theta})}{|\nabla_{z}\Phi(z,\hat{\theta})|}\big|^{2}+|\nabla_{\theta}\Phi(z,\hat{\theta})-y|^{2}+|\hat{\theta}-\nu|^{2},
\end{aligned}
\end{equation}
a general fact that is easily checked by splitting into the cases where the anisotropic terms are either bigger or smaller than a fraction of $\delta$. Next, one has $1+\tau^{-1}|\hat{\theta}-\nu|^{2}\lesssim 1$ on $\supp(a_{2})$, and $\tau^{-1}|x-z|^{2}+\tau^{-1}|\nabla_{\theta}\Phi(z,\hat{\theta})-y|^{2}\leq (f_{1,1}f_{2,1})^{-2}$. Also, since $\big|\frac{\nabla_{z}\Phi(z,\hat{\theta})}{|\nabla_{z}\Phi(z,\hat{\theta})|}-\w\big|\geq \frac{1}{2}$ on $\supp(a_{2})$, one has $\big|\frac{\nabla_{z}\Phi(z,\hat{\theta})}{|\nabla_{z}\Phi(z,\hat{\theta})|}-\alpha\w\big|\geq \frac{1}{4}$ for any $\alpha\geq0$. In particular, using the bounds on $\Phi$ one obtains
\begin{align*}
&\tau^{-1} \big|\tfrac{\nabla_{z}\Phi(z,\hat{\theta})}{|\nabla_{z}\Phi(z,\hat{\theta})|}-\w\big|^{2} \lesssim \tau^{-1} \big|\tfrac{\nabla_{z}\Phi(z,\hat{\theta})}{|\nabla_{z}\Phi(z,\hat{\theta})|}-\tfrac{|\zeta|}{|\theta||\nabla_{z}\Phi(z,\hat{\theta})|} \w\big|^2\\
&\lesssim \tau | \zeta - \nabla_{z}\Phi(z,\theta)|^2 + \tau |\zeta - |\zeta| \omega |^2\lesssim 1+\tau|\zeta - \nabla_{z}\Phi(z,\theta)|^2= f_{6,1}^{-2}
\end{align*}
on $\supp(a_{2})$. By combining all this with \eqref{eq:dbound}, one has $1+\tau^{-1}d((x,\w),\hchi(y,\nu))^{2}\leq (f_{1,1}f_{2,1}f_{6,1})^{-2}$. Now, as in case 1, one integrates by parts with respect to $D_{1}$, $D_{2}$ and $D_{6}$ in the representation \eqref{eq:kernelrep} for $K^{2}_{\sigma,\tau}(x,y)$. Lemma \ref{lem:fij} deals with the additional terms that arise in this manner, and one obtains the required decay in $\tau$ by using that $d((x,\w),\chi(y,\nu))\gtrsim 1$. This ensures in particular that do not have to apply $D_{4}$ for the anisotropic substitution procedure that was used in case 1.

To conclude the proof of case 2, we need to deal with the case where $\big|\frac{\nabla_{z}\Phi(\wt{y},\nu)}{|\nabla_{z}\Phi(\wt{y},\nu)|}-\w\big|<\frac{1}{4}$. This is done in an analogous manner, using now that $\big|\frac{\nabla_{z}\Phi(\wt{y},\nu)}{|\nabla_{z}\Phi(\wt{y},\nu)|}-\frac{\nabla_{z}\Phi(z,\hat{\theta})}{|\nabla_{z}\Phi(z,\hat{\theta})|}\big|>\frac{1}{4}$ on $\supp(a_{2})$ to ignore the anisotropic terms and to see that $\tau\gtrsim (f_{1,1}f_{2,1}f_{6,1})^{2}$.

\subsubsection{Case 3}

Here we assume that $c\tau\leq \sigma\leq\frac{1}{c}\tau$ and $\sigma,\tau<1$, that either $(y, \nu)\notin\dom(\hchi)$ or $(x,\w)\notin\ran(\hchi)$, and we consider $K^{1}_{\sigma,\tau}(x,y)$ and $K^{2}_{\sigma,\tau}(x,y)$.

First note that, if $d((y,\nu),A)\geq\veps$ for some $\veps>0$ independent of $(y,\nu)$, then as in case 2 one may ignore the anisotropic terms in the distance to obtain
\begin{align*}
&\tau^{-1}\lesssim 1+\tau^{-1}d((y,\nu),A)^{2}\leq 1+\tau^{-1}d((y,\nu),(\nabla_{\theta}\Phi(z,\hat{\theta}),\hat{\theta}))^{2}\\
&\eqsim 1+\tau^{-1}(|\nabla_{\theta}\Phi(z,\theta)-y|^{2}+|\hat{\theta}-\nu|^{2})\lesssim 1+\tau^{-1}|\nabla_{\theta}\Phi(z,\theta)-y|^{2}=f_{2,1}^{-2}
\end{align*}
for all $(z,\zeta,\theta)\in\supp(a_{1})\cup\supp(a_{2})$. On the other hand, if $d((y,\nu),A)<\veps$ then clearly $1+\tau^{-1}d((y,\nu),A)^{2}\lesssim \tau^{-1}$. 

Similarly, if $d((x,\w),\chi(A))\geq\veps$ then
\begin{align*}
&\tau^{-1}\lesssim 1+\tau^{-1}d((x,\w),\hchi(A))^{2}\leq 1+\tau^{-1}d\big((x,\w),\big(z,\tfrac{\nabla_{z}\Phi(z,\hat{\theta})}{|\nabla_{z}\Phi(z,\hat{\theta})|}\big)\big)^{2}\\
&\eqsim 1+\tau^{-1}\big(|x-z|^{2}+\big|\w-\tfrac{\nabla_{z}\Phi(z,\hat{\theta})}{|\nabla_{z}\Phi(z,\hat{\theta})|}\big|^{2}\big)\leq f_{1,1}^{-2}\big(1+\tau^{-1}\big|\w-\tfrac{\nabla_{z}\Phi(z,\hat{\theta})}{|\nabla_{z}\Phi(z,\hat{\theta})|}\big|^{2}\big).
\end{align*}
Now, as in case 1, the term in brackets can be bounded by a multiple of $f_{3,1}^{-2}$ on $\supp(a_{1})$. And as in case 2, it can be bounded by a multiple of $f_{6,1}^{-2}$ on $\supp(a_{2})$. In both cases one finds
\[
\tau^{-1}\lesssim 1+\tau^{-1}d((x,\w),\chi(A))^{2}\lesssim (f_{1,1}f_{3,1}f_{6,1})^{-2}.
\]
Again, if $d((x,\w),\hchi(A))<\veps$ then $1+\tau^{-1}d((x,\w),\hchi(A))^{2}\lesssim \tau^{-1}$.

Next, we recall that $\dom(\hchi)$ is an $\epsilon$-neighborhood of $A$ for some $\veps>0$, by \eqref{it:phase4} in Definition \ref{def:operator}. Since $\hchi$ is bi-Lipschitz, we may also assume that $\ran(\hchi)$ is an $\veps$-neighborhood of $\hchi(A)$. 

Now suppose that $(y,\nu)\notin\dom(\chi)$ and $d((x,\w),\chi(A))<\veps$. Then $d((y,\nu),A)\geq\veps$, so as above we obtain
\[
(1+\tau^{-1}d((x,\w),\hchi(A))^{2})(1+\tau^{-1}d((y,\nu),A)^{2})\lesssim \tau^{-1}f_{2,1}^{-2}\lesssim f_{2,1}^{-4}
\] 
on $\supp(a_{1})\cup\supp(a_{2})$. Hence one can directly integrate by parts with respect to $D_{1}$ and $D_{2}$ in the representation \eqref{eq:kernelrep} for $K^{1}_{\sigma,\tau}(x,y)$ and $K^{2}_{\sigma,\tau}(x,y)$, by Lemma \ref{lem:fij}. This allows one to write, for $j=1,2$ and $m\in\N$ large enough,
\begin{align*}
&|K^{j}_{\sigma,\tau}(x,y)|\lesssim (\sigma\tau)^{-\frac{n-1}{4}}\int_{\supp(a_{j})}f_{1,1}^{m}f_{2,1}^{4N}\ud\theta\ud z\ud\zeta\\
&\lesssim(1+\tau^{-1}d((x,\w),\hchi(A))^{2})(1+\tau^{-1}d((y,\nu),A)^{2})(\sigma\tau)^{-\frac{n-1}{4}}\int_{\supp(a_{j})}f_{1,1}^{m}\ud\theta\ud z\ud\zeta\\
&\lesssim (1+\tau^{-1}d((x,\w),\hchi(A))^{2})(1+\tau^{-1}d((y,\nu),A)^{2}),
\end{align*}
as required.

Next, if $(y,\nu)\notin\dom(\hchi)$ and $d((x,\w),\hchi(A))\geq\veps$, then
\[
(1+\tau^{-1}d((x,\w),\hchi(A))^{2})(1+\tau^{-1}d((y,\nu),A)^{2})\lesssim (f_{1,1}f_{2,1}f_{3,1}f_{6,1})^{-2}
\] 
on $\supp(a_{1})\cup\supp(a_{2})$. Again, integrating by parts with respect to $D_{1}$, $D_{2}$, $D_{3}$ and $D_{6}$ in \eqref{eq:kernelrep} and using Lemma \ref{lem:fij} then concludes the proof of case 4 for $(y,\nu)\notin\dom(\hchi)$.

Next, suppose that $(x,\w)\notin\ran(\hchi)$. If $d((y,\nu),A)<\veps$, then
\begin{align*}
(1+\tau^{-1}d((x,\w),\hchi(A))^{2})(1+\tau^{-1}d((y,\nu),A)^{2})&\lesssim \tau^{-1}(f_{1,1}f_{3,1}f_{6,1})^{-2}\\
&\lesssim (f_{1,1}f_{3,1}f_{6,1})^{-4}
\end{align*}
on $\supp(a_{1})\cup\supp(a_{2})$. Hence integrating by parts with respect to $D_{1}$, $D_{3}$ and $D_{6}$ deals with this case. For $d((y,\nu),A)\geq\veps$ one may again integrate by parts with respect to $D_{1}$, $D_{2}$, $D_{3}$ and $D_{6}$, since 
\[
(1+\tau^{-1}d((x,\w),\hchi(A))^{2})(1+\tau^{-1}d((y,\nu),A)^{2})\lesssim (f_{1,1}f_{2,1}f_{3,1}f_{6,1})^{-2}
\]
on $\supp(a_{1})\cup\supp(a_{2})$. This concludes the proof of case 3.

\subsubsection{Case 4}

Here we we assume that $c\tau\leq \sigma\leq\frac{1}{c}\tau$ and $\max(\sigma,\tau)\in[1,e]$, and we consider both $K^{1}_{\sigma,\tau}(x,y)$ and $K^{2}_{\sigma,\tau}(x,y)$. 

First suppose in addition that $(y,\nu)\in\dom(\hchi)$. Since $\tau\in[c,e]$ and $S^{n-1}$ is compact, one can use that $\hchi^{-1}$ is Lipschitz to see that
\begin{align*}
\tau^{-1}d((x,\w),\hchi(y,\nu))^{2}&\lesssim d\big((x,\w),\big(z,\tfrac{\nabla_{z}\Phi(z,\hat{\theta})}{|\nabla_{z}\Phi(z,\hat{\theta})|}\big)\big)^{2}+d\big(\big(z,\tfrac{\nabla_{z}\Phi(z,\hat{\theta})}{|\nabla_{z}\Phi(z,\hat{\theta})|}\big),\hchi(y,\nu)\big)^{2}\\
&\lesssim d\big((x,\w),\big(z,\tfrac{\nabla_{z}\Phi(z,\hat{\theta})}{|\nabla_{z}\Phi(z,\hat{\theta})|}\big)\big)^{2}+d((\nabla_{\theta}\Phi(z,\hat{\theta}),\hat{\theta}),(y,\nu)\big)^{2}\\
&\lesssim 1+|x-z|^{2}+|\nabla_{\theta}\Phi(z,\hat{\theta})-y|^{2}\lesssim (f_{1,1}f_{2,1})^{-2}
\end{align*}
for $(z,\zeta,\theta)\in \supp(a_{1})\cup\supp(a_{2})$ with $\theta\neq0$. By combining this with part \eqref{it:fij3} of Lemma \ref{lem:fij}, we can integrate by parts with respect to $D_{1}$ and $D_{2}$ in \eqref{eq:kernelrep} for both $K^{1}_{\sigma,\tau}(x,y)$ and $K^{2}_{\sigma,\tau}(x,y)$. As in the other cases, additional factors of $f_{1,1}$ lead to an absolutely convergent integral, and here factors of $\tau$ can be ignored. 

Next, suppose that $(y,\nu)\notin \dom(\hchi)$. Again, one has
\[
1+\tau^{-1}d((x,\w),\hchi(A))^{2}\lesssim 1+d\big((x,\w),\big(z,\tfrac{\nabla_{z}\Phi(z,\hat{\theta})}{|\nabla_{z}\Phi(z,\hat{\theta})|}\big)\big)^{2}\lesssim 1+|x-z|^{2}\lesssim f_{1,1}^{-2}
\]
and
\[
1+\tau^{-1}d((y,\nu),A)^{2}\lesssim 1+d((y,\nu),(\nabla_{\theta}\Phi(z,\hat{\theta}),\hat{\theta}))^{2}\lesssim 1+|\nabla_{\theta}\Phi(z,\hat{\theta})-y|^{2}\lesssim f_{2,1}^{-2},
\]
so one can integrate by parts with respect to $D_{1}$ and $D_{2}$ as before.

\subsubsection{Case 5} 

Here we we assume that $\sigma\notin [c\tau,\frac{1}{c}\tau]$ and $(y,\nu)\in\dom(\chi)$, and we consider both $K^{1}_{\sigma,\tau}(x,y)$ and $K^{2}_{\sigma,\tau}(x,y)$. 

Recall that, by the choice of $c$ in \eqref{eq:defc}, one has
\begin{equation}\label{eq:sigmasmall}
|\nabla_{z}\Phi(z,\theta)-\zeta|\geq c'\max(\sigma^{-1},\tau^{-1})
\end{equation}
for all $(z,\zeta,\theta)\in\supp(a_{1})\cup\supp(a_{2})$ with $\theta\neq 0$. 

First suppose that $\sigma<c\tau$. Then, by \eqref{eq:sigmasmall}, one has
\[
f_{6,1}^{2}\lesssim \frac{1}{1+\tau\sigma^{-2}}\leq \sigma\frac{\sigma}{\tau}
\]
for all $(z,\zeta,\theta)\in\supp(a_{1})\cup\supp(a_{2})$. Since $\hchi^{-1}$ is Lipschitz, this yields
\begin{align*}
&\frac{\tau}{\sigma}(1+\sigma^{-1}d((x,\w),\hchi(y,\nu))^{2})\\
&\lesssim \frac{\tau}{\sigma}\big(1+\sigma^{-1}\big(d\big((x,\w),\big(z,\tfrac{\nabla_{z}\Phi(z,\hat{\theta})}{|\nabla_{z}\Phi(z,\hat{\theta})|}\big)\big)^{2}+d\big(\big(z,\tfrac{\nabla_{z}\Phi(z,\hat{\theta})}{|\nabla_{z}\Phi(z,\hat{\theta})|}\big),\hchi(y,\nu)\big)^{2}\big)\big)\\
&\lesssim \frac{\tau}{\sigma}\big(1+\sigma^{-1}\big(d\big((x,\w),\big(z,\tfrac{\nabla_{z}\Phi(z,\hat{\theta})}{|\nabla_{z}\Phi(z,\hat{\theta})|}\big)\big)^{2}+d((\nabla_{\theta}\Phi(z,\hat{\theta}),\hat{\theta}),(y,\nu))^{2}\big)\big)\\
&\lesssim \frac{\tau}{\sigma}(1+\sigma^{-1}(1+|x-z|^{2}+|\nabla_{\theta}\Phi(z,\hat{\theta})-y|^{2}))\lesssim (f_{1,1}f_{2,1}f_{6,1})^{-2}.
\end{align*}
It follows that we can integrate by parts sufficiently many times in \eqref{eq:kernelrep} with respect to $D_{1}$, $D_{2}$ and $D_{6}$, using parts \eqref{it:fij2} and \eqref{it:fij3} of Lemma \ref{lem:fij}, to obtain the required conclusion.

Next, suppose that $\sigma>\frac{1}{c}\tau$. The argument here is analogous. By \eqref{eq:sigmasmall}, one has $f_{6,1}^{2}\lesssim \tau^{3}$. Moreover, by Lemma \ref{lem:fij} \eqref{it:fij4}, each application of $D_{1}$ to $a_{1}$ or $a_{2}$ yields a factor which is bounded by $(\frac{\sigma}{\tau})^{1/2}$. By applying $D_{6}$ sufficiently many times, one can use $f_{6,1}$ to cancel out these factors as well.

\subsubsection{Case 6} 

Here we we assume that $\sigma\notin [c\tau,\frac{1}{c}\tau]$ and either $(y,\nu)\notin\dom(\hchi)$ or $(x,\w)\notin\ran(\hchi)$, and we consider both $K^{1}_{\sigma,\tau}(x,y)$ and $K^{2}_{\sigma,\tau}(x,y)$. 

This case is very similar to Case 5. For $\sigma<c\tau$, one again uses \eqref{eq:sigmasmall} to see that 
\[
\frac{\tau}{\sigma}(1+\sigma^{-1}d((x, \omega), \hchi(A))^2) \lesssim \frac{\tau}{\sigma}( 1 +\sigma^{-1}(1+|x-z|^2))\lesssim (f_{1,1}f_{6,1})^{-2}
\]
and
\[
\frac{\tau}{\sigma}(1+\sigma^{-1}d((y, \nu), A)^2)\lesssim\frac{\tau}{\sigma}(1+\sigma^{-1}(1+|\nabla_{\theta} \Phi(z, \theta)-y|^2))\lesssim (f_{2,1}f_{6,1})^{-2}
\]
for all $(z,\zeta,\theta)\in\supp(a_{1})\cup\supp(a_{2})$ with $\theta\neq0$. Now one simply integrates by parts with respect to $D_{1}$, $D_{2}$ and $D_{6}$. For $\sigma>\frac{1}{c}\tau$ the argument is similar, but one has to cancel out factors of $\frac{\sigma}{\tau}$ by applying $D_{6}$ sufficiently many times.

This concludes the proof of Theorem \ref{thm:offsingFIO}.

\begin{remark}\label{rem:bounds}
The constant $C$ in Theorem \ref{thm:offsingFIO} only depends on $T$ through finitely many of the $S^{0}_{\frac{1}{2},\frac{1}{2},1}$ norms of the symbol $a$, finitely many of the uniform bounds in \eqref{eq:phibounds} and \eqref{eq:phibounds2}, through the $\veps>0$ for which $\dom(\hchi)$ is an $\veps$-neighborhood of $A$, and through the Lipschitz constants of $\hchi$ and $\hchi^{-1}$. In particular, given a collection of operators $\{T_{i}\mid i\in I\}$ for which these quantities are uniformly bounded in $i\in I$, for each $N\geq0$ the constants $C_{i}$ associated with $T_{i}$ in Theorem \ref{thm:offsingFIO} are also uniformly bounded. Also, by Proposition \ref{prop:Lipschitz} and the proof of Lemma \ref{lem:philip}, the Lipschitz constants of $\hchi$ and $\hchi^{-1}$ depend only on \eqref{eq:phibounds} and \eqref{eq:phibounds2} if $\dom(\hchi)=\ran(\hchi)=\Sp$. By combining this with the remark after Theorem \ref{thm:tentbounded}, it follows that the operator norms of $WT_{i}V^{*}\in \La(T^{p}(\Sp))$ in Corollary \ref{cor:normalbounds} are uniformly bounded in $i$, for each $p\in[1,\infty]$.
\end{remark}

\begin{remark}\label{rem:symbols}
It should be noted that the symbol class $S^{0}_{\frac{1}{2},\frac{1}{2},1}(\R^{2n})$ arises naturally in the proof of Theorem \ref{thm:offsingFIO}. Indeed, even if the symbol $a$ itself is in fact of class $S^{0}(\R^{2n})$, then the symbol $\psi_{\w,\sigma}(\zeta)a(z,\theta)\wt{\psi}_{\nu,\tau}(\theta)$ behaves asymptotically like an $S^{0}_{\frac{1}{2},\frac{1}{2},1}$ symbol in the $\zeta$ and $\theta$ variables, due to the behavior of the wave packets $\psi_{\w,\sigma}$ and $\wt{\psi}_{\nu,\tau}$. In turn, the integration by parts procedure from the proof allows for the appropriate growth of $a$ in the $z$-variable. 
\end{remark}

\section{Hardy spaces for Fourier integral operators}\label{sec:Hardy spaces}

In this section we define Hardy spaces for Fourier integral operators and derive some of their basic properties, such as interpolation and duality theorems.

\subsection{Definition and basic properties}

For the rest of the article, fix wave packets $\psi_{\w,\sigma}$, for $w\in S^{n-1}$ and $\sigma>0$, as in Section \ref{sec:wavetransforms}, and the associated $r\in C^{\infty}_{c}(\Rn)$ and wave packet transform $W$. Recall that, for $f\in\Sw'(\Rn)$ and $(y,\nu,\sigma)\in\Spp$,
\[
Wf(y,\nu,\sigma)=W_{\sigma}f(y,\nu)=
\begin{cases}
\psi_{\nu,\sigma}(D)f(y)&\text{if }\sigma\in(0,1),\\
V(S^{n-1})^{-1/2}\ind_{[1,e]}(\sigma)r(D)f(y)&\text{if }\sigma\geq1.
\end{cases}
\]
We now define the Hardy spaces for Fourier integral operators through wave packet embeddings into the tent spaces $T^{p}(\Sp)$, $p\in[1,\infty]$.
 
\begin{definition}\label{def:spaces}
Let $p\in[1,\infty]$. Then $\HT^{p}_{FIO}(\Rn)$ consists of all $f\in\Sw'(\Rn)$ such that $Wf\in T^{p}(\Sp)$, with
\[
\|f\|_{\HT^{p}_{FIO}(\Rn)}:=\|Wf\|_{T^{p}(\Sp)}\quad(f\in \HT^{p}_{FIO}(\Rn)).
\]
\end{definition}

Note that, for $p\in[1,\infty)$ and $f\in\HT^{p}_{FIO}(\Rn)$, 
\begin{equation}\label{eq:pnorm}
\|f\|_{\Hps}=\Big(\int_{\Sp}\Big(\int_{0}^{e}\fint_{B_{\sqrt{\sigma}}(x,\w)}|W_{\sigma}f(y,\nu)|^{2}\ud y\ud\nu\frac{\ud\sigma}{\sigma}\Big)^{p/2}\ud x\ud\w\Big)^{1/p}.
\end{equation}
Similarly, for $f\in \HT^{\infty}_{FIO}(\Rn)$,
\begin{equation}\label{eq:inftynorm}
\|f\|_{\HT^{\infty}_{FIO}}=\sup_{(x,\w)\in\Sp}\sup_{B}\Big(\frac{1}{V(B)}\int_{T(B)}|W_{\sigma}f(y,\nu)|^{2}\ud y\ud\nu\frac{\ud\sigma}{\sigma}\Big)^{1/2},
\end{equation}
where the second supremum is taken over all balls $B\subseteq\Sp$ containing $(x,\w)$. 

\begin{remark}\label{rem:lowfrequency}
For $f\in\Sw'(\Rn)$ and $(x,\w)\in\Sp$, set
\begin{equation}\label{eq:squarefunction}
S f(x,\w):=\Big(\int_{0}^{1}\fint_{B_{\sqrt{\sigma}}(x,\w)}|\psi_{\nu,\sigma}(D)f(y)|^{2}\ud y\ud\nu\frac{\ud\sigma}{\sigma}\Big)^{1/2}\in[0,\infty].
\end{equation}
Let $q\in C^{\infty}_{c}(\Rn)$ be such that $q(\zeta)=1$ if $|\zeta|\leq 2$. It is shown in Corollary \ref{cor:equivalentnorm} that 
\begin{equation}\label{eq:alternativenorm}
\|f\|_{\Hp}\eqsim \|Sf\|_{L^{p}(\Sp)}+\|q(D)f\|_{L^{p}(\Rn)}
\end{equation}
for all $p\in[1,\infty)$, and similarly for $p=\infty$. Hence most of the information contained in \eqref{eq:pnorm} and \eqref{eq:inftynorm} pertains to the high frequencies of $f$, and $\Hp$ is a \emph{local} Hardy space. We could equally well have defined the $\Hp$-norm using \eqref{eq:alternativenorm}, but for technical purposes it turns out to be more convenient to include the low frequencies of $f$ in the square functions in \eqref{eq:pnorm} and \eqref{eq:inftynorm}.
\end{remark}

To derive some basic properties of these spaces we use results from Sections \ref{sec:wavetransforms} and \ref{sec:offsingFIO}.

\begin{proposition}\label{prop:Banach}
Let $p\in[1,\infty]$. Then $\|\cdot\|_{\HT^{p}_{FIO}(\Rn)}$ is a norm on $\HT^{p}_{FIO}(\Rn)$, $W^{*}\in \La(T^{p}(\Sp),\HT^{p}_{FIO}(\Rn))$, and 
\[
W^{*}:T^{p}(\Sp)/\ker(W^{*})\to \HT^{p}_{FIO}(\Rn)  
\]
is an isomorphism. In particular, $\HT^{p}_{FIO}(\Rn)$ is a Banach space.
\end{proposition}
\begin{proof}
For the first statement, note that if $f\in\Sw'(\Rn)$ is such that 
\[
\|f\|_{\HT^{p}_{FIO}(\Rn)}=\|Wf\|_{T^{p}(\Sp)}=0,
\]
then $f=W^{*}Wf=0$. Next, it follows from Corollary \ref{cor:normalbounds}, with $T$ the identity operator, that $WW^{*}\in\La(T^{p}(\Sp)$. Hence 
\[
\|W^{*}F\|_{\HT^{p}_{FIO}(\Rn)}=\|WW^{*}F\|_{T^{p}(\Sp)}\lesssim \|F\|_{T^{p}(\Sp)}
\] 
for all $F\in T^{p}(\Sp)$. For the third statement, it suffices to show that
\[
\|f\|_{\HT^{p}_{FIO}(\Rn)}\eqsim \inf\{\|F\|_{T^{p}(\Sp)}\mid F\in T^{p}(\Sp), W^{*}F=f\}
\]
for all $f\in \HT^{p}_{FIO}(\Rn)$. Since $Wf\in T^{p}(\Sp)$ and $f=W^{*}Wf$, one has 
\[
\|f\|_{\HT^{p}_{FIO}(\Rn)}\geq \inf\{\|F\|_{T^{p}(\Sp)}\mid F\in T^{p}(\Sp), W^{*}F=f\}.
\]
On the other hand, for $F\in T^{p}(\Sp)$ such that $W^{*}F=f$, Corollary \ref{cor:normalbounds} yields
\[
\|f\|_{\Hp}=\|WW^{*}F\|_{T^{p}(\Sp)}\lesssim \|F\|_{T^{p}(\Sp)}.
\]
Finally, $\ker(W^{*})\subseteq T^{p}(\Sp)$ is closed because $W^{*}\in\La(T^{p}(\Sp),\HT^{p}_{FIO}(\Rn))$, and therefore $T^{p}(\Sp)/\ker(W^{*})$ and $\HT^{p}_{FIO}(\Rn)$ are Banach spaces.
\end{proof}

Note that $\|\cdot\|_{\HT^{p}_{FIO}(\Rn)}$ depends on the choice of $\ph$ and $\Psi$, via the wave packet transform $W$. We now show that $\HT^{p}_{FIO}(\Rn)$ itself does not depend on the choice of $\ph$ and $\Psi$, up to norm equivalence. Consider functions $\wt{\ph}$ and $\wt{\Psi}$ with the same properties as $\ph$ and $\Psi$, and let $V$ be the associated wave packet transform, as in Section \ref{sec:offsingFIO}. For $p\in[1,\infty]$, let $\wt{\HT}^{p}_{FIO}(\Rn)$ consist of all $f\in\Sw'(\Rn)$ such that $Vf\in T^{p}(\Sp)$, with 
\[
\|f\|_{\wt{\HT}^{p}_{FIO}(\Rn)}:=\|Vf\|_{T^{p}(\Sp)}.
\]

\begin{proposition}\label{prop:independence}
Let $p\in[1,\infty]$. Then there exists a constant $C>0$ such that $\HT^{p}_{FIO}(\Rn)=\wt{\HT}^{p}_{FIO}(\Rn)$ and
\[
\frac{1}{C}\|f\|_{\wt{\HT}^{p}_{FIO}(\Rn)}\leq \|f\|_{\HT^{p}_{FIO}(\Rn)}\leq C\|f\|_{\wt{\HT}^{p}_{FIO}(\Rn)}
\]
for all $f\in \HT^{p}_{FIO}(\Rn)$. 
\end{proposition}
\begin{proof}
By symmetry, it suffices to show that $\|f\|_{\HT^{p}_{FIO}(\Rn)}\lesssim\|f\|_{\wt{\HT}^{p}_{FIO}(\Rn)}$ for all $f\in\wt{\HT}^{p}_{FIO}(\Rn)$. But this follows from Corollary \ref{cor:normalbounds}: 
\[
\|f\|_{\HT^{p}_{FIO}(\Rn)}=\|WV^{*}Vf\|_{T^{p}(\Sp)}\lesssim \|Vf\|_{T^{p}(\Sp)}=\|f\|_{\wt{\HT}^{p}_{FIO}(\Rn)}.\qedhere
\]
\end{proof}

Next, we consider the case where $p=2$, which is particularly simple.

\begin{lemma}\label{lem:L2}
One has $\HT^{2}_{FIO}(\Rn)=L^{2}(\Rn)$, with equal norms.
\end{lemma}
\begin{proof}
For each $f\in L^{2}(\Rn)\cup \HT^{2}_{FIO}(\Rn)$ one has, by Proposition \ref{prop:bddwavetransform},
\[
\|f\|_{\HT^{2}_{FIO}(\Rn)}=\|Wf\|_{T^{2}(\Sp)}=\|Wf\|_{L^{2}(\Spp)}=\|f\|_{L^{2}(\Rn)}.\qedhere
\]
\end{proof}

To conclude this subsection, we consider the Schwartz functions as a subset of $\Hp$.

\begin{proposition}\label{prop:Schwartzdense}
Let $p\in[1,\infty]$. Then $\Sw(\Rn)\subseteq\Hp$ continuously. If $p<\infty$, then $\Sw(\Rn)$ is dense in $\Hp$.
\end{proposition}
\begin{proof}
For the first statement note that $W:\Sw(\Rn)\to \Da(\Spp)\subseteq T^{p}(\Sp)$ continuously, by Proposition \ref{prop:bddwavetransform} and Lemma \ref{lem:distributions}. Next, suppose that $p\in[1,\infty)$ and let $f\in \Hp$. By Lemma \ref{lem:distributions}, there exists a sequence $(F_{k})_{k=1}^{\infty}\subseteq \Da(\Spp)$ such that $\lim_{k\to\infty}F_{k}=Wf$ in $T^{p}(\Sp)$. Moreover, $f_{k}:=W^{*}F_{k}\in \Sw(\Rn)$ for all $k\in\N$, by Proposition \ref{prop:bddwavetransform}. Now Proposition \ref{prop:Banach} yields
\[
\lim_{k\to\infty}f_{k}=\lim_{k\to\infty}W^{*}F_{k}= W^{*}Wf=f
\]
in $\Hp$.
\end{proof}

\subsection{Interpolation and duality}

We now relate the Hardy spaces for Fourier integral operators to each other. We first show that these spaces form a complex interpolation scale.

\begin{proposition}\label{prop:interpolation}
Let $p_{1},p_{2}\in[1,\infty]$ and $\theta\in[0,1]$, and let $p\in[1,\infty]$ be such that $\frac{1}{p}=\frac{1-\theta}{p_{1}}+\frac{\theta}{p_{2}}$. Then 
\[
[\HT^{p_{1}}_{FIO}(\Rn),\HT^{p_{2}}_{FIO}(\Rn)]_{\theta}=\Hp,
\]
with equivalent norms.
\end{proposition}
\begin{proof}
By definition, one has $W\in\La(\HT^{p_{j}}_{FIO}(\Rn),T^{p_{j}}(\Sp))$ for $j=1,2$. Now Lemma \ref{lem:tentint} shows that
\[
W:[\HT^{p_{1}}_{FIO}(\Rn),\HT^{p_{2}}_{FIO}(\Rn)]_{\theta}\to [T^{p_{1}}(\Sp),T^{p_{2}}(\Sp)]_{\theta}=T^{p}(\Sp)
\]
boundedly, and therefore $[\HT^{p_{1}}_{FIO}(\Rn),\HT^{p_{2}}_{FIO}(\Rn)]_{\theta}\subseteq \Hp$. On the other hand, by Lemma \ref{lem:tentint} and Proposition \ref{prop:Banach}, 
\[
W^{*}:T^{p}(\Sp)\to [\HT^{p_{1}}_{FIO}(\Rn),\HT^{p_{2}}_{FIO}(\Rn)]_{\theta}.
\]
So if $f\in\Hp$, then $f=W^{*}Wf\in [\HT^{p_{1}}_{FIO}(\Rn),\HT^{p_{2}}_{FIO}(\Rn)]_{\theta}$ and
\[
\|f\|_{[\HT^{p_{1}}_{FIO}(\Rn),\HT^{p_{2}}_{FIO}(\Rn)]_{\theta}}\lesssim \|Wf\|_{T^{p}(\Sp)}=\|f\|_{\Hp}.\qedhere
\]
\end{proof}

The following proposition on duality shows in particular that $\HT^{\infty}_{FIO}(\Rn)$ is a $\text{bmo}$ type space. 

\begin{proposition}\label{prop:duality}
Let $p\in[1,\infty)$. Then $\Hp^{*}=\HT^{p'}_{FIO}(\Rn)$, with equivalent norms, where the duality pairing is given by the distributional duality $(f,g)\to\lb f,g\rb$ for $f\in\Sw(\Rn)\subseteq\Hp$ and $g\in\HT^{p'}_{FIO}(\Rn)\subseteq\Sw'(\Rn)$. 
\end{proposition}
\begin{proof}
First let $g\in \HT^{p'}_{FIO}(\Rn)$. Then Proposition \ref{prop:bddwavetransform} and Lemma \ref{lem:tentdual} yield
\[
|\lb f,g\rb|=|\lb Wf,Wg\rb_{\Spp}|\lesssim \|Wf\|_{T^{p}(\Sp)}\|Wg\|_{T^{p'}(\Sp)}=\|f\|_{\Hps}\|g\|_{\HT^{p'}_{FIO}}
\]
for all $f\in\Sw(\Rn)$. By combining this with Proposition \ref{prop:Schwartzdense}, one obtains that $\HT^{p'}_{FIO}(\Rn)\subseteq\Hp^{*}$ continuously. 

Conversely, let $l\in \Hp^{*}$. Then $l\circ W^{*}\in T^{p}(\Sp)^{*}$, and Lemma \ref{lem:tentdual} yields a $G\in T^{p'}(\Sp)$ such that $l(W^{*}F)=\lb F,G\rb_{\Spp}$ for all $F\in\Da(\Spp)$. Set $g:=W^{*}G\in\HT^{p'}_{FIO}(\Rn)$. Then
\[
\lb f,g\rb=\lb Wf,G\rb_{\Spp}=l(W^{*}Wf)=l(f)
\]
for all $f\in\Sw(\Rn)$, and Corollary \ref{cor:normalbounds} yields
\begin{align*}
|\lb F,Wg\rb|&=|\lb WW^{*}F,G\rb|=|l(W^{*}F)|\leq \|l\|\,\|W^{*}F\|_{\Hps}\\
&=\|l\|\,\|WW^{*}F\|_{T^{p}(\Sp)}\lesssim\|l\|\,\|F\|_{T^{p}(\Sp)}
\end{align*}
for all $F\in \Da(\Spp)$. Now Lemmas \ref{lem:tentdual} and \ref{lem:distributions} show that 
\[
\|g\|_{\HT^{p'}_{FIO}(\Rn)}=\|Wg\|_{T^{p'}(\Sp)}\lesssim \|l\|_{\Hp^{*}}.\qedhere
\]
\end{proof}

\begin{remark}\label{rem:alternative}
We expect that the duality pairing in Proposition \ref{prop:duality} cannot be replaced by the pointwise pairing
\begin{equation}\label{eq:standardproduct}
(f,g)\mapsto \int_{\Rn}f(x)\overline{g(x)}\ud x
\end{equation}
for all $f\in\Hp$ and $g\in\HT^{p'}_{FIO}(\Rn)$. This is similar to the classical duality between $H^{1}(\Rn)$ and $\text{BMO}(\Rn)$, in the sense that \eqref{eq:standardproduct} need not be well defined for all $f\in H^{1}(\Rn)$ and $g\in BMO(\Rn)$. On the other hand, it is straightforward to check that the duality pairing is given in a pointwise sense by
\[
(f,g)\mapsto \int_{\Spp}Wf(x,\w,\sigma)\overline{Wg(x,\w,\sigma)}\ud x\ud\w\frac{\ud\sigma}{\sigma}
\]
for all $f\in \Hp$ and $g\in\HT^{p'}_{FIO}(\Rn)$. Moreover, Proposition \ref{prop:bddwavetransform} shows that the duality pairing is in fact given by \eqref{eq:standardproduct} if $f\in\Hp\cap L^{2}(\Rn)$ and $g\in \HT^{p'}_{FIO}(\Rn)\cap L^{2}(\Rn)$.
\end{remark}

\subsection{Boundedness of Fourier integral operators}\label{subsec:boundedFIOs}

Next, we show that our Hardy spaces are invariant under suitable oscillatory integral operators. 

\begin{theorem}\label{thm:FIObddHardy}
Let $T$ be one of the following:
\begin{enumerate}
\item\label{it:bdd1} A normal oscillatory integral operator of order $0$ and type $(\frac{1}{2},\frac{1}{2},1)$ with symbol $a$ and phase function $\Phi$ with the following properties. The induced contact transformation $\hchi$ satisfies $\dom(\hchi)=\ran(\hchi)=\Sp$, and either $(z,\theta)\mapsto\Phi(z,\theta)$ is linear in $\theta$ or there exists an $\veps>0$ such that $a(z,\theta)=0$ for all $(z,\theta)\in \R^{2n}$ with $|\theta|<\veps$;
\item\label{it:bdd2} A Fourier integral operator of order $0$ and type $(\rho,1-\rho,1)$, for $\rho\in(\frac{1}{2},1]$, associated with a local canonical graph, such that the Schwartz kernel of $T$ has compact support.
\end{enumerate}
Then $T\in \La(\Hp)$ for all $p\in[1,\infty]$.
\end{theorem}
\begin{proof}
By Corollaries \ref{cor:normalbounds} and \ref{cor:FIObounds},
\[
\|Tf\|_{\Hp}=\|WTW^{*}Wf\|_{T^{p}(\Sp)}\lesssim \|Wf\|_{T^{p}(\Sp)}=\|f\|_{\Hp}
\]
for all $f\in\Hp$.
\end{proof}

Let $\{T_{i}\mid i\in I\}$ be a collection of operators as in \eqref{it:bdd1}, with uniform bounds in $i$ for the $S^{0}_{\frac{1}{2},\frac{1}{2},1}(\R^{2n})$ norms of the symbols $a_{i}$ and the bounds for the phase functions $\Phi_{i}$ from Definition \ref{def:operator}. Then it follows from Remark \ref{rem:bounds} that one has $\sup_{i\in I}\|T_{i}\|_{\La(\Hp)}<\infty$ for each $p\in[1,\infty]$.

\begin{corollary}\label{cor:wavebdd}
The families $(e^{it\sqrt{-\Delta}})_{t\in\R}$, $(\cos(t\sqrt{-\Delta}))_{t\in\R}$ and $(\sin(t\sqrt{-\Delta}))_{t\in\R}$ are locally uniformly bounded on $\Hp$ for all $p\in[1,\infty]$, and strongly continuous on $\Hp$ for $p<\infty$.
\end{corollary}
\begin{proof}
It suffices to prove the statement for $(e^{it\sqrt{-\Delta}})_{t\in\R}$, and by Propositions \ref{prop:duality} and \ref{prop:interpolation} we may consider $p=1$ for the first statement. Let $q\in C^{\infty}_{c}(\Rn)$ be such that $q\equiv1$ near zero,  and $q(\zeta)=0$ for $|\zeta|\geq\frac{1}{2}$. Then $((1-q)(D)e^{it\sqrt{-\Delta}})_{t\in\R}$ is uniformly bounded on $\HT^{1}_{FIO}(\Rn)$, by Theorem \ref{thm:FIObddHardy} and the remark following it. On the other hand, $(q(D)e^{it\sqrt{-\Delta}})_{t\in\R}$ is a locally uniformly bounded subset of $\La(L^{1}(\Rn))$, because these are convolution operators with kernels that are locally uniformly bounded in $L^{1}(\Rn)$. In fact, the kernels are smooth with decay of order $|x-y|^{-(n+1)}$ as $|x-y| \to \infty$. Now Lemma \ref{lem:lowfrequency} yields
\begin{align*}
\|e^{it\sqrt{-\Delta}}f\|_{\HT^{1}_{FIO}(\Rn)}&\leq \|(1-q)(D)e^{it\sqrt{-\Delta}}f\|_{\HT^{1}_{FIO}(\Rn)}+\|q(D)e^{it\sqrt{-\Delta}}f\|_{\HT^{1}_{FIO}(\Rn)}\\
&\lesssim \|f\|_{\HT^{1}_{FIO}(\Rn)}+\|r(D)f\|_{L^{1}(\Rn)}\lesssim \|f\|_{\HT^{1}_{FIO}(\Rn)}
\end{align*}
for all $f\in\HT^{1}_{FIO}(\Rn)$, locally uniformly in $t$. 

The strong continuity for $p<\infty$ now follows from the strong continuity of these families on the Schwartz class, which is dense in $\Hp$ by Proposition \ref{prop:Schwartzdense}.
\end{proof}

It follows in particular that the initial value problem for the wave equation,
\[
\begin{cases} 
\partial_{t}^{2}u=\Delta u&\text{on }\R\times\Rn,\\
(u,\partial_{t}u)|_{t=0}=(f,g)&\text{on }\Rn,
\end{cases}
\]
has a unique solution $u\in C^{2}(\R;\Hp)\cap C(\R;\HT^{p,2}_{FIO}(\Rn))$ for $p<\infty$ if $f\in \HT^{p,2}_{FIO}(\Rn)$ and $g\in \HT_{FIO}^{p,1}(\Rn)$. Here $\HT^{p,s}_{FIO}(\Rn):=\lb D\rb^{-s}\Hp$ for $s\in\R$.

\begin{remark}\label{rem:otherwaves}
The statement of Corollary~\ref{cor:wavebdd} also holds when the flat Laplacian is replaced by a metric Laplacian $\Delta_g$, provided that the metric tensor $(g_{ij})$ is uniformly bounded in $C^\infty$, and uniformly elliptic, on $\R^n$. 
\end{remark}

\section{Sobolev embeddings}\label{sec:Sobolev}

In this section we prove Sobolev embeddings that relate $\Hp$ to the $L^{p}$-scale. 

For later use, we first prove in a fairly direct manner a version of one of the embeddings for $p=1$ that is optimal up to an $\veps$-loss of smoothness.

\begin{lemma}\label{lem:sobolevloss}
For all $\veps>0$ one has $\HT^{1}_{FIO}(\Rn)\subseteq W^{-\frac{n-1}{4}-\veps,1}(\Rn)$ continuously.
\end{lemma}
\begin{proof}
First set 
\[
\wt{\Psi}(\zeta):=\frac{|\zeta|^{-\frac{n-1}{4}-\veps}\Psi(\zeta)}{\big(\int_{0}^{\infty}|\tau\zeta|^{-\frac{n-1}{2}-2\veps}\Psi(\tau\zeta)^{2}\frac{\ud\tau}{\tau}\big)^{1/2}}
\]
for $\zeta\neq0$, and note that the denominator in this expression is a positive constant independent of $\zeta$, since $\Psi$ is radial. Next, set $\wt{\psi}_{\w,\sigma}(\zeta):=\ph_{\w,\sigma}(\zeta)\wt{\Psi}_{\sigma}(\zeta)$ for $\w\in S^{n-1}$ and $\sigma>0$, and note that these wave packets have the properties collected in Lemma \ref{lem:packetbounds}. Let $f\in\HT^{1}_{FIO}(\Rn)$ and write
\[
f=W^{*}Wf=\int_{0}^{1}\int_{S^{n-1}}\psi_{\w,\sigma}(D)^{2}f\ud\w\frac{\ud\sigma}{\sigma}+r(D)^{2}f.
\]
Since $r\in C^{\infty}_{c}(\Rn)$, Lemma \ref{lem:lowfrequency} (with $q=r$) yields
\[
\|r(D)^{2}f\|_{W^{-\frac{n-1}{4}-\veps,1}(\Rn)}\lesssim \|r(D)f\|_{L^{1}(\Rn)}\lesssim \|Wf\|_{T^{1}(\Sp)}=\|f\|_{\HT^{1}_{FIO}(\Rn)}.
\]
On the other hand, $\{\sigma^{\frac{n-1}{4}}\wt{\psi}_{\w,\sigma}(D)\mid\w\in S^{n-1},\sigma>0\}\subseteq \La(L^{1}(\Rn))$ is uniformly bounded, by Lemma \ref{lem:packetbounds}. Hence, using the Cauchy-Schwarz inequality and \eqref{eq:vertical} in the last two inequalities, one has
\begin{align*}
&\Big\|\int_{0}^{1}\int_{S^{n-1}}\psi_{\w,\sigma}(D)^{2}f\ud\w\frac{\ud\sigma}{\sigma}\Big\|_{W^{-\frac{n-1}{4}-\veps,1}(\Rn)}\\
&\lesssim \Big\||D|^{-\frac{n-1}{4}-\veps}\int_{0}^{1}\int_{S^{n-1}}\psi_{\w,\sigma}(D)^{2}f\ud\w\frac{\ud\sigma}{\sigma}\Big\|_{L^{1}(\Rn)}\\
&\lesssim\int_{0}^{1}\int_{S^{n-1}}\int_{\Rn}\sigma^{\frac{n-1}{4}+\veps}|\wt{\psi}_{\w,\sigma}(D)\psi_{\w,\sigma}(D)f(x)|\ud x\ud\w\frac{\ud\sigma}{\sigma}\\
&\lesssim\int_{\Sp}\int_{0}^{1}\sigma^{\veps}|\psi_{\w,\sigma}(D)f(x)|\frac{\ud\sigma}{\sigma}\ud x\ud\w\\
&\lesssim \int_{\Sp}\Big(\int_{0}^{1}|\psi_{\w,\sigma}(D)f(x)|^{2}\frac{\ud\sigma}{\sigma}\Big)^{1/2}\ud x\ud\w\lesssim \|Wf\|_{T^{1}(\Sp)}=\|f\|_{\HT^{1}_{FIO}(\Rn)},
\end{align*}
which concludes the proof.
\end{proof}

For the optimal Sobolev embeddings, we recall the definitions of $\HT^{1}(\Rn)$ and $\bmo(\Rn)$. Let $q\in C^{\infty}_{c}(\Rn)$ be such that $q\equiv1$ in a neighborhood of the origin. The local real Hardy space $\HT^{1}(\Rn)$ consists of all $f\in\Sw'(\Rn)$ such that $(1-q)(D)f\in H^{1}(\Rn)$ and $q(D)f\in L^{1}(\Rn)$, with
\[
\|f\|_{\HT^{1}(\Rn)}:=\|(1-q)(D)f\|_{H^{1}(\Rn)}+\|q(D)f\|_{L^{1}(\Rn)}.
\] 
Here $H^{1}(\Rn)$ is the real Hardy space of Fefferman and Stein. Similarly, $\bmo(\Rn)$ consists of all $f\in \Sw'(\Rn)$ such that $(1-q)(D)f\in\BMO(\Rn)$ and $q(D)f\in L^{\infty}(\Rn)$. For the basic theory of these spaces we refer to \cite{Stein93}.

\begin{remark}\label{rem:H1}
Up to norm equivalence, $\HT^{1}(\Rn)$ is independent of the choice of $q$. Let $q_* \in C_c^\infty(\Rn)$ be such that $q_{*}\equiv 1$ in a neighbourhood of the origin and such that $q_{*}(\zeta)=0$ if $|\zeta|>\frac{1}{2}$, and let $q^* \in C_c^\infty(\Rn)$ be such that $q^{*}(\zeta)=1$ if $|\zeta| \leq 4$. It is straightforward to check that an equivalent norm on $\HT^1(\Rn)$ is given by
\[
\|q^*(D)f\|_{L^{1}(\Rn)}+\|(1-q_*)(D)f\|_{H^{1}(\Rn)}
\]
for $f\in\HT^{1}(\Rn)$.
\end{remark}

In the spirit of this article, we shall prove one of the Sobolev embeddings using the tent space characterization of $H^{1}(\Rn)$. We recall that the tent space $T^{p}(\Rn)$, for $p\in[1,\infty)$, consists of all $F\in L^{2}_{\loc}(\Rn\times(0,\infty))$ such that
\[
\|F\|_{T^{p}(\Rn)}:=\Big(\int_{\Rn}\Big(\int_{0}^{\infty}\fint_{B_{\sigma}(x)}|F(y,\sigma)|^{2}\ud y\frac{\ud\sigma}{\sigma}\Big)^{p/2}\ud x\Big)^{1/p}<\infty.
\]
Moreover, $T^{p}(\Rn)$ has many of the same properties as $T^{p}(\Sp)$, collected in Section \ref{subsec:tent}. In particular, $T^{1}(\Rn)$ has an atomic decomposition. A $T^{1}(\Rn)$-atom, associated with a Euclidean ball $B\subseteq\Rn$ of radius $\tau>0$, is a function $A\in L^{2}(\Rn\times(0,\infty),\ud x\frac{\ud\sigma}{\sigma})$ with $\supp(A)\subseteq B\times[0,\tau]$ and $\|A\|_{L^{2}(\Rn\times(0,\infty))}\leq \tau^{-n/2}$. Then $T^{1}(\Rn)$ has a decomposition as in Lemma \ref{lem:atomictent} involving these atoms. 

The connection to the classical Hardy space $H^{1}(\Rn)$ is as follows. Let $\wt{\Psi}\in C_{c}^{\infty}(\Rn)$ be such that $\wt{\Psi}(\zeta)=1$ for $|\zeta|\in[\frac{1}{2},2]$. Then $f\in H^{1}(\Rn)$ if and only if the function $(x,\sigma)\mapsto \wt{\Psi}_{\sigma}(D)f(x)$ is in the tent space $T^{1}(\Rn)$, and the $T^{1}(\Rn)$-norm of this function is equivalent to the $H^{1}(\Rn)$-norm of $f$. Also, although the conical $T^{1}(\Rn)$-norm is not equivalent to its vertical analogue (see the remarks following \eqref{eq:vertical} in the analogous case of $T^{1}(\Sp)$), they are equivalent on the subspace obtained by embedding $H^{1}(\Rn)$. The corresponding statement for $\HT^{1}(\Rn)$ which is most useful for us is that (see \cite[Sections 1.5.1 and 2.4.2]{Triebel92})
\begin{equation}\label{eq:localHardy}
\|\lb D\rb^{s}f\|_{\HT^{1}(\Rn)}\eqsim \|q(D)f\|_{L^{1}(\Rn)}+\int_{\Rn}\Big(\int_{0}^{1}\sigma^{-2s}|\Psi_{\sigma}(D)f(x)|^{2}\frac{\ud\sigma}{\sigma}\Big)^{1/2}\ud x
\end{equation}
for all $s\in\R$ and $f\in\Sw(\Rn)$, where $q\in C^{\infty}_{c}(\Rn)$ is such that $q(\zeta)=1$ for $|\zeta|\leq 2$. 

The heart of the proof of one of the Sobolev embeddings is the following lemma about averaging operators with respect to balls on $\Sp$. Throughout this section, we write $m(\zeta):=\lb\zeta\rb^{-\frac{n-1}{4}}$ for $\zeta\in\Rn$. 

\begin{lemma}\label{lem:averaging}
For each $c>0$ there exists a constant $C\geq0$ such that the following holds. For $F\in T^{1}(\Rn)$ and $(x,\w,\sigma)\in\Spp$, set
\[
M_{\omega}F(x,\sigma):= \Big( \fint _{B_{c\sqrt{\sigma}}(x,\omega)} 
|m(D)\psi_{\nu,\sigma}(D)F(\cdot,\sigma)(y)|^{2} \ud y \ud\nu \Big)^{1/2}
\]
if $\sigma<1$, and $M_{\w}F(x,\sigma):=0$ if $\sigma\geq1$. Let $A$ be a $T^{1}(\Rn)$-atom associated with a ball $B\subseteq\Rn$ of radius $\tau\in(0,2]$. Then
\[
\int_{S^{n-1}} 
\|M_{\omega}A \|_{T^{1}(\R^{n})} \ud\omega 
 \leq C.
\]
\end{lemma}
\begin{proof}
Without loss of generality, we may assume throughout that $A(\cdot,\sigma)=0$ for $\sigma\geq1$. Let $c_{B}\in\Rn$ be the center of $B$, and for $\w\in S^{n-1}$ set
\[
C_{0,\w}(B):= \{ y \in \R^{n} \mid |\langle \omega, c_{B}-y \rangle|+|c_{B}-y|^{2} \leq \tau\}
\]
and
\[
C_{j,\omega}(B) := \{ y \in \R^{n} \mid 2^{j-1}\tau < |\langle \omega, c_{B}-y \rangle|+|c_{B}-y|^{2} \leq 2^{j}\tau\}
\]
for $j\in\N$. Then $M_{\omega}A = \sum_{j=0}^{\infty} \ind_{C_{j,\omega}(B)}M_{\omega}A$, so it suffices to show that
\[
\int_{S^{n-1}} \|\ind_{C_{j,\omega}(B)}M_{\omega}A\|_{T^{1}(\R^{n})} \ud\omega\lesssim 2^{-j}
\]
for an implicit constant independent of $j\in\Z_{+}$ and $A$. Moreover, for all $\w\in S^{n-1}$ one has $M_{\w}A(\cdot,\sigma)=0$ for $\sigma>\tau$, and for $\sigma\leq \tau$ one has $B_{\sigma}(x)\cap C_{j,\w}(B)=\emptyset$ if $x\notin K_{j,\w}(B)$ for a set $K_{j,\w}(B)\subseteq\Rn$ with $V(K_{j,\w}(B))\lesssim (2^{j}\tau)^{\frac{n+1}{2}}$. Now, by the Cauchy-Schwarz inequality,
\begin{align*}
&\|\ind_{C_{j,\omega}(B)}M_{\omega}A\|_{T^{1}(\R^{n})}=\int_{K_{j,\w}(B)}\Big(\int_{0}^{\tau}\fint_{B_{\sigma}(x)}\ind_{C_{j,\w}(B)}(y)|M_{\w}A(y,\sigma)|^{2}\ud y\frac{\ud\sigma}{\sigma}\Big)^{1/2}\ud x\\
&\lesssim (2^{j}\tau)^{\frac{n+1}{4}} \Big(\int_{\Rn}\int_{0}^{\tau}\int_{\Rn}\ind_{B_{\sigma}(x)}(y)\ind_{C_{j,\omega}(B)}(y)|M_{\omega}A(y,\sigma)|^{2}\frac{\ud y}{V(B_{\sigma}(y))}\frac{\ud\sigma}{\sigma}\ud x\Big)^{1/2}\\
&=(2^{j}\tau)^{\frac{n+1}{4}}\Big(\int_{0}^{\tau}\int_{\Rn}\ind_{C_{j,w}(B)}(y)|M_{\w}A(y,\sigma)|^{2}\ud y\frac{\ud\sigma}{\sigma}\Big)^{1/2}.
\end{align*}
So it suffices to show that 
\[
\int_{S^{n-1}}\|\ind_{C_{j,\w}(B)}M_{\w}A\|_{L^{2}(\Rn\times(0,\infty),\ud x\ud\sigma/\sigma)}\ud\w\lesssim 2^{-j} (2^{j}\tau)^{-\frac{n+1}{4}}.
\]

We first consider $j\leq j_{0}$, for some fixed $j_{0}\in\N$ to be chosen later. Recall that $\int_{S^{n-1}} |\ph_{\nu,\sigma}(\zeta)|^{2} \ud\nu = 1$ for all $\zeta \neq0$. By combining this with the Sobolev embedding $m(D)=\lb D\rb^{-\frac{n-1}{4}} \in \mathcal{L}(L^{p}(\Rn),L^{2}(\Rn))$ for $p = \frac{4n}{3n-1}$, we obtain
\begin{align*}
&\Big(\int_{S^{n-1}} \|\ind_{C_{j,\w}(B)}M_{\w}A\|_{L^{2}(\Rn\times(0,\infty))} \ud\omega\Big)^{2} \lesssim \int_{S^{n-1}} \|M_{\omega}A\|_{L^{2}(\Rn\times(0,\infty))}^{2} \ud \omega \\
&= \!\int_{0}^{\tau}\!\int_{\Sp}\int_{\Sp}\!\ind_{B_{c\sqrt{\sigma}}(x,\w)}(y,\nu)|m(D)\psi_{\nu,\sigma}(D)A(y,\sigma)|^{2}\frac{\ud x\ud\w}{V(B_{c\sqrt{\sigma}}(y,\nu))}\ud y\ud\nu\frac{\ud\sigma}{\sigma}\\
&=\int_{0}^{\tau}\int_{\Sp}|m(D)\psi_{\nu,\sigma}(D)A(y,\sigma)|^{2}\ud y\ud\nu\frac{\ud\sigma}{\sigma}\\
&=(2\pi)^{-n}\int_{0}^{\tau}\int_{S^{n-1}} \int_{\Rn}
\big|m(\zeta){\ph}_{\nu,\sigma}(\zeta)\Psi(\sigma\zeta) \wh{A}(\cdot,\sigma)(\zeta)\big|^{2} \ud\zeta \ud\nu\frac{\ud\sigma}{\sigma}\\
&\lesssim \int_{0}^{\tau} 
\|m(D)A(\cdot,\sigma)\|_{L^{2}(\R^{n})} ^{2} \frac{\ud\sigma}{\sigma} 
 \lesssim \int_{0}^{\tau} \|A(\cdot,\sigma)\|_{L^{p}(\R^{n})} ^{2} \frac{\ud\sigma}{\sigma} \\
&\lesssim \int_{0}^{\tau} \tau^{\frac{n-1}{2}}
\|A(\cdot,\sigma)\|_{L^{2}(\R^{n})} ^{2} \frac{\ud\sigma}{\sigma} 
=\tau^{\frac{n-1}{2}} \|A\|_{L^{2}(\Rn\times(0,\infty))}^{2} \lesssim 2^{-2j}(2^{j}\tau)^{-\frac{n+1}{2}},
\end{align*}
where we also used the defining properties of the atom $A$.

We claim that, by choosing $j_{0}=j_{0}(c,n)$ sufficiently large, we can ensure that for all $j>j_{0}$, $\w\in S^{n-1}$, $\sigma\leq \tau$, $x \in C_{j,\omega}(B)$, $(y,\nu)\in B_{c\sqrt{\sigma}}(x,\w)$ and $z \in B$, one has
\[
(1+\sigma^{-1}|z-y|^{2}+\sigma^{-2}|\lb\nu,z-y\rb|^{2})^{-1}
\lesssim \frac{\sigma}{2^{j}\tau}.
\]
Indeed, Lemma \ref{lem:metric} yields a constant $c'=c'(n)>0$ such that, if $|\lb\w,c_{B}-x\rb|\geq 2^{j-2}\tau$, then
\begin{align*}
|\lb\nu,z-y\rb|&\geq |\lb\w,c_{B}-x\rb|-|\lb\nu-\w,c_{B}-x\rb|-|\lb \nu,z-c_{B}\rb|-|\lb \nu,x-y\rb|\\
&\geq 2^{j-2}\tau-c'c\sqrt{\sigma}2^{j/2}\sqrt{\tau}-\tau-c'c\sigma \gtrsim 2^{j}\tau
\end{align*}
for $j>j_{0}=j_{0}(c,c')$ large enough. Similarly, if $|c_{B}-x|^{2}\geq 2^{j-2}\tau$, then 
\[
|z-y|\geq |c_{B}-x|-|z-c_{B}|-|x-y|\geq \sqrt{2^{j-2}\tau}-\tau-c'c\sqrt{\sigma} \gtrsim 2^{j/2}\sqrt{\tau},
\]
thereby proving the claim.

Now fix $j>j_{0}$ and set
\[
\wt{\Psi}(\zeta):=\frac{|\zeta|^{-\frac{n-1}{4}}\Psi(\zeta)}{\big(\int_{0}^{\infty}|\tau\zeta|^{-\frac{n-1}{2}}\Psi(\tau\zeta)^{2}\frac{\ud\tau}{\tau}\big)^{1/2}}
\]
and $\wt{m}(\zeta)=|\zeta|^{\frac{n-1}{4}}m(\zeta)$
for $\zeta\neq0$. Let $\wt{\psi}_{\w,\sigma}(\zeta):=\ph_{\w,\sigma}(\zeta)\wt{\Psi}_{\sigma}(\zeta)$ for $\w\in S^{n-1}$ and $\sigma>0$. 
We combine the claim above with the properties of $\wt{\psi}_{\nu,\sigma}$ from Lemma \ref{lem:packetbounds}. This yields, for each $M\geq\frac{n+3}{2}$ and uniformly in $\omega \in S^{n-1}$,
\begin{align*}
&\|\ind_{C_{j,\omega}(B)} M_{\omega}A\|_{L^{2}(\Rn\times(0,\infty))} ^{2}\\
&=\int_{0}^{\tau} \int_{C_{j,\omega}(B)} \fint _{B_{c\sqrt{\sigma}}(x,\omega)} 
|m(D)\psi_{\nu,\sigma}(D)A(y,\sigma)|^{2}\ud y\ud\nu\ud x\frac{\ud\sigma}{\sigma}\\
&\eqsim\int_{0}^{\tau} \int_{C_{j,\omega}(B)} \fint _{B_{c\sqrt{\sigma}}(x,\omega)} 
|\sigma^{\frac{n-1}{4}}\wt{\psi}_{\nu,\sigma}(D)\wt{m}(D)A(y,\sigma)|^{2}\ud y\ud\nu\ud x\frac{\ud\sigma}{\sigma}\\
&=\int_{0}^{\tau} \int_{C_{j,\omega}(B)} \fint _{B_{c\sqrt{\sigma}}(x,\omega)}\!\Big(\int_{B} \sigma ^{\frac{n-1}{4}} |\mathcal{F}^{-1}(\wt{\psi}_{\nu,\sigma})(z-y)\wt{m}(D)A(z,\sigma)|\ud z\Big)^{2} \ud y \ud\nu \ud x \frac{\ud\sigma}{\sigma}\\
&\lesssim \int_{0}^{\tau} \int_{C_{j,\omega}(B)}\fint _{B_{c\sqrt{\sigma}}(x,\omega)} 
\Big(\int_{B} \sigma ^{-\frac{n+1}{2}} \Big(\frac{\sigma}{2^{j}\tau}\Big)^{M} |\wt{m}(D)A(z,\sigma)|\ud z\Big)^{2} \ud y \ud\nu \ud x \frac{\ud\sigma}{\sigma}\\
&\lesssim (2^{j}\tau)^{-n-1}\!\int_{0}^{\tau}\!\int_{C_{j,\omega}(B)}\fint_{B_{c\sqrt{\sigma}}(x,\omega)}\!\Big(\!\int_{B} \!\Big(\frac{\sigma}{2^{j}\tau}\Big)^{M-\frac{n+1}{2}} |\wt{m}(D)A(z,\sigma)|\ud z\!\Big)^{2} \ud y \ud\nu \ud x \frac{\ud\sigma}{\sigma} \\
&\lesssim 2^{-j(2M-n-1)}(2^{j}\tau)^{-n-1}\tau^{n}\!
\int_{0}^{\tau}\int_{C_{j,\omega}(B)} \fint _{B_{c\sqrt{\sigma}}(x,\omega)} 
\int_{B}|\wt{m}(D)A(z,\sigma)|^{2} \ud z \ud y \ud \nu \ud x \frac{\ud\sigma}{\sigma} \\
&\lesssim 2^{-j(2M-n-1)} (2^{j}\tau)^{-\frac{n+1}{2}}\tau^{n} \int_{0} ^{\infty} \|\wt{m}(D)A(\cdot,\sigma)\|_{L^{2}(\R^{n})} ^{2} 
\frac{\ud\sigma}{\sigma}\\
& \lesssim 2^{-2j} (2^{j}\tau)^{-\frac{n+1}{2}},
\end{align*}
where in the last step we also used that $\wt{m}\in L^{\infty}(\Rn)$. Finally, the inequality 
\[
\int_{S^{n-1}}\|\ind_{C_{j,\omega}(B)} M_{\omega}A\|_{L^{2}(\Rn\times(0,\infty))}\ud\w\lesssim \Big(\int_{S^{n-1}}\|\ind_{C_{j,\omega}(B)} M_{\omega}A\|_{L^{2}(\Rn\times(0,\infty))}^{2}\ud\w\Big)^{1/2}
\]
concludes the proof.
\end{proof}

We are now ready for the proof of the Sobolev embeddings.

\begin{theorem}\label{thm:Sobolev}
Let $p\in(1,\infty)$. Then the continuous embeddings 
\[
W^{\frac{n-1}{2}|\frac{1}{p}-\frac{1}{2}|,p}(\Rn)\subseteq\Hp\subseteq W^{-\frac{n-1}{2}|\frac{1}{p}-\frac{1}{2}|,p}(\Rn)
\]
hold. Moreover, 
\begin{equation}\label{eq:embedding1}
\HT^{1}(\Rn)\stackrel{\lb D\rb^{-\frac{n-1}{4}}}{\longrightarrow} \HT^{1}_{FIO}(\Rn)\stackrel{\lb D\rb^{-\frac{n-1}{4}}}{\longrightarrow}\HT^{1}(\Rn)
\end{equation}
and
\[
\bmo(\Rn)\stackrel{\lb D\rb^{-\frac{n-1}{4}}}{\longrightarrow} \HT^{\infty}_{FIO}(\Rn)\stackrel{\lb D\rb^{-\frac{n-1}{4}}}{\longrightarrow}\bmo(\Rn)
\]
continuously.
\end{theorem}
\begin{proof}
By Propositions \ref{prop:interpolation} and \ref{prop:duality} and by basic facts about interpolation and duality of $\HT^{1}(\Rn)$ and $\bmo(\Rn)$, it suffices to prove \eqref{eq:embedding1}. 

Our proof of the first embedding in \eqref{eq:embedding1} is similar to that in \cite[Theorem 4.2]{Smith98a}, although we use tent spaces and Lemma \ref{lem:averaging} instead of the atomic decomposition of $\HT^{1}(\Rn)$. Let $q_* \in C_c^\infty(\Rn)$ be such that $q_{*}\equiv 1$ in a neighbourhood of the origin and such that $q_{*}(\zeta)=0$ if $|\zeta|>\frac{1}{2}$, and let $q^* \in C_c^\infty(\Rn)$ be such that $q^{*}(\zeta)=1$ if $|\zeta| \leq 4$. Let $f\in\HT^{1}(\Rn)$, and recall that $m(\zeta)=\lb\zeta\rb^{-\frac{n-1}{4}}$ for $\zeta\in\Rn$. By Remark \ref{rem:H1}, for the first embedding it suffices to show that
\begin{equation}\label{eq:Sobolev1}
\|Wm(D)f\|_{T^{1}(\Sp)}\lesssim \|q^{*}(D)f\|_{L^{1}(\Rn)}+\|(1-q_{*})(D)f\|_{H^{1}(\Rn)}.
\end{equation}
Write $f=f_{1}+f_{2}$ with $f_{1}:= q_{*}(D)f$ and $f_2 := (1-q_*)(D) f$. For $g\in\Sw'(\Rn)$ and $(x,\w,\sigma)\in\Spp$, set 
\[
W_{<1}g(x,\w,\sigma):=\ind_{(0,1)}(\sigma)Wg(x,\w,\sigma)
\]
and 
\[
W_{\geq 1}g(x,\w,\sigma):=\ind_{[1,\infty)}(\sigma)Wg(x,\w,\sigma).
\]
Note that, by construction, $W_{<1}m(D)f_1 = 0$. 

We first deal with the low-frequency components, $W_{\geq 1} m(D) f_1$ and $W_{\geq 1} m(D) f_2$. By Lemma~\ref{lem:lowfrequency} (with $q=r$) and because $q_{*}=q_{*}q^{*}$ and $r=rq^{*}$, we have 
\[
\| W_{\geq 1} m(D) f_1 \|_{T^1(\Sp)} \lesssim \| r(D)m(D) q_*(D) f \|_{L^1(\Rn)}\lesssim\|q^*(D) f \|_{L^1(\Rn)},
\]
and
\[
\| W_{\geq 1} m(D) f_2 \|_{T^{1}(\Sp)} \lesssim\| m(D) r(D)(1-q_{*})(D) f \|_{L^1(\Rn)}\lesssim \|q^{*}(D)f\|_{L^{1}(\Rn)}.
\]
Hence $\| W_{\geq 1} m(D) f \|_{T^1(S^* \Rn)} \lesssim \| q^*(D) f \|_{L^1(\Rn)}$.

It remains to consider $W_{<1} m(D)f_2$. Let $\wt{\Psi}\in C^{\infty}_{c}(\Rn)$ be such that $\wt{\Psi}(\zeta)=1$ if $|\zeta|\in[\frac{1}{2},2]$, and set $F(x,\sigma):=\wt{\Psi}_{\sigma}(D)f_{2}(x)$ for $x\in\Rn$ and $\sigma>0$. Then $F\in T^{1}(\Rn)$ and 
\[
\|F\|_{T^{1}(\Rn)}\eqsim \|f_{2}\|_{H^{1}(\Rn)}\eqsim \|f_{2}\|_{\HT^{1}(\Rn)}.
\]
Also note that there exists a $c>0$ such that $B_{\sqrt{\sigma}}(x,\w)\subseteq B_{c\sqrt{\sigma}}(z,\omega)$ for all $(x,w,\sigma)\in\Spp$ with $\sigma<1$ and all $z\in B_{\sigma}(x)$, by Lemma \ref{lem:metric}. Let $M_{\w}$ be as in Lemma \ref{lem:averaging}, for this $c$. Now use the fact that $\psi_{\w,\sigma}(D)=\psi_{\w,\sigma}(D)\wt{\Psi}_{\sigma}(D)$ for all $\w\in S^{n-1}$ and $\sigma>0$ to see that
\begin{align*}
&\int_{S^{n-1}}\|M_{\w}F\|_{T^{1}(\Rn)}\ud\w\\
&=\int_{\Sp}\Big(\int_{0}^{1}\fint_{B_{\sigma}(x)}\fint_{B_{c\sqrt{\sigma}}(z,\w)}|m(D)\psi_{\nu,\sigma}(D)F(\cdot,\sigma)(y)|^{2}\ud y\ud\nu\ud z\frac{\ud\sigma}{\sigma}\Big)^{1/2}\ud x\ud\w\\
&\gtrsim \int_{\Sp}\Big(\int_{0}^{1}\fint_{B_{\sigma}(x)}\fint_{B_{\sqrt{\sigma}}(x,\w)}|m(D)\psi_{\nu,\sigma}(D)f_{2}(y)|^{2}\ud y\ud\nu\ud z\frac{\ud\sigma}{\sigma}\Big)^{1/2}\ud x\ud\w\\
&=\|W_{<1}m(D)f_{2}\|_{T^{1}(\Sp)}.
\end{align*}
Hence, to conclude the proof of \eqref{eq:Sobolev1} and the first embedding in \eqref{eq:embedding1}, it suffices to show that
\[
\int_{S^{n-1}}\|M_{\w}F\|_{T^{1}(\Rn)}\ud\w\lesssim \|F\|_{T^{1}(\Rn)}.
\]
But this follows from Lemma \ref{lem:averaging}, since we may write $F=\sum_{k=1}^{\infty}\alpha_{k}A_{k}$ for a sequence of $T^{1}(\Rn)$-atoms $(A_{k})_{k=1}^{\infty}$ associated with balls of radius at most $2$ (see also Remark \ref{rem:localatoms}), and a sequence $(\alpha_{k})_{k=1}^{\infty}\subseteq\C$ such that $\|F\|_{T^{1}(\Rn)}\eqsim \sum_{k=1}^{\infty}|\alpha_{k}|$.

We now turn to the second embedding in \eqref{eq:embedding1}, for which we use the following reproducing formula: for all $\sigma>0$ there exists a $C_{\sigma}>0$ such that
\begin{equation}
\sigma ^{-\frac{n-1}{4}}\int_{S^{n-1}} \ph_{\omega,\sigma}(D)f \ud\omega = C_{\sigma}f
\label{eq:Csigma}\end{equation}
for all $f \in \Sw(\Rn)$. Moreover, one has $\frac{1}{C'}\leq C_{\sigma}\leq C'$ for all $\sigma>0$ and some constant $C'>0$ independent of $\sigma$. This is shown exactly as in \eqref{eq:csigma}, keeping in mind that $\ph_{\w,\sigma}(\zeta)=c_{\sigma}\ph(\tfrac{\hat{\zeta}-\w}{\sqrt{\sigma}})$ for $\zeta\neq0$, and that $c_{\sigma}\eqsim \sigma^{-\frac{n-1}{4}}$. 

Now let $f\in \Sw(\Rn)$. Since $\Sw(\Rn)\subseteq\HT^{1}_{FIO}(\Rn)$ is dense (see Proposition \ref{prop:Schwartzdense}), by \eqref{eq:localHardy} it suffices to show that
\[
\|q(D)f\|_{L^{1}(\Rn)}+\int_{\Rn}\Big(\int_{0}^{1}\sigma^{\frac{n-1}{2}}|\Psi_{\sigma}(D)f(x)|^{2}\frac{\ud\sigma}{\sigma}\Big)^{1/2}\ud x\lesssim \|f\|_{\HT^{1}_{FIO}(\Rn)},
\]
where $q\in C^{\infty}_{c}(\Rn)$ is such that $q(\zeta)=1$ for $|\zeta|\leq 2$. For the first term in this inequality one simply uses Lemma \ref{lem:sobolevloss}:
\[
\|q(D)f\|_{L^{1}(\Rn)}\lesssim \|f\|_{W^{-\frac{n-1}{4}-\veps,1}(\Rn)}\lesssim \|f\|_{\HT^{1}_{FIO}(\Rn)}
\]
for any $\veps>0$. 
For the second term, we use \eqref{eq:Csigma}, Minkowski's integral inequality and the bound for vertical square functions in terms of conical ones from \eqref{eq:vertical} to write  
\begin{align*}
& \int_{\Rn}\Big(\int_{0}^{1}\sigma^{\frac{n-1}{2}}|\Psi_{\sigma}(D)f(x)|^{2}\frac{\ud\sigma}{\sigma}\Big)^{1/2}\ud x\\
&\eqsim \int_{\Rn}\Big(\int_{0}^{1}|C_{\sigma}\sigma^{\frac{n-1}{4}}\Psi_{\sigma}(D)f(x)|^{2}\frac{\ud\sigma}{\sigma}\Big)^{1/2}\ud x\\
&=\int_{\Rn}\Big(\int_{0}^{1}\Big|\int_{S^{n-1}}\psi_{\w,\sigma}(D)f(x)\ud\w\Big|^{2}\frac{\ud\sigma}{\sigma}\Big)^{1/2}\ud x\\
&\leq \int_{\Sp}\Big(\int_{0}^{1}|\psi_{\w,\sigma}(D)f(x)|^{2}\frac{\ud\sigma}{\sigma}\Big)^{1/2}\ud x\ud\w\lesssim \|Wf\|_{T^{1}(\Sp)}=\|f\|_{\HT^{1}_{FIO}(\Rn)},
\end{align*}
thereby concluding the proof.
\end{proof}

\begin{remark}\label{rem:why}
The reason for first proving the suboptimal Sobolev embedding in Lemma \ref{lem:sobolevloss}, apart from the intrinsic interest in obtaining in a fairly straightforward manner a slightly weaker result, is that it allows us to estimate $\|q(D)f\|_{L^{1}(\Rn)}\lesssim \|f\|_{\HT^{1}_{FIO}(\Rn)}$ for any $f\in\HT^{1}_{FIO}(\Rn)$ and $q\in C^{\infty}_{c}(\Rn)$. One could alternatively try to obtain this estimate directly, and in fact it will follow a fortiori from Corollary \ref{cor:equivalentnorm} below.     
\end{remark}

As a corollary of Theorem \ref{thm:Sobolev}, we obtain two equivalent norms on $\Hp$. For $f\in\Sw'(\Rn)$ and $(x,\w)\in\Sp$, set
\[
Sf(x,\w):=\Big(\int_{0}^{1}\fint_{B_{\sqrt{\sigma}}(x,\w)}|\psi_{\nu,\sigma}(D)f(y)|^{2}\ud y\ud\nu\frac{\ud\sigma}{\sigma}\Big)^{1/2}
\]
and
\[
Q f(x,\w):=\sup_{B}\Big(\frac{1}{V(B)}\int_{T(B)}\ind_{[0,1]}(\sigma)|\psi_{\nu,\sigma}(D)f(y)|^{2}\ud y\ud\nu\frac{\ud\sigma}{\sigma}\Big)^{1/2},
\]
where the supremum is taken over all balls $B\subseteq\Sp$ containing $(x,\w)$.

\begin{corollary}\label{cor:equivalentnorm}
Let $q\in C^{\infty}_{c}(\Rn)$ be such that $q(\zeta)=1$ if $|\zeta|\leq 2$, and let $p\in[1,\infty]$ and $f\in\Sw'(\Rn)$. Then the following assertions hold.
\begin{enumerate}
\item For $p<\infty$, one has $f\in\Hp$ if and only if $Sf\in L^{p}(\Sp)$ and $q(D)f\in L^{p}(\Rn)$, and then 
\[
\|f\|_{\Hps}\!\eqsim \|Sf\|_{L^{p}(\Sp)}+\|q(D)f\|_{L^{p}(\Rn)}\!\eqsim\|Sf\|_{L^{p}(\Sp)}+\|f\|_{W^{-\frac{n-1}{2}|\frac{1}{p}-\frac{1}{2}|,p}(\Rn)}
\]
for implicit constants independent of $f$.
\item One has $f\in\HT^{\infty}_{FIO}(\Rn)$ if and only if $Qf\in L^{\infty}(\Sp)$ and $q(D)f\in L^{\infty}(\Rn)$, and then
\[
\|f\|_{\HT^{\infty}_{FIO}}\!\eqsim \|Q f\|_{L^{\infty}\!(\Sp)}+\|q(D)f\|_{L^{\infty}\!(\Rn)}\!\eqsim\! \|Q f\|_{L^{\infty}\!(\Sp)}+\|\lb D\rb^{-\frac{n-1}{4}}\!f\|_{\bmo(\Rn)}
\]
for implicit constants independent of $f$.
\end{enumerate}
\end{corollary}
\begin{proof}
First consider the case where $p\in[1,\infty)$, and set $s:=\frac{n-1}{2}|\frac{1}{p}-\frac{1}{2}|$. Suppose that $f\in \Hp$. Since $q\in C^{\infty}_{c}(\Rn)$, one has 
\begin{equation}\label{eq:Sobolevcompact}
\|q(D)f\|_{L^{p}(\Rn)}\eqsim \|q(D)f\|_{W^{s,p}(\Rn)}\eqsim \|q(D)f\|_{W^{-s,p}(\Rn)}\lesssim \|f\|_{W^{-s,p}(\Rn)}
\end{equation}
Also, 
\begin{equation}\label{eq:H1}
\|f\|_{W^{-\frac{n-1}{4},1}(\Rn)}\lesssim \|\lb D\rb^{-\frac{n-1}{4}}f\|_{\HT^{1}(\Rn)}
\end{equation}
and $\|Sf\|_{L^{p}(\Sp)}\leq \|f\|_{\Hp}$. Now Theorem \ref{thm:Sobolev} yields
\[
\|Sf\|_{L^{p}(\Sp)}+\|q(D)f\|_{L^{p}(\Rn)}\lesssim \|Sf\|_{L^{p}(\Sp)}+\|f\|_{W^{-s,p}(\Rn)}\lesssim \|f\|_{\Hp}.
\]

Next, suppose that $Sf\in L^{p}(\Sp)$ and $q(D)f\in L^{p}(\Rn)$. We first show that $(1-q)(D)f\in \Hp$ with
\begin{equation}\label{eq:qonHp}
\|(1-q)(D)f\|_{\Hp}\lesssim \|Sf\|_{L^{p}(\Sp)}.
\end{equation}
To this end, for $(x,\w,\sigma)\in\Spp$, set
\[
F(x,\w,\sigma):=\ind_{[0,1]}(\sigma)\psi_{\w,\sigma}(D)f(x).
\] 
Then $W^{*}F=(1-r^{2})(D)f$, as is straightforward to check, and $\|F\|_{T^{p}(\Sp)}=\|Sf\|_{L^{p}(\Sp)}$. Since $(1-q)(1-r^{2})=(1-q)$, \eqref{eq:qonHp} follows from Corollary \ref{cor:normalbounds}:
\begin{align*}
&\|(1-q)(D)f\|_{\Hp}=\|W(1-q)(D)f\|_{T^{p}(\Sp)}\\
&=\|W(1-q)(D)(1-r^{2})(D)f\|_{T^{p}(\Sp)}=\|W(1-q)(D)W^{*}F\|_{T^{p}(\Sp)}\\
&\lesssim \|F\|_{T^{p}(\Sp)}=\|Sf\|_{L^{p}(\Sp)}.
\end{align*}
Also, as in \eqref{eq:Sobolevcompact}, it is straightforward to see that
\begin{equation}\label{eq:H12}
\|\lb D\rb^{\frac{n-1}{4}}q(D)f\|_{\HT^{1}(\Rn)}\lesssim \|q(D)f\|_{L^{1}(\Rn)}.
\end{equation}
Now combine \eqref{eq:qonHp} with Theorem \ref{thm:Sobolev} and \eqref{eq:Sobolevcompact}, or \eqref{eq:H12} for $p=1$, to see that
\[
\|f\|_{\Hps}\leq \|(1-q)(D)f\|_{\Hps}\!+\|q(D)f\|_{\Hps}\!\lesssim \|Sf\|_{L^{p}(\Sp)}+\|q(D)f\|_{L^{p}(\Rn)}.
\]
This proves the required statements for $p\in[1,\infty)$. 

The proof for $p=\infty$ is similar, except that \eqref{eq:H1} needs to be replaced by
\[
\|q(D)f\|_{L^{\infty}(\Rn)}\lesssim \|\lb D\rb^{-\frac{n-1}{4}}f\|_{\bmo(\Rn)},
\]
and \eqref{eq:H12} by
\[
\|\lb D\rb^{\frac{n-1}{4}}q(D)f\|_{\bmo(\Rn)}\lesssim \|q(D)f\|_{L^{\infty}(\Rn)}.\qedhere
\]
\end{proof}

As another corollary, we extend the main results of \cite{SeSoSt91} and \cite{Coriasco-Ruzhansky14} to a wider class of oscillatory integrals (see also \cite{IsRoSt18} for recent related results). For $s\in\R$, write $\HT^{1,s}(\Rn):=\lb D\rb^{-s}\HT^{1}(\Rn)$ and $\bmo^{s}(\Rn):=\lb D\rb^{-s}\bmo(\Rn)$.

\begin{corollary}\label{cor:SSS}
Let $T$ be one of the following:
\begin{enumerate}
\item\label{it:SSS1} A normal oscillatory integral operator of order $0$ and type $(\frac{1}{2},\frac{1}{2},1)$ with symbol $a$ and phase function $\Phi$ with the following properties. The induced contact transformation $\hchi$ satisfies $\dom(\hchi)=\ran(\hchi)=\Sp$, and either $(z,\theta)\mapsto\Phi(z,\theta)$ is linear in $\theta$ or there exists an $\veps>0$ such that $a(z,\theta)=0$ for all $(z,\theta)\in \R^{2n}$ with $|\theta|<\veps$;
\item\label{it:SSS2} A Fourier integral operator of order $0$ and type $(\rho,1-\rho,1)$, for $\rho\in(\frac{1}{2},1]$, associated with a local canonical graph, such that the Schwartz kernel of $T$ has compact support.
\end{enumerate}
Then
\[
T:W^{\frac{n-1}{2}|\frac{1}{p}-\frac{1}{2}|,p}(\Rn)\to W^{-\frac{n-1}{2}|\frac{1}{p}-\frac{1}{2}|,p}(\Rn)
\]
continuously for all $p\in(1,\infty)$. Moreover, 
\[
T:\HT^{1,\frac{n-1}{4}}(\Rn)\to \HT^{1,-\frac{n-1}{4}}(\Rn)
\]
and
\[
T:\bmo^{\frac{n-1}{4}}(\Rn)\to \bmo^{-\frac{n-1}{4}}(\Rn)
\]
continuously.
\end{corollary}
\begin{proof}
This follows directly from Theorems \ref{thm:FIObddHardy} and \ref{thm:Sobolev}.
\end{proof}

\begin{remark}\label{rem:Lp}
One can extend Corollary \ref{cor:SSS} to general Sobolev spaces if $T$ is either a translation-invariant operator as in \eqref{it:SSS1} or an operator as in \eqref{it:SSS2}. Then, for all $p \in (1, \infty)$ and $s \in \R$, one has
\[
T:W^{s,p}(\Rn)\to W^{s-(n-1)|\frac{1}{p}-\frac{1}{2}|,p}(\Rn)
\]
and 
\[
T:\HT^{1,s}(\Rn)\to \HT^{1,s-\frac{n-1}{2}}(\Rn),\quad T:\bmo^{s}(\Rn)\to \bmo^{s-\frac{n-1}{2}}(\Rn).
\]
For translation-invariant operators as in \eqref{it:SSS1}, this follows from the fact that $T$ commutes with $\lb D\rb^{s}$ for all $s\in\R$. For operators as in \eqref{it:SSS2}, we can use Egorov's theorem to write $T \lb D\rb^{s} = A \, \lb D \rb^{s} T + R$ for a suitable pseudodifferential operator $A$ of order zero and a smoothing operator $R$. 

It is also worth noting that our proof of Corollary \ref{cor:SSS} does not rely a priori on the $L^{2}$-boundedness of Fourier integral operators. In fact, Corollary \ref{cor:SSS} directly recovers the $L^{2}$-boundedness of both pseudodifferential operators and FIOs with suitable exotic symbols. 
\end{remark}

\begin{remark}\label{rem:sharpness}
It is known that Corollary \ref{cor:SSS} is sharp with respect to the Sobolev exponents (see e.g.~\cite[Section IX.6.13]{Stein93}). It follows that the Sobolev exponents in Theorem \ref{thm:Sobolev} cannot be improved either. 
\end{remark}

\section{Molecular decomposition}\label{sec:decomp}

In this section we obtain a molecular decomposition for $\HT^{1}_{FIO}(\Rn)$. The decomposition and the arguments used to obtain it are similar to those in \cite{Smith98a}. However, we cannot appeal \emph{a priori} to those results, because we use the molecular decomposition to argue that $\HT^{1}_{FIO}(\Rn)$ is a Sobolev space over the space in \cite{Smith98a} (see Remark \ref{rem:Hartspace}). Moreover, instead of constructing the $\HT^{1}_{FIO}(\Rn)$-molecules `by hand', we use the established atomic decomposition of the tent space $T^{1}(\Sp)$ and project these atoms down to $\HT^{1}_{FIO}(\Rn)$ using the adjoint of the wave packet transform.

Let $s>\frac{n}{2}$ and $C\geq0$. We call a function $f\in L^{2}(\Rn)$ a \emph{coherent molecule} of type $(s,C)$ associated with a ball $B_{\sqrt{\tau}}(y,\nu)$, for $(y,\nu)\in\Sp$ and $\tau>0$, if 
\begin{equation}\label{eq:supportmolecule}
\supp(\wh{f}\,)\subseteq \{\zeta\in\Rn\mid |\zeta|\geq \tau^{-1}, |\hat{\zeta}-\nu|\leq \sqrt{\tau}\}
\end{equation}
and 
\begin{equation}\label{eq:decaymolecule}
\int_{\Rn}(1+\tau^{-1}|\lb\nu,x-y\rb|)^{2s}(1+\tau^{-1}|x-y|^{2})^{2s}|f(x)|^{2}\ud x\leq C\tau^{-n}.
\end{equation}
The choice of $s>\frac{n}{2}$ and $C\geq0$ is irrelevant for most purposes, but our methods only yield a molecular decomposition with $C\geq0$ sufficiently large. Note that \eqref{eq:supportmolecule} implies that $f$ satisfies a strong cancellation condition. Note also that a coherent molecule is not compactly supported unless $f=0$; it is merely concentrated on an anisotropic ball around $y$. This is because, throughout, we work with wave packets $\psi_{\w,\sigma}$ that have compact support. In fact, each such wave packet is a coherent molecule of type $(s,C_{s})$ associated with the ball $B_{2\sqrt{\sigma}}(0,\w)$, for each $s>\frac{n}{2}$ and some $C_{s}\geq0$. If one were instead to use wave packets $\psi_{\w,\sigma}$ such that $\F^{-1}(\psi_{\w,\sigma})$ has compact support, then the procedure below would yield a decomposition into compactly supported atoms whose Fourier transform is concentrated on suitable cones. 

\begin{remark}\label{rem:molecule}
The terminology `coherent molecule' comes from \cite{Smith98a}. Essentially the same molecules are considered there, except that $f$ is replaced by $\lb D\rb^{\frac{n-1}{4}}f$, and the weight in \eqref{eq:decaymolecule} is replaced by a weight equivalent to
\[
(1+\tau^{-1}(|\lb\nu,x-y\rb|+|x-y|^{2}))^{2s}\eqsim (1+\tau^{-1}d((x,\nu),(y,\nu))^{2})^{2s}.
\]
One has
\begin{align*}
1+\tau^{-1}(|\lb\nu,x-y\rb|+|x-y|^{2})&\leq (1+\tau^{-1}|\lb\nu,x-y\rb|)(1+\tau^{-1}|x-y|^{2})\\
&\leq (1+\tau^{-1}(|\lb\nu,x-y\rb|+|x-y|^{2}))^{2},
\end{align*}
so that up to a change in $s$ either of the two weights can be used; we work with the weight in \eqref{eq:decaymolecule} because it is more convenient when integrating by parts.
\end{remark}

We need a lemma, similar to \cite[Lemma 2.14]{Smith98a}, about the mapping properties of the wave packet transform with respect to weighted $L^{2}$-spaces. Throughout, for $(x,\nu,\tau)\in\Spp$, we let
\begin{equation}\label{eq:moleculeweight}
m^{\nu,\tau}(x):=(1+\tau^{-2}\lb\nu,x\rb^{2})^{1/2}(1+\tau^{-1}|x|^{2}).
\end{equation}
Note that the weight occurring in \eqref{eq:decaymolecule} is equivalent to $m^{\nu,\tau}(\cdot-y)^{2s}$. 

\begin{lemma}\label{lem:weightedL2}
For all $s,\kappa\geq0$ there exists a constant $C\geq0$ such that the following holds. Fix $(y,\nu,\tau)\in\Spp$ and set
\[
E:=\{(\w,\sigma)\in S^{n-1}\times(0,1)\mid \sigma\leq \kappa\tau,|\w-\nu|\leq \kappa\sqrt{\tau}\}
\]
and $W_{E}f:=\ind_{E}Wf$ for $f\in\Sw(\Rn)$. Let $V_{E}:=W^{*}_{E}$ be the adjoint of $W_{E}\in \La(L^{2}(\Rn),L^{2}(\Spp))$. For $(x,\w,\sigma)\in\Sp$, set $m_{1}(x):=m^{\nu,\tau}(x-y)^{2s}$ and $m_{2}(x,\w,\sigma):=m_{1}(x)$. Then 
\[
\big\|W_{E}\big\|_{\La(L^{2}(\Rn,m_{1}),L^{2}(\Spp,m_{2}))}\leq C
\]
and 
\[
\big\|V_{E}\big\|_{\La(L^{2}(\Spp,m_{2}),L^{2}(\Rn,m_{1}))}\leq C.
\]
\end{lemma}
\begin{proof}
By complex interpolation, we may assume that $s\in2\Z_{+}$. Also, since $W$ commutes with translations, we may assume that $y=0$. Set
\[
L:=(1-\tau^{-2}\lb\nu,\nabla_{\zeta}\rb^{2})^{s/2}(1-\tau^{-1}\Delta_{\zeta})^{s}.
\]
Then, for $f\in\Sw(\Rn)$ and $(x,\w,\sigma)\in\Spp$, one has 
\[
m^{\nu,\tau}(x)^{s}\psi_{\w,\sigma}(D)f(x)=(2\pi)^{-n}\int_{\Rn}L(e^{ix\cdot\zeta})\psi_{\w,\sigma}(\zeta)\wh{f}(\zeta)\ud\zeta.
\]
Integration by parts expresses the right-hand side as a linear combination of terms of the form
\[
\int_{\Rn}e^{i x\cdot\zeta}(\tau^{-\frac{|\alpha|}{2}-\beta}\lb\nu,\nabla_{\zeta}\rb^{\beta}\partial_{\zeta}^{\alpha}\psi_{\w,\sigma})(\zeta)\wh{g}(\zeta)\ud\zeta,
\]
for a $g\in L^{2}(\Rn)$ with $\|g\|_{L^{2}(\Rn)}\leq \|f\|_{L^{2}(\Rn,m_{1})}$, and $\alpha\in\Z_{+}^{n}$ and $\beta\in\Z_{+}$ with $|\alpha|\leq 2s$ and $\beta\leq s$. Fix such $\alpha$ and $\beta$ and, for $h\in\Sw(\Rn)$, let $\wt{W}h:\Spp\to\C$ be given by
\[
\wt{W}h(x,\w,\sigma)=\ind_{E}(\w,\sigma)\F^{-1}\big(\tau^{-\frac{|\alpha|}{2}-\beta}(\lb\nu,\nabla_{\zeta}\rb^{\beta}\partial_{\zeta}^{\alpha}\psi_{\w,\sigma})\wh{h}\big)(x).
\]
For the first statement it then suffices to show that $\wt{W}:L^{2}(\Rn)\to L^{2}(\Spp)$ with 
\[
\big\|\wt{W}\big\|_{\La(L^{2}(\Rn),L^{2}(\Spp)}\lesssim 1,
\]
for an implicit constant independent of $\nu$ and $\tau$. In turn, this follows if we show that
\[
\int_{E}\tau^{-|\alpha|-2\beta}|\lb\nu,\nabla_{\zeta}\rb^{\beta}\partial_{\zeta}^{\alpha}\psi_{\w,\sigma}(\zeta)|^{2}\ud\w\frac{\ud\sigma}{\sigma}\lesssim 1
\]
for an implicit constant independent of $\nu$, $\tau$ and $\zeta$. But this inequality holds by the properties of $\psi_{\w,\sigma}$ from Lemma \ref{lem:packetbounds} and the definition of $E$:
\begin{align*}
\int_{E}\tau^{-|\alpha|-2\beta}|\lb\nu,\nabla_{\zeta}\rb^{\beta}\partial_{\zeta}^{\alpha}\psi_{\w,\sigma}(\zeta)|^{2}\ud\w\frac{\ud\sigma}{\sigma}\lesssim \tau^{-\frac{n-1}{2}}\int_{E}\ind_{[|\zeta|^{-1}/2,2|\zeta|^{-1}]}(\sigma)\ud\w\frac{\ud\sigma}{\sigma}\lesssim 1.
\end{align*}

The proof of the second statement is similar. It suffices to show that
\[
\Big|\int_{\Rn}\sqrt{m_{1}(x)}V_{E}F(x)g(x)\ud x\Big|\lesssim \|F\|_{L^{2}(\Spp,m_{2})}\|g\|_{L^{2}(\Rn)}
\]
for $g\in\Sw(\Rn)$ and $F\in L^{2}(\Spp,m_{2})$ with $F=\ind_{E}F$. The left-hand side equals
\begin{equation}\label{eq:W*term}
\Big|\int_{\Spp}F(x,\w,\sigma)\psi_{\w,\sigma}(D)(\sqrt{m_{1}}g)(x)\ud x\ud\w\frac{\ud\sigma}{\sigma}\Big|,
\end{equation}
and a straightforward calculation shows that
\[
\psi_{\w,\sigma}(D)(\sqrt{m_{1}}g)(x)=(2 \pi)^{-n}\int_{\Rn}e^{ix\cdot\zeta}\psi_{\w,\sigma}(\zeta)L(\wh{g})(\zeta)\ud\zeta.
\]
Now one integrates by parts again. Partial derivatives that hit $e^{ix\cdot\zeta}$ produce weights that are bounded by $\sqrt{m_{1}(x)}$, and partial derivatives of $\psi_{\w,\sigma}$ are dealt with in the same way as before. By using the Cauchy-Schwarz inequality, \eqref{eq:W*term} can then be bounded by a sum of terms of the form
\[
\int_{\Spp}m_{2}(x,w,\sigma)|F(x,\w,\sigma)|^{2}|\wt{W}(g)(x,\w,\sigma)|^{2}\ud x\ud\w\frac{\ud\sigma}{\sigma},
\]
for suitable $\wt{W}:L^{2}(\Rn)\to L^{2}(\Spp)$. This concludes the proof.
\end{proof}

We now give a molecular decomposition of $\HT^{1}_{FIO}(\Rn)$. More precisely, we show that any $f\in \HT^{1}_{FIO}(\Rn)$ can be decomposed into coherent molecules, modulo the low frequencies of $f$. One could additionally decompose the low-frequency component of $f$, by including molecules that do not satisfy a cancellation condition, but we will not do so here. 

\begin{theorem}\label{thm:molecular}
The following assertions hold.
\begin{enumerate}
\item\label{it:mol1} For all $C,\tau_{0}>0$ and $s>\frac{n}{2}$, there exists a constant $C_{1}\geq0$ such that, for each sequence $(f_{k})_{k=1}^{\infty}$ of coherent molecules of type $(s,C)$ associated with balls of radius at most $\sqrt{\tau_{0}}$, and each $(\alpha_{k})_{k=1}^{\infty}\in \ell^{1}$, one has $\sum_{k=1}^{\infty}\alpha_{k}f_{k}\in\HT^{1}_{FIO}(\Rn)$ and 
\[
\Big\|\sum_{k=1}^{\infty}\alpha_{k}f_{k}\Big\|_{\HT^{1}_{FIO}(\Rn)}\leq C_{1}\sum_{k=1}^{\infty}|\alpha_{k}|.
\]
\item\label{it:mol2} Let $q\in C^{\infty}_{c}(\Rn)$ be such that $q(\zeta)=1$ if $|\zeta|\leq 2$, and let $s>\frac{n}{2}$. Then there exist $C_{2},\tau_{0}>0$ such that, for each $f\in\HT^{1}_{FIO}(\Rn)$, there exist a sequence $(f_{k})_{k=1}^{\infty}$ of coherent molecules of type $(s,C_{2})$ associated with balls of radius at most $\sqrt{\tau_{0}}$, and an $(\alpha_{k})_{k=1}^{\infty}\in \ell^{1}$, such that 
\[
f=\sum_{k=1}^{\infty}\alpha_{k}f_{k}+q(D)f
\]
and 
\[
\sum_{k=1}^{\infty}|\alpha_{k}|\leq C_{2}\|f\|_{\HT^{1}_{FIO}(\Rn)}.
\]
\end{enumerate}
\end{theorem}
\begin{proof}
\eqref{it:mol1}: By Proposition \ref{prop:Banach}, $\HT^{1}_{FIO}(\Rn)$ is a Banach space. Hence it suffices to show that the collection of coherent molecules of type $(s,C)$, associated with balls of radius at most $\sqrt{\tau_{0}}$, is a uniformly bounded subset of $\HT^{1}_{FIO}(\Rn)$. Let $f$ be such a coherent molecule, associated with a ball $B_{\sqrt{\tau}}(y,\nu)$, for $(y,\nu)\in\Sp$ and $\tau\in(0,\tau_{0})$. Write 
\[
F(x,\w,\sigma):=\ind_{(0,1)}(\sigma)Wf(x,\w,\sigma)
\]
for $(x,\w,\sigma)\in\Spp$. We separately bound the low-frequency component $\|Wf-F\|_{T^{1}(\Sp)}$ and the high-frequency component $\|F\|_{T^{1}(\Sp)}$ of $Wf$.

For the low-frequency component, note that $r(D)f=0$, and therefore also $Wf-F=0$, unless $\tau\geq \frac{1}{2}$. In the latter case we use Lemma \ref{lem:lowfrequency}, that $r\in C^{\infty}_{c}(\Rn)$, the Cauchy-Schwarz inequality, a substitution and the defining property of $f$ to obtain
\begin{align*}
&\|Wf-F\|_{T^{1}(\Sp)}\lesssim \|r(D)f\|_{L^{1}(\Rn)}\lesssim \|f\|_{L^{1}(\Rn)}\\
&\leq \|m^{\nu,\tau}(\cdot-y)^{-s}\|_{L^{2}(\Rn)}\|f(\cdot)m^{\nu,\tau}(\cdot-y)^{s}\|_{L^{2}(\Rn)}\lesssim \tau^{-n/2}\lesssim 1,
\end{align*}
where $m^{\nu,\tau}$ is as in \eqref{eq:moleculeweight}. 

Next, to bound $\|F\|_{T^{1}(\Sp)}$, set $B:=B_{\sqrt{\tau}}(y,\nu)$, $A_{1}:=\ind_{T(4B)}F$ and $A_{k}:=\ind_{T(2^{k+1}B)\setminus T(2^{k}B)}F$ for $k\geq2$. It suffices to show that there exist $\veps,C'>0$, independent of $f$, such that $\|A_{k}\|_{L^{2}(\Spp)}\leq C'2^{-k\veps}V(2^{k+1}B)^{-1/2}$ for all $k\in\N$. Indeed, then $\frac{1}{C'}2^{k\veps}A_{k}$ is a $T^{1}(\Sp)$-atom associated with $2^{k+1}B$ for each $k\in\N$, and the required statement follows from the fact that 
\[
\|F\|_{T^{1}(\Sp)}=\Big\|\sum_{k=1}^{\infty}A_{k}\Big\|_{T^{1}(\Sp)}
\]
and that the collection of $T^{1}(\Sp)$-atoms is uniformly bounded in $T^{1}(\Sp)$.

Let $c_{1}>0$ be a constant independent of $(y,\nu)$ such that 
\[
d((x,\w),(y,\nu))\leq c_{1}(|\lb\nu,x-y\rb|+|x-y|^{2}+|\w-\nu|^{2})^{1/2}
\]
for all $(x,\w)\in\Sp$, as in Lemma \ref{lem:metric}. Also, let $k_{0}\geq2$ be such that $\delta:=\frac{1}{4}c^{-2}_{1}-25\cdot 2^{-2k_{0}}>0$. For $k<k_{0}$, one has 
\[
\|A_{k}\|_{L^{2}(\Spp)}\leq \|Wf\|_{L^{2}(\Spp)}=\|f\|_{L^{2}(\Rn)}\lesssim \tau^{-n/2}\lesssim V(2^{k+1}B)^{-1/2}
\]
by Proposition \ref{prop:bddwavetransform}, \eqref{eq:decaymolecule} and Lemma \ref{lem:doubling}. 

For $k\geq k_{0}\geq 2$, first note that if $\psi_{\w,\sigma}(D)f\neq 0$ then $\sigma\leq 2\tau$ and $|\w-\nu|\leq \sqrt{\tau}+2\sqrt{\sigma}\leq 5\sqrt{\tau}$. Moreover, if $(x,\w,\sigma)\in T(2^{k+1}B)\setminus T(2^{k}B)$ and $\sigma\leq 2^{2k-2}\tau$, then $(x,\w)\in 2^{k+1}B\setminus 2^{k-1}B$. In this case,
\[
c_{1}^{2}(|\lb\nu,x-y\rb|+|x-y|^{2}+|\w-\nu|^{2})\geq d((x,\w),(y,\nu))^{2}\geq 2^{2k-2}\tau.
\]
Now let $E$ be as in Lemma \ref{lem:weightedL2} with $\kappa=5$. Note that, if $(x,\w,\sigma)\in T(2^{k+1}B)\setminus T(2^{k}B)$ is such that $\psi_{\w,\sigma}(D)f(x)\neq 0$, then $(\w,\sigma)\in E$ and 
\[
m^{\nu,\tau}(x-y)\gtrsim 1+\tau^{-1}(|\lb\nu,x-y\rb|+|x-y|^{2})\geq c_{1}^{-2}2^{2k-2}-25\geq\delta 2^{2k}.
\]
Hence, by Lemma \ref{lem:weightedL2},
\begin{align*}
&\|A_{k}\|_{L^{2}(\Spp)}^{2}\lesssim 2^{-4ks}\int_{\Spp}\ind_{E}(\w,\sigma)m^{\nu,\tau}(x-y)^{2s}|\psi_{\w,\sigma}(D)f(x)|^{2}\ud x\ud\w\frac{\ud\sigma}{\sigma}\\
&\lesssim 2^{-4ks}\int_{\Rn}m^{\nu,\tau}(x-y)^{2s}|f(x)|^{2}\ud x\leq 2^{-4ks}\tau^{-n}\lesssim 2^{-2k(2s-n)}V(2^{k+1}B)^{-1},
\end{align*}
where we used that $f$ is a coherent molecule associated with $B_{\sqrt{\tau}}(y,\nu)$, as well as Lemma \ref{lem:doubling}. This suffices, since $s>\frac{n}{2}$.

\eqref{it:mol2}: First note that $(1-q)(D)f\in\HT^{1}_{FIO}(\Rn)$ with $\|(1-q)(D)f\|_{\HT^{1}_{FIO}(\Rn)}\lesssim \|f\|_{\HT^{1}_{FIO}(\Rn)}$, by Theorem \ref{thm:FIObddHardy}. Hence Lemma \ref{lem:atomictent} yields a sequence $(A_{k})_{k=1}^{\infty}$ of $T^{1}(\Sp)$-atoms and a sequence $(\alpha_{k})_{k=1}^{\infty}\in\ell^{1}$ such that $W((1-q)(D)f)=\sum_{k=1}^{\infty}\alpha_{k}A_{k}$ and
\begin{equation}\label{eq:l1sum}
\begin{aligned}
\sum_{k=1}^{\infty}|\alpha_{k}|&\lesssim \|W(1-q)(D)f\|_{T^{1}(\Sp)}=\|(1-q)(D)f\|_{\HT^{1}_{FIO}(\Rn)}\\
&\lesssim \|f\|_{\HT^{1}_{FIO}(\Rn)},
\end{aligned}
\end{equation}
for implicit constants independent of $f$. Moreover, since $r(1-q)=0$, one has $W((1-q)(D)f)(x,\w,\sigma)=0$ for all $(x,\w,\sigma)\in\Sp$ with $\sigma\geq1$. Hence we may assume that $A_{k}(x,\w,\sigma)=0$ for all $k\in\N$ and $\sigma\geq1$, and that each $A_{k}$ is associated with a ball of radius at most $2$ (see Remark \ref{rem:localatoms}). Because one also has
\[
f=W^{*}Wf=\sum_{k=1}^{\infty}\alpha_{k}W^{*}A_{k}+q(D)f,
\]
it suffices to show that each $W^{*}A_{k}$ is a coherent molecule of type $(s,C'')$ associated with a ball of radius at most $\sqrt{\tau_{0}}$, for some fixed $C'',\tau_{0}>0$. 

Let $A:=A_{k}$, for some $k\in\N$, be such a $T^{1}(\Sp)$-atom, associated with a ball $B_{\sqrt{\tau}}(y,\nu)$ for $(y,\nu)\in\Sp$ and $\tau\leq 4$. First note that $f:=W^{*}A\in L^{2}(\Rn)$, by Proposition \ref{prop:bddwavetransform}. Next, recall that $A(x,\w,\sigma)=0$ if $(x,\w,\sigma)\in\Spp$ satisfies $d((x,\w),B_{\sqrt{\tau}}(y,\nu)^{c})<\sqrt{\sigma}$, i.e.~if $d((x,\w),(y,\nu))>\sqrt{\tau}-\sqrt{\sigma}$. It follows that
\begin{equation}\label{eq:supportA}
A(x,\w,\sigma)=0 \text{ if }\sigma>\min(\tau,1)\text{ or }|\lb\nu,x-y\rb|+|x-y|^{2}+|\w-\nu|^{2}>c_{2}^{-2}\tau,
\end{equation}
where $c_{2}\in(0,1)$ is a constant independent of $(x,\w)$ and $(y,\nu)$ such that 
\[
d((x,\w),(y,\nu))\geq c_{2}(|\lb\nu,x-y\rb|+|x-y|^{2}+|\w-\nu|^{2})^{1/2},
\]
as in Lemma \ref{lem:metric}. Now Lemma \ref{lem:packetbounds} shows that $\psi_{\w,\sigma}(\zeta) \F(A(\cdot,\w,\sigma))(\zeta)=0$ for all $(\zeta,\w,\sigma)\in\Spp$ such that 
\[
|\zeta|< \tfrac{1}{2}\tau^{-1}\text{ or }|\hat{\zeta}-\nu|>(2+c_{2}^{-1})\sqrt{\tau}.
\]
For such $\zeta$ one has
\[
\wh{f}(\zeta)=\int_{0}^{1}\int_{S^{n-1}}\psi_{\w,\sigma}(\zeta)\F(A(\cdot,\w,\sigma))(\zeta)\ud\w\frac{\ud\sigma}{\sigma}=0.
\]
Hence, if we show that 
\begin{equation}\label{eq:molcondition2}
\int_{\Rn}(1+\tau^{-1}|\lb\nu,x-y\rb|)^{2s}(1+\tau^{-1}|x-y|^{2}))^{2s}|f(x)|^{2}\ud x\lesssim\tau^{-n}
\end{equation}
for an implicit constant independent of $f$, then it follows that $f$ is a coherent molecule of type $(s,C'')$ associated with $B_{(2+1/c_{2})\sqrt{\tau}}(y,\nu)$ for some $C''\geq0$, as required.

To prove \eqref{eq:molcondition2}, let $E$ be as in Lemma \ref{lem:weightedL2} with $\kappa:=c_{2}^{-1}\geq1$. By \eqref{eq:supportA}, one has $\ind_{E}A=A$ and
\[
(1+\tau^{-1}|\lb\nu,x-y\rb|)(1+\tau^{-1}|x-y|^{2})\leq (1+\kappa^{2})^{2}
\]
for all $(x,\w,\sigma)\in\Spp$ with $A(x,\w,\sigma)\neq0$. It now follows from Lemma \ref{lem:weightedL2}, with notation as in that lemma, that
\begin{align*}
&\int_{\Rn}(1+\tau^{-1}|\lb\nu,x-y\rb|)^{2s}(1+\tau^{-1}|x-y|^{2}))^{2s}|f(x)|^{2}\ud x\eqsim\|V_{E}A\|_{L^{2}(\Rn,m_{1})}^{2}\\
&\lesssim \|A\|_{L^{2}(\Spp,m_{2})}^{2}\lesssim \|A\|_{L^{2}(\Spp)}^{2}\leq V(B_{\sqrt{\tau}}(y,\nu))^{-1}\lesssim \tau^{-n}.\qedhere
\end{align*}
\end{proof}

\begin{remark}\label{rem:Hartspace}
Note from the proof of part \eqref{it:mol2} of Theorem \ref{thm:molecular}, and in particular \eqref{eq:l1sum}, that one in fact has the stronger estimate
\[
\sum_{k=1}^{\infty}|\alpha_{k}|\leq \int_{\Sp}\Big(\int_{0}^{1}\fint_{B_{\sqrt{\sigma}}(x,\w)}|\psi_{\nu,\sigma}(D)f(y)|^{2}\ud y\ud\nu\frac{\ud\sigma}{\sigma}\Big)^{1/2}\ud x\ud\w,
\]
corresponding to a bound involving only the high-frequency part of $f$. Also note that the sequences $(f_{k})_{k=1}^{\infty}$ and $(\alpha_{k})_{k=1}^{\infty}$ that we obtained in the proof of part \eqref{it:mol2} are independent of the choice of $s>\frac{n}{2}$. 

It is now easy to show, by combining Theorems \ref{thm:molecular} and \ref{thm:Sobolev}, Remark \ref{rem:molecule} and \cite[Lemma 3.4 and Theorem 3.6]{Smith98a}, that $\lb D\rb^{-\frac{n-1}{4}}(\HT^{1}_{FIO}(\Rn))$ coincides with the Hardy space for Fourier integral operators from \cite{Smith98a}. Of course, one has a choice of smoothness in defining $\Hp$, and the choice in \cite{Smith98a} for $p=1$ yields a subspace of $L^{1}(\Rn)$, cf.~Theorem \ref{thm:Sobolev}. The choice in the present article is motivated by the fact that it allows us to study all $\Hp$ simultaneously using a single wave packet transform $W$, while still ensuring that $\HT^{2}_{FIO}(\Rn)=L^{2}(\Rn)$. 
\end{remark}

\section*{Acknowledgements}

Part of this research was conducted during a visit by the third author to the University of Washington in Seattle. He would like to thank Hart Smith for his hospitality and for various helpful comments. The authors also wish to thank the anonymous referee for  suggestions that have improved the manuscript, and Yi Huang for a useful comment concerning the list of references.

\appendix

\section{Low-frequency component}

In this appendix we prove a lemma to deal with the low-frequency part of the $\HT^{1}_{FIO}(\Rn)$-norm.

\begin{lemma}\label{lem:lowfrequency}
Let $q\in C^{\infty}_{c}(\Rn)$. Then there exists a constant $C=C(n,q)>0$ such that the following holds. Let $f\in\Sw'(\Rn)$ be such that $q(D)f\in L^{1}(\Rn)$. For $(x,\w,\sigma)\in\Spp$, set 
\[
\wt{W}f(x,\w,\sigma):=\ind_{[1,e]}(\sigma)q(D)f(x).
\]
Then $\wt{W}f\in T^{1}(\Sp)$ and
\begin{equation}\label{eq:tentL1}
\frac{1}{C}\|\wt{W}f\|_{T^{1}(\Sp)}\leq \|q(D)f\|_{L^{1}(\Rn)}\leq C\|\wt{W}f\|_{T^{1}(\Sp)}.
\end{equation}
\end{lemma}
\begin{proof}
First note that, for all $R>0$ and $g\in L^{1}(\Rn)$, Fubini's theorem yields 
\begin{equation}\label{eq:doubleintegral}
\int_{\Rn}\int_{B_{R}(x)}|g(y)|\ud y\ud x= V(B_{R}(0))\|g\|_{L^{1}(\Rn)}.
\end{equation}
Using this for sufficiently small $R$, we estimate 
\begin{align*}
\|q(D)f\|_{L^{1}(\Rn)} &\lesssim \int_{\Rn}\int_{B_{R}(x)}|q(D)f(y)|\ud y\ud x\lesssim \int_{\Rn}\Big(\int_{B_{R}(x)}|q(D)f(y)|^{2}\ud y\Big)^{1/2}\ud x  \\
&\lesssim  \int_{\Rn}\int_{S^{n-1}}\Big(\int_{1}^{e}\int_{B_{R}(\w)}\int_{B_{R}(x)}|q(D)f(y)|^{2}\ud y\ud\nu\ud\sigma\Big)^{1/2}\ud \w\ud x\\
& \leq\|\wt{W}f\|_{T^{1}(\Sp)}.
\end{align*}
Here we used the fact that $B_R(x) \times B_R(\omega) $ is contained in $B_{\sqrt{\sigma}}(x, \omega)$ for sufficiently small $R$ and for $\sigma \geq 1$, as well as the fact that powers of $\sigma$ are uniformly bounded for $\sigma \in [1, e]$. This proves the right-hand inequality in \eqref{eq:tentL1}.

The proof of the left-hand inequality in \eqref{eq:tentL1} is similar, except that we use a Sobolev embedding. Let $R>0$ be large enough that $B_{\sqrt{\sigma}}(x, \omega) \subset B_{R/2}(x) \times S^{n-1}$ for all $\sigma\leq e$. Let $\rho\in C_{c}^{\infty}(\Rn)$ be such that $\rho\equiv 1$ on $B_{R/2}(0)$, and $\supp(\rho)\subseteq B_{R}(0)$. Set $\rho_{x}(y):=\rho(x-y)$ for $x,y\in\Rn$ and let $m\in\Z_{+}$ be such that $m\geq n/2$. Then
\begin{align*}
\|\wt{W}f\|_{T^{1}(\Sp)}&\leq \int_{\Rn}\int_{S^{n-1}}\Big(\int_{1}^{e}\int_{S^{n-1}}\int_{B_{R/2}(x)}|q(D)f(y)|^{2}\ud y\ud\nu\ud\sigma\Big)^{1/2}\ud\w\ud x\\
&\lesssim \int_{\Rn}\Big(\int_{B_{R/2}(x)}|q(D)f(y)|^{2}\ud y\Big)^{1/2}\ud x\leq \!\int_{\Rn}\|\rho_{x}\cdot q(D)f\|_{L^{2}(\Rn)}\ud x .
\end{align*}
Now a Sobolev embedding and \eqref{eq:doubleintegral} yield, using the compact support of $q$ in the last step, 
\begin{align*}
&\int_{\Rn}\|\rho_{x}\cdot q(D)f\|_{L^{2}(\Rn)}\ud x\lesssim \int_{\Rn}\|\rho_{x}\cdot q(D)f\|_{W^{m,1}(\Rn)}\ud x\\
&=\sum_{|\alpha|\leq m}\int_{\Rn}\int_{\Rn}|\partial^{\alpha}_{y}(\rho_{x}\cdot q(D)f)|(y)\ud y\ud x\\
&\leq \sum_{|\alpha|\leq m}\sum_{\beta\leq \alpha}\int_{\Rn}\int_{\Rn}|\partial_{y}^{\alpha-\beta}\rho_{x}(y)\partial_{y}^{\beta}(q(D)f)(y)|\ud y\ud x\\
&\lesssim \sum_{|\beta|\leq m}\int_{\Rn}\int_{B_{R}(x)}|\partial_{y}^{\beta}(q(D)f)(y)|\ud y\ud x\eqsim \sum_{|\beta|\leq m}\|\partial^{\beta}(q(D)f)\|_{L^{1}(\Rn)}\\
&=\|q(D)f\|_{W^{m,1}(\Rn)}\lesssim \|q(D)f\|_{L^{1}(\Rn)}
\end{align*}
for implicit constants which depend only on $n$ and $q$. 
\end{proof}

\bibliographystyle{amsplain}
\bibliography{Bibliography}

\end{document}